\documentclass[11pt]{elsarticle}
\journal{Journal of Computational Physics}
\usepackage[margin=2.5cm]{geometry}% by courtesy of Mico
\biboptions{sort&compress}
\makeatletter
\providecommand{\doi}[1]{%
	\begingroup
	\let\bibinfo\@secondoftwo
	\urlstyle{rm}%
	\href{http://dx.doi.org/#1}{%
		doi:\discretionary{}{}{}%
		\nolinkurl{#1}%
	}%
	\endgroup
}
\makeatother

\usepackage{graphicx}% Include figure files
\usepackage{dcolumn}% Align table columns on decimal point
\usepackage{bm}% bold math
\usepackage{hyperref}% add hypertext capabilities
\usepackage{epstopdf}
\usepackage{amsmath,esint}
\usepackage{amsmath}
\usepackage{amsfonts}
\usepackage{amssymb}
\usepackage{graphicx}
\usepackage{amsmath}
\usepackage{amsthm}
\usepackage{eucal}
\usepackage{amssymb}
\usepackage{mathrsfs}
\usepackage{float}
\usepackage{hyperref}
\usepackage{multirow}
\usepackage{booktabs}
\usepackage{graphicx}
\usepackage{subfigure}
\usepackage{dashrule}
\usepackage{adjustbox}
\usepackage{xcolor}
\usepackage{graphicx}% Include figure files
\usepackage{dcolumn}% Align table columns on decimal point
\usepackage{bm}
\usepackage{natbib}
\usepackage{amssymb}
\usepackage{amsthm}
\usepackage{mathtools}
\usepackage{empheq}
\usepackage[most]{tcolorbox}
\usepackage{mathtools}
\usepackage{cleveref}
\usepackage{soul}
\usepackage{array}

\usepackage{tikz}
\usetikzlibrary{shapes,arrows,chains}
\usetikzlibrary{calc}
\usetikzlibrary{positioning}

\usepackage{tgpagella}
\usepackage{newpxmath}

\renewcommand{\vec}[1]{\mathbf{#1}}

\usepackage{xspace}
\let\storeBeta=\beta
\renewcommand\beta{\relax\ifmmode{\storeBeta}\else{$\storeBeta$}\fi\xspace}

\let\storeAlpha=\alpha
\renewcommand\alpha{\relax\ifmmode{\storeAlpha}\else{$\storeAlpha$}\fi\xspace}

\newcommand{\abs}[1]{\left| #1 \right|} % for absolute value
\renewcommand{\d}[2]{\frac{\mathrm{d} #1}{\mathrm{d} #2}} % for derivatives
\newcommand{\pd}[2]{\frac{\partial #1}{\partial #2}} 
\let\baraccent=\= % rename builtin command \= to \baraccent
\renewcommand{\=}[1]{\stackrel{#1}{=}} % for putting numbers above =

\newcount\colveccount
\newcommand*\colvec[1]{
	\global\colveccount#1
	\begin{pmatrix}
		\colvecnext
	}
	\def\colvecnext#1{
		#1
		\global\advance\colveccount-1
		\ifnum\colveccount>0
		\\
		\expandafter\colvecnext
		\else
	\end{pmatrix}
	\fi
}

\newcommand{\norm}[1]{\left\lVert#1\right\rVert}

\def\tphi{\widetilde{\phi}}
\def\tmu{\widetilde{\mu}}
\def\trhok{\widetilde{\rho}^{\, k}}
\def\tetak{\widetilde{\eta}^{\, k}}
\def\rhokplusOne{\rho^{\, {k+1}}}
\def\rhok{\rho^{\, k}}

\def\tp{\widetilde{p}}
\def\tvi{\widetilde{v}_i}
\def\tvj{\widetilde{v}_j}
\def\tui{\widetilde{u}_i}

\def\tJi{\widetilde{J}_i}
\def\tJj{\widetilde{J}_j}
\def\tpsi{\widetilde{\psi}}

\def\hui{\widehat{u}_i}
\def\huj{\widehat{u}_j}

\def\hJj{\widehat{J}_j}

\def\dt{\delta t}

\hypersetup{
	colorlinks,
	linkcolor={blue!50!black},
	citecolor={blue!50!black},
	urlcolor={blue!80!black},
	anchorcolor = {blue!80!black},
	filecolor = {blue!80!black},
	menucolor = {blue!80!black},
	runcolor = {blue!80!black}
}

\newtheorem{lemma}{Lemma}
\newtheorem{proposition}{Proposition}

\newtheorem{definition}{Definition}
\newtheorem{remark}{Remark}

\usepackage{pgfplots}
\pgfplotsset{
	table/search path={Figures},
}
\usepgfplotslibrary{units} % Allows to enter the units nicely
\usepgfplotslibrary{colorbrewer}
\usetikzlibrary{arrows,shapes}
\pgfplotsset{compat=1.8}
\pgfplotsset{
	discard if/.style 2 args={
		x filter/.code={
			\edef\tempa{\thisrow{#1}}
			\edef\tempb{#2}
			\ifx\tempa\tempb
			\def\pgfmathresult{inf}
			\fi
		}
	},
	discard if not/.style 2 args={
		x filter/.code={
			\edef\tempa{\thisrow{#1}}
			\edef\tempb{#2}
			\ifx\tempa\tempb
			\else
			\def\pgfmathresult{inf}
			\fi
		}
	}
}
\usetikzlibrary{plotmarks}
\usetikzlibrary{spy}
\usetikzlibrary{arrows,shapes,plotmarks}
\usepgfplotslibrary{groupplots}
\usetikzlibrary{snakes}

\usetikzlibrary{calc}

\usetikzlibrary{matrix}

\usepackage{xargs}                      % Use more than one optional parameter in a new commands
\usepackage{siunitx}

\usepackage[colorinlistoftodos,prependcaption,textsize=tiny]{todonotes}

\newcounter{bgunsure}
\newcounter{bgchange}
\newcounter{bginfo}
\newcounter{bgimprovement}
\newcounter{bgthiswillnotshow}
\newcommandx{\bgunsure}[2][1=]{\refstepcounter{bgunsure}\todo[linecolor=red,backgroundcolor=red!25,bordercolor=red,#1]{[BG\thebgunsure:] #2}}
\newcommandx{\bgchange}[2][1=]{\refstepcounter{bgchange}\todo[linecolor=blue,backgroundcolor=blue!25,bordercolor=blue,#1]{[BG\thebgchange:] #2}}
\newcommandx{\bginfo}[2][1=]{\refstepcounter{bgimprovement}\todo[linecolor=OliveGreen,backgroundcolor=OliveGreen!25,bordercolor=OliveGreen,#1]{[BG\thebgimprovement:] #2}}
\newcommandx{\bgimprovement}[2][1=]{\refstepcounter{bginfo}\todo[linecolor=Plum,backgroundcolor=Plum!25,bordercolor=Plum,#1]{[BG\thebginfo:] #2}}
\newcommandx{\bgthiswillnotshow}[2][1=]{\refstepcounter{bgthiswillnotshow}\todo[disable,#1]{[BG\thebgthiswillnotshow:] #2}}

\newcommand{\Frontera}{\href{https://www.tacc.utexas.edu/systems/frontera}{Frontera}}

\newcounter{mkunsure}
\newcounter{mkchange}
\newcounter{mkinfo}
\newcounter{mkimprovement}
\newcounter{mkthiswillnotshow}
\newcommandx{\mkunsure}[2][1=]{\refstepcounter{mkunsure}\todo[linecolor=red,backgroundcolor=red!25,bordercolor=red,#1]{[MAK\themkunsure:] #2}}
\newcommandx{\mkchange}[2][1=]{\refstepcounter{mkchange}\todo[linecolor=blue,backgroundcolor=blue!25,bordercolor=blue,#1]{[MAK\themkchange:] #2}}
\newcommandx{\mkinfo}[2][1=]{\refstepcounter{mkimprovement}\todo[linecolor=OliveGreen,backgroundcolor=OliveGreen!25,bordercolor=OliveGreen,#1]{[MAK\themkimprovement:] #2}}
\newcommandx{\mkimprovement}[2][1=]{\refstepcounter{mkinfo}\todo[linecolor=Plum,backgroundcolor=Plum!25,bordercolor=Plum,#1]{[MAK\themkinfo:] #2}}
\newcommandx{\mkthiswillnotshow}[2][1=]{\refstepcounter{mkthiswillnotshow}\todo[disable,#1]{[MAK\themkthiswillnotshow:] #2}}

\newcommand{\petsc}{\textsc{PETSc}}

\newcommand{\dendroFive}{\textsc{Dendro5}}
\newcommand{\dendroKT}{\textsc{Dendro-KT}}

\usepackage{listings}

\usepackage{xparse}

\NewDocumentCommand{\codeword}{v}{%
	\texttt{\textcolor{blue}{#1}}%
}

\definecolor{codegreen}{rgb}{0,0.6,0}
\definecolor{codegray}{rgb}{0.5,0.5,0.5}
\definecolor{codepurple}{rgb}{0.58,0,0.82}
\definecolor{backcolour}{rgb}{0.95,0.95,0.92}

\lstdefinestyle{mystyle}{
	backgroundcolor=\color{backcolour},   
	commentstyle=\color{codegreen},
	keywordstyle=\color{magenta},
	stringstyle=\color{codepurple},
	basicstyle=\ttfamily\footnotesize,
	breakatwhitespace=false,         
	breaklines=true,                 
	captionpos=b,                    
	keepspaces=true,                  
	showspaces=false,                
	showstringspaces=false,
	showtabs=false,                  
	tabsize=2
}

\definecolor{sq_b1}{RGB}{37,52,148}
\definecolor{sq_b2}{RGB}{44,127,184}
\definecolor{sq_b3}{RGB}{65,182,196}
\definecolor{sq_b4}{RGB}{127,205,187}
\definecolor{sq_b5}{RGB}{199,233,180}
\definecolor{sq_b6}{RGB}{255,255,204}

\definecolor{sq_r1}{RGB}{189,0,38}
\definecolor{sq_r2}{RGB}{240,59,32}
\definecolor{sq_r3}{RGB}{253,141,60}
\definecolor{sq_r4}{RGB}{254,178,76}
\definecolor{sq_r5}{RGB}{254,217,118}
\definecolor{sq_r6}{RGB}{255,255,178}

\definecolor{sq_g1}{RGB}{0,104,55}
\definecolor{sq_g2}{RGB}{49,163,84}
\definecolor{sq_g3}{RGB}{120,198,121}
\definecolor{sq_g4}{RGB}{173,221,142}
\definecolor{sq_g5}{RGB}{217,240,163}
\definecolor{sq_g6}{RGB}{255,255,204}

\definecolor{div_c1}{RGB}{230,171,2}
\definecolor{div_c2}{RGB}{102,166,30}
\definecolor{div_c3}{RGB}{231,41,138}
\definecolor{div_c4}{RGB}{117,112,179}
\definecolor{div_c5}{RGB}{217,95,2}
\definecolor{div_c6}{RGB}{27,158,119}
\definecolor{div_c7}{RGB}{215,48,39}

\definecolor{div_d1}{RGB}{215,25,28}
\definecolor{div_d2}{RGB}{253,174,97}
\definecolor{div_d3}{RGB}{255,255,191}
\definecolor{div_d4}{RGB}{171,217,233}
\definecolor{div_d5}{RGB}{44,123,182}

\allowdisplaybreaks
\usetikzlibrary{external}
\begin{document}

\begin{frontmatter}

%\title{Projection based semi-implicit method for solving the Cahn-Hilliard Navier-Stokes system using a conforming Galerkin finite element method on adaptive octree meshes}
%\title{A projection based continuous Galerkin finite element method with semi-implicit time-stepping for the two-phase Cahn-Hilliard Navier-Stokes equations on solution-adaptive octree meshes}
\title{A projection-based,  semi-implicit time-stepping approach for the Cahn-Hilliard Navier-Stokes equations on adaptive octree meshes}

%\tnotetext[mytitlenote]{text}

\author[isuMechEAddress,isuMathAddress]{Makrand A. Khanwale\fnref{MakFootnote}}
\ead{khanwale@iastate.edu}

\author[isuMechEAddress]{Kumar Saurabh}
\ead{maksbh@iastate.edu}

\author[utahAddress]{Masado Ishii}
\ead{masado@cs.utah.edu}

\author[utahAddress]{Hari Sundar}
\ead{hari@cs.utah.edu}

\author[isuMathAddress]{James A. Rossmanith}
\ead{rossmani@iastate.edu}

\author[isuMechEAddress]{Baskar~Ganapathysubramanian\corref{correspondingAuthor}}
\ead{baskarg@iastate.edu}

\cortext[correspondingAuthor]{Corresponding author}

\address[isuMechEAddress]{Department of Mechanical Engineering, Iowa State University, Iowa, USA 50011}
\address[isuMathAddress]{Department of Mathematics, Iowa State University, Iowa, USA 50011}
\address[utahAddress]{School of Computing, The University of Utah, Salt Lake City, Utah, USA 84112}

\fntext[MakFootnote]{Currently at Center for Turbulence Research, Department of Mechanical Engineering, Stanford University, CA, USA}

\begin{abstract}
The Cahn-Hilliard Navier-Stokes (CHNS) system provides a computationally
tractable model that can be used to effectively capture interfacial dynamics in 
two-phase fluid flows. In this work we present a semi-implicit, projection-based finite element framework for solving the CHNS system.
	%In this work we extend the fully implicit method presented in Khanwale et al. [{\it A fully-coupled framework for solving Cahn-Hilliard Navier-Stokes equations: Second-order, energy-stable numerical methods on adaptive octree based meshes.},  arXiv:2009.06628 (2020)], to a block iterative hybrid method. 
	We use a projection-based semi-implicit time discretization for the Navier-Stokes equation and a fully-implicit time discretization for the Cahn-Hilliard equation. We use a conforming continuous Galerkin (cG) finite element method in space equipped with a residual-based variational multiscale (RBVMS) formulation.  Pressure is decoupled using a projection step, which results in two linear positive semi-definite systems for velocity and pressure, instead of the saddle point system of a pressure-stabilized method.  All the linear systems are solved using an efficient and scalable algebraic multigrid (AMG) method. We deploy this approach on a massively parallel numerical implementation using parallel octree-based adaptive meshes. The overall approach allows the use of relatively large time steps with much faster time-to-solve than similar fully-implicit methods.
	We present comprehensive numerical experiments showing detailed comparisons with results from the literature for canonical cases, including the single bubble rise and Rayleigh-Taylor instability. 
\end{abstract}
 

\end{frontmatter}

%\linenumbers

%!TEX root = main.tex
\section{Introduction}
\label{sec:introduction}
Capturing the interfacial dynamics is crucial for gaining fundamental understanding of two-phase flows. In several applications like bio-microfluidics and advanced manufacturing such interface resolved simulations are critical for analysis and design~\citep{Gibou2019, Xia2019,Zhang2021}. More generally, the availability of robust and fast interface-resolved two-phase simulations can greatly enable the development of (data-driven) coarse-scale models that need not be interface-resolving. This is the primary motivation of the current work~\citep{Ganti2020, Brunton2020, Duraisamy2021,Karniadakis2021}.

We are particularly interested in the coupled Cahn-Hilliard Navier-Stokes equations as a means to capture interfacial dynamics. Cahn-Hilliard Navier-Stokes (CHNS) models have received increasing attention for their ability to capture interface resolved two-phase phenomena in various applications~\citep{Guo2015,Xie2015,Jain2020,Huang2021,Mirjalili2021}. Such an approach provides a thermodynamically consistent description of the interfacial processes and evolution~\citep{ Anderson1998}. CHNS belongs to a family of models where a diffused field (also known as phase field) is used to track the interface. 
%Level set methods also belong to this family~\citep{Osher2006}.  
Such an approach circumvents modeling the jump discontinuities at the interface and allows for a diffuse transition between the physical properties from one phase to the other.  Using Cahn-Hilliard equations to track the interface produces several advantages including mass conservation, thermodynamic consistency, and a free-energy-based description of surface tension with a well-established theory from non-equilibrium thermodynamics \citep{ Jacqmin1996, Jacqmin2000}.  However, numerical schemes need to be carefully designed to allow the discrete numerical solutions of these CHNS models to inherit these continuous properties.
In particular, CHNS models that use a volume averaged mixture velocity are attractive as they ensure a divergence-free mixture velocity, which, crucially, allows numerical schemes for such CHNS models to be designed by appealing to the extensive literature on incompressible single phase flows (see~\citet{Volker2016}). Therefore, in this work, we continue to use a volume averaged mixture velocity for numerical development\footnote{CHNS models with mass averaged mixture velocity have also been explored in the literature~\citep{Guo2017,Shokrpour2018}.}.     	
	
We previously developed an energy-stable, fully-coupled, second order scheme in~\citet{Khanwale2021}. This scheme uses an equal order interpolation for pressure and velocity and uses a pressure stabilization based on the variational multi-scale method. While the approach of \cite{Khanwale2021} is well-suited for systems affording naturally large timesteps (or where steady state solutions are desired), the coupled non-linear solves become expensive for turbulent systems, especially when small time steps are naturally required. Moreover, the analysis in \cite{Khanwale2021} suggests that a fully-coupled pressure stabilized approach requires very careful design of preconditioners for efficient execution of large scale production simulations.  To address these issues %in this work we develop a second order projection method based on the provably energy stable scheme and mass conserving scheme in~\citet{Khanwale2021}, which maintains these critical properties\footnote{We verify these properties numerically}.  
we seek to improve the results of~\citep{ Khanwale2021} in three key aspects. 
\begin{enumerate}
\item \textbf{Projection scheme for decoupled pressure solve:} We modify the fully coupled-scheme presented in~\citet{Khanwale2021}
by decoupling the pressure equation from the velocity equations using a projection method; this is done in a manner that preserves second-order accuracy, energy-stability, and mass conservation. In conjunction with a block solution scheme, this results in elliptic (positive semi-definite) operators.  Additionally, a projection variant of VMS method to bypass the discrete inf-sup condition is presented.  %VMS based stabilization allows us to circumvent discrete inf-sup condition which is can cause pressure instabilities for conforming Galerkin elements with the projection method.  
%		
%\item \textbf{VMS and conforming Galerkin elements:} Since we use same order basis functions to approximate pressure and velocity, we rely on a variational multi-scale (VMS) based approach to stabilize the pressure Poisson equation. This allows us to circumvent the discrete inf-sup condition in equal order polynomial representations for velocity and pressure (e.g.,~\citet{ Volker2016}). The VMS approach also provides a variational form of a Large eddy simulation-type filtering, which is useful for turbulent multi-phase flows\footnote{We simulate the Rayleigh-Taylor instability to demonstrate this point.}.  
%		
\item \textbf{Improved octree-based adaptive mesh:} We improve upon the octree based meshing framework presented in ~\citet{ Khanwale2021} through the use of a mesh-free $kD$ tree construction. This new approach
yields improved parallel performance~\citep{ishii2019solving, saurabh2021scalable} with a smaller memory footprint, and provides support for simulating on non-cubic domains via incomplete octrees.

\item \textbf{Scaling and speedup comparison:} We present a detailed scaling and performance analysis of the projection method on up to 7000 processors and compare it to the fully-coupled method presented in~\citep{ Khanwale2021} to assess relative speedup.  
\end{enumerate}
\subsection{Second-order energy-stable scheme projection scheme}
%There are two main approaches for CHNS modelling depending on the averaging used to define the mixture velocity.
%\textcolor{cyan}{
	We deploy our numerical method on adaptive meshes (see~\cref{subsec:octree_intro}), which can result in very restrictive time-steps for explicit schemes (like in~\citet{Badalassi2003}). This is in part due to the viscous and surface tension terms in Navier-Stokes, even if the Cahn-Hilliard equations are discretized using implicit or semi-implicit linearized methods~\citep{Kim2004,Shen2010a,Shen2018}. Unlike other related works \citet{Badalassi2003, Kim2004, Yue2004}, we develop a numerical method for a thermodynamically consistent CH-NS model~\citep{Abels2012}, which allows for unequal densities. This results in a variable coefficient pressure Poisson system  that is carefully constructed such that we can use state-of-the-art Algebraic Multigrid solvers from \texttt{PETSc}.
%}

There have been a number of attempts to use projection-based approaches to model CHNS models~\citep{Shen2010a, Shen2010b, Dong2012, Shen2015, Chen2016, Dong2016, Guo2017, Zhu2019}.  \citet{Shen2010a, Shen2010b, Shen2015}~used a block-solver strategy with linearized time-schemes that reduce their discretization to a sequence of elliptic equations for the velocity and phase fields; however, in these papers %\citep{Shen2010a, Shen2010b, Shen2015}~
the numerical methods were limited to first order provably energy-stable schemes. An upside of the first order nature of these schemes is the need to solve each block only once.  However, the velocity used to advect the phase field in~\citep{Shen2010a, Shen2010b, Shen2015} is not guaranteed to be solenoidal, and needs to be modified to maintain energy-stability. A simplifying assumption made in~\citep{Shen2010a, Shen2010b, Shen2015} was to use a constant coefficient pressure Poisson equation\footnote{The Poisson equation was made constant coefficient by replacing the variable specific density by the minimum of the specific densities of the two fluids.}. This is an elegant simplification, particularly for cases where the system behavior is not very sensitive to the density ratio\footnote{As an example, the dynamics in the case of a rising bubble does not change once the density ratio is increased beyond a certain asymptotic value.}. In this work we relax this assumption -- with an eye on applications where the density ratio critically determines dynamics%%%~\cite{} 
-- and use the mixture density resulting in a variable coefficient pressure Poisson equation. We use a well-tuned Algebraic Multigrid (AMG) method to efficiently solve this system.

\citet{Zhu2019} extended stabilized linearization of Cahn-Hilliard presented in~\citep{Shen2010a, Shen2010b, Shen2015} to a Scalar Auxiliary Variable (SAV) linearized Cahn-Hilliard time-scheme.  The stabilized and SAV formulations for developing linear semi-implicit time schemes for the Cahn-Hilliard part of the CHNS models are very effective; however, they introduce tight time-step requirements. Instead of trying to linearize the Cahn-Hilliard equations, we use a fully implicit non-linear time-scheme similar to~\citet{Han2015} that allows us to take larger timesteps. This choice is a trade-off between time-step restrictions and complexity of each time step -- both of which impact the total time-to-solve. Interestingly, we find that the fully-implicit non-linear scheme for the Cahn-Hilliard equations reduces to a one step Newton iteration for smaller timesteps, making it computationally equivalent to linear formulations.  

\textit{\textbf{Other related work}}: \citet{Han2015}~used a block-iterative strategy with an energy-stable time scheme but with a non-linear scheme for Cahn-Hilliard.  However, the time scheme in \citet{Han2015} is limited to the equal density case; and solenoidality of the velocity used for the advective phase field is not guaranteed. \citet{ Guo2017}~recently reported a detailed analysis for a mass-averaged mixture velocity CHNS system using a fully coupled strategy using both projection and pressure coupled methods. While we do not include a rigorous proof of energy stability (we defer this to a subsequent publication), we numerically confirm energy-stability for a variety of canonical benchmarks\footnote{We note that the current method is constructed by decoupling pressure and velocity in the provably energy-stable method of~\citet{Khanwale2021}.}.  
\subsection{VMS-based stabilization for conforming Galerkin elements}
Generally projection methods enforce solenoidality of the mixture velocity through a Helmholtz-Hodge decomposition step\citep{Guermond2006}. When we take the divergence of the Helmholtz-Hodge decomposition step we obtain the pressure Poisson equation used to update the pressure. When using continuous Galerkin finite elements and equal order polynomial basis functions for velocity and pressure, the discrete inf-sup condition~\citep{Volker2016} is not satisfied, which can lead to checkerboard instabilities for pressure when solving the pressure Poisson step.  Therefore, a pressure stabilization approach is necessary even when using the projection approach.  However, continuous Galerkin methods are attractive as they are comparatively easy to implement with adaptive octree based meshes~\citep{ishii2019solving, Xu2020}. To circumvent the inf-sup condition and stabilize the pressure Poisson equation we use a residual based VMS~\citep{ Hughes2000, Bazilevs2007, Ahmed2017} approach. Moreover, the VMS approach allows for LES-type modeling of the flow physics as the mesh is coarsened away from the interface.    
\subsection{Scalable octree-based adaptive mesh generation} 
\label{subsec:octree_intro}
%\textcolor{red}{Masado's part}
Adaptive spatial discretizations are popular in computational science~\citep{ Coupez2013, Hachem2013, Hachem2016} to improve efficiency and resolution quality. In several large-scale applications~\citep{ ishii2019solving, mantle, Fernando2018_GR}, the feasibility of high-fidelity simulations on modern supercomputers is primarily due to the localized resolution of adaptive discretizations. There are computational challenges to achieve spatial adaptivity in the distributed-memory setting, such as load-balancing, low-cost mesh generation, and mesh-partitioning. The scalability of the algorithms used for mesh generation and partitioning is crucial, especially when the problem requires frequent re-meshing. 
This is only possible when the overhead of solving problems on an adaptive mesh are significantly lower that the savings achieved from the reduction of problem-size.
Octrees are widely used in the community~\citep{ Sundar2008, BursteddeWilcoxGhattas11, Fernando2018_GR, ishii2019solving} due to their simplicity and their extreme parallel scalability.

In this work we employ an updated version of a parallel octree library, called \dendroKT, that provides high-quality space-adaptive resolution along with efficient mesh generation, partitioning, and traversal methods.
The \dendroKT~framework is a freely available open-source library. The \dendroKT~library supports 2D, 3D, and 4D trees, and has been used to manage spacetime-adaptive discretizations for solving time-dependent PDEs~\citep{ishii2019solving}. A distinguishing feature of \dendroKT~is the elimination of the element-to-nodal maps, instead giving access to nodal data at the finite elements using top-down and bottom-up traversals. The benefits of this are minimizing data footprint, avoiding indirect memory accesses, and avoiding computing complex topological maps for complex domains. Additionally, \dendroKT~efficiently manages octrees over non-isotropic carved-out domains~\citep{saurabh2021scalable}. The methods of mesh-generation, partitioning, and traversal are detailed in~\cref{sec:octree_mesh}.

\section{Governing equations}
\label{sec:governing_equations}

Consider a bounded domain $\Omega \subset \mathbb{R}^d$, for $d = 2,3$ containing two immiscible fluids, and the time interval $[0, T]$. Let $\rho_{+}$ ($\eta_{+}$ ) and $\rho_{-}$ ($\eta_{-}$) denote the specific density (viscosity) of the two phases. Let the phase field function, $\phi$, be the variable  that tracks the location of the phases and varies smoothly between $+1$ to $-1$.  We define a non-dimensional density, $\rho(\phi)$, such that 
\begin{equation}
\rho(\phi) = \alpha\phi + \beta, \quad
\text{where} \quad \alpha = \frac{\rho_+ - \;\rho_-}{2\rho_+} 
\quad \text{and} \quad \beta = \frac{\rho_+ + \;\rho_-}{2\rho_+}.
\end{equation}
Note that our non-dimensional form uses the specific density of fluid 1 as the non-dimensionalizing density. Similarly, the non-dimensional viscosity is defined as
\begin{equation}
\eta(\phi) = \gamma\phi + \xi, \quad \text{where} \quad 
\gamma = \frac{\eta_+ - \;\eta_-}{2\eta_+} \quad \text{and}
\quad \xi = \frac{\eta_+ + \;\eta_-}{2\eta_+}.
\end{equation}

Let, $\vec{v}$ be the volume averaged mixture velocity, $p$ the volume averaged pressure, $\phi$ the phase field (interface tracking variable), and $\mu$ the chemical potential.  Then the governing equations in their non-dimensional form are as follows:
\begin{align}
%\begin{empheq}[box=\fbox]{gather}
\begin{split}
	\text{Momentum Eqns:} & \quad \pd{\left(\rho(\phi) v_i\right)}{t} + \pd{\left(\rho(\phi)v_iv_j\right)}{x_j} + \frac{1}{Pe}\pd{\left(J_jv_i\right)}{x_j} +\frac{Cn}{We} \pd{}{x_j}\left({\pd{\phi}{x_i}\pd{\phi}{x_j}}\right) \\
	& \quad \quad \quad + 
	\frac{1}{We}\pd{p}{x_i} - \frac{1}{Re}\pd{}{x_j}\left({\eta(\phi)\pd{v_i}{x_j}}\right) - \frac{\rho(\phi)\hat{{g_i}}}{Fr} = 0, \label{eqn:nav_stokes} %- f_v(x_i,t)
	\end{split} \\
	%Thermodynamic consistency
	\text{Thermo Consistency:} & \quad J_i = \frac{\left(\rho_- - \rho_+\right)}{2\;\rho_+ Cn} \, m(\phi) 
	\, \pd{\mu}{x_i},\\
	%Solenoidality
	\text{Solenoidality:} & \quad \pd{v_i}{x_i} = 0, \label{eqn:solenoidality}\\
	\text{Continuity:} & \quad \pd{\rho(\phi)}{t} + \pd{\left(\rho(\phi)v_i\right)}{x_i}+
	\frac{1}{Pe} \pd{J_i}{x_i} = 0, \label{eqn:cont}\\
	\text{Chemical Potential:} & \quad \mu = \psi'(\phi) - Cn^2 \pd{}{x_i}\left({\pd{\phi}{x_i}}\right)  ,\label{eqn:mu_eqn} %+ f_{\mu}(x_i,t)
	\\ 
	\text{Cahn-Hilliard Eqn:} & \quad \pd{\phi}{t} + \pd{\left(v_i \phi\right)}{x_i} - \frac{1}{PeCn} \pd{}{x_i}\left(m(\phi){\pd{\mu}{x_i}}\right) = 0. %- f_{\phi}(x_i,t)
	\label{eqn:phi_eqn}
%\end{empheq}	
\end{align}
Non-dimensional mobility $m(\phi)$ is assumed to be a constant value of one. $\hat{\vec{g}}$ is a unit vector defined as $\left(0, -1, 0\right)$ denoting the direction of gravity.  We use the polynomial form of the free energy density $\psi(\phi(\vec{x}))$ defined as follows:
\begin{align}
	\psi(\phi) = \frac{1}{4}\left( \phi^2 - 1 \right)^2 \qquad \text{with} \qquad
	\psi'(\phi) = \phi^3 - \phi.
	\label{eqn:free_energy}
\end{align}
%Note that we use Einstein notation throughout this work; in this notation $v_i$ represents the $i^{\text{th}}$ component of the vector $\vec{v}$, and any repeated index is implicitly summed over.

\subsection{Details of non-dimensionalization} Let $u_{r}$ and $L_r$ denoting the reference velocity and length, $m_{r}$ is the reference mobility, $\sigma$ is the scaling interfacial tension, $\nu_{r} = \eta_{+}/\rho_{+}$, $\varepsilon$ is the interface thickness, $g$ is gravitational acceleration. Then relevant non-dimensional parameters are as follows: Peclet, $Pe = \frac{u_{r} L_{r}^2}{m_{r}\sigma}$; Reynolds, $Re = \frac{u_{r} L_{r}}{\nu_{r}}$; Weber, $We = \frac{\rho_{r}u_{r}^2 L_{r}}{\sigma}$; Cahn, $Cn = \frac{\varepsilon}{L_{r}}$; and Froude, $Fr = \frac{u_{r}^2}{gL_{r}}$, with $u_{r}$ and $L_r$ denoting the reference velocity and length, respectively. %Here, $m_{r}$ is the reference mobility, $\sigma$ is the scaling interfacial tension, $\nu_{r} = \eta_{+}/\rho_{+}$, $\varepsilon$ is the interface thickness, $g$ is gravitational acceleration. 
We use scaling relation suggested by~\citep{ Magaletti2013} of $1/Pe = 3 Cn^2$ to calculated appropriate $Pe$ for a chosen $Cn$\footnote{See Remark 1 in~\citet{Khanwale2021}.}.   
%
%The source terms $f_{v}, f_{\mu}, f_{\phi}$ are only used for numerical testing using manufactured solutions and are set to zero for the simulations of realistic systems. 
%
%

\subsection{Energy law} The system of equations \cref{eqn:nav_stokes} -- \cref{eqn:phi_eqn} has a dissipative law given by: 
\begin{equation}
\d{E_{\text{tot}}}{t} = -\frac{1}{Re}  \int_{\Omega} \frac{\eta(\phi)}{2} \norm{\nabla \vec{v}}_F^2 \mathrm{d}\vec{x} - \frac{Cn}{We} \int_{\Omega}m(\phi) \norm{\nabla \mu}^2  \mathrm{d}\vec{x},
\end{equation}
where the total energy is 
\begin{equation}
E_{\text{tot}}(\vec{v},\phi, t) = \int_{\Omega}\frac{1}{2}\rho(\phi) \norm{\vec{v}}^2 \mathrm{d}\vec{x} + \frac{1}{CnWe}\int_{\Omega} \left(\psi(\phi) + \frac{Cn^2}{2} \norm{\nabla\phi}^2 + \frac{1}{Fr} \rho(\phi) y \right) \mathrm{d}\vec{x}.
\label{eqn:energy_functional}
\end{equation}
The norms used in the above expression are the Euclidean vector norm and the Frobenius matrix norm:
\begin{equation}
\norm{\vec{v}}^2 := \sum_i \abs{v_i}^2 \qquad \text{and} \qquad 
\norm{\nabla\vec{v}}^2_F := \sum_i \sum_j \abs{\frac{\partial v_i}{\partial x_j}}^2.
\end{equation}
This energy functional does not consider energy fluxes into/out of the boundaries. This scenario is representative of closed systems, systems with neutral contact line dynamics, or generally boundary conditions that do not contribute to energy.

\section{Numerical method and its properties}
\label{sec:numerical_tecniques}

We utilize a semi-implicit time-marching scheme for Navier Stokes (\cref{eqn:nav_stokes}) with projection to decouple the pressure computation.  Similar to \citet{Khanwale2021} we use a fully implicit time marching scheme for the Cahn-Hilliard equations~\crefrange{eqn:mu_eqn}{eqn:phi_eqn}.  This strategy delivers second order accuracy accuracy and energy-stability, which we verify using numerical experiments, for larger time-steps that are comparable to the fully implicit scheme in~\citet{Khanwale2021}. 
%Additionally, using this implicit time scheme allows us to prove existence of solutions in the semi-discrete sense for the Cahn-Hilliard equation.

%Next we write the discrete time integration scheme. 
Let $\delta t$ be a time-step and let $t^k := k \delta t$. We define the following time-averages: 
\begin{gather}
\label{eqn:averaging1}
\begin{split}
\widetilde{\vec{v}}^{k} := \frac{\vec{u}^{k} + \vec{v}^{k+1}}{2}, \quad \widetilde{\vec{u}}^{k} :=& \frac{\vec{u}^{k} + \vec{u}^{k+1}}{2}, \quad \widehat{\vec{v}}^{k} := \frac{3\vec{u}^{k} + \vec{u}^{k-1}}{2}, \quad \tp^{k} := \frac{{p}^{k+1}+{p}^{k}}{2}, \\
\tphi^{k} :=& \frac{{\phi}^{k+1} + {\phi}^{k}}{2}, \quad \text{and} \quad \tmu^{k} := \frac{{\mu}^{k+1} + {\mu}^{k}}{2},
\end{split}
\end{gather}
as well as the following function evaluations:
\begin{equation}
\label{eqn:psi_ave_def}
\begin{split}
	\tpsi^k := \psi\left( \tphi^k \right), \quad %\;\; %\text{and} \qquad
	\tpsi'^{k} :=& \psi'\left( \tphi^k \right), \quad %\;\;
	\trhok := \rho\left(\tphi^k\right), \quad
	\rhokplusOne := \rho\left(\phi^{k+1}\right), \\
	\rhok :=& \rho\left(\phi^{k}\right), \quad \text{and} \quad
	\tetak := \eta\left(\tphi^k\right).
\end{split}
\end{equation}
Using these temporal values, we define our time-discretized weak form of
the Cahn-Hilliard Navier-Stokes (CNHS) equations. 

\begin{definition}[Variational form ]
	Let $(\cdot,\cdot)$ be the standard $L^2$ inner product. We state the time-discrete variational problem as follows: find $\vec{v}^{k+1}(\vec{x}), \vec{u}^{k+1}(\vec{x}) \in \vec{H}_0^1(\Omega)$, $p^{k+1}(\vec{x})$, $\phi^{k+1}(\vec{x})$, $\mu^{k+1}(\vec{x})$ $\in {H}^1(\Omega)$ such that
	\begin{align}
	\text{Cahn-Hilliard Eqn:} & \quad \left(q, \frac{\phi^{k+1} - \phi^k}{\delta t} \right) -
	\left(\pd{q}{x_i}, \, \tui^{k} \tphi^{k} \right) + \frac{1}{PeCn} \left(\pd{q}{x_i}, \, {\pd{\tmu^{k}}{x_i}} \right) = 0,
	\label{eqn:phi_eqn_var_semi_disc}\\
	\text{Chemical Potential:} & \quad 
	-\left(q,\tmu^{k}\right) + \left(q, \tpsi'^{k} \right) + Cn^2 \left(\pd{q}{x_i}, \, {\pd{\tphi^{k}}{x_i}}\right)   = 0, \label{eqn:mu_eqn_var_semi_disc}\\
	\begin{split}
	\text{Velocity prediction:} & \quad \left(w_i, \, \trhok \, \frac{v^{k+1}_i - u^k_i}{\delta t}\right) 
	%+ \frac{1}{2}\left(w_i, \, \frac{\left(\rhokplusOne - \rhok\right)}{\delta t}\tvi^{k}\right)\\ & \quad 
	+ \left(w_i, \, \trhok \, \huj^{k} \, \pd{\tvi^{k}}{x_j}\right) 
	%+ \frac{1}{2}\left(w_i, \, \pd{\left(\trhok \huj^k \right)}{x_j}\tvi^{k}\right)\\ & \quad 	
	+ \frac{1}{Pe}\left(w_i, \, \hJj^{k} \, \pd{\tvi^{k}}{x_j}\right) \\ & \quad
	%+ \frac{1}{2 Pe}\left(w_i, \,  \pd{\hJj}{x_j}\tvi^{k}\right)\\ & \quad
	- \frac{Cn}{We} \left(\pd{w_i}{x_j}, \, {\pd{\tphi^{k}}{x_i}\pd{\tphi^{k}}{x_j}}\right)  
	+ \frac{1}{We}\left(w_i, \, \pd{p^{k}}{x_i} \right) \\ & \quad 
	+ \frac{1}{Re}\left(\pd{w_i}{x_j}, \,\tetak {\pd{\tvi^{k}}{x_j}} \right) -
	\left(w_i,\frac{\trhok \, \hat{g_i}}{Fr}\right) = 0, 
	\label{eqn:nav_stokes_var_semi_disc}
	\end{split} \\ 
	\text{Thermo Consistency:} & \quad \tJi^{k} = \frac{\left(\rho_- - \rho_+ \right)}{2\;\rho_+ Cn} \, \pd{\tmu^{k}}{x_i}, \label{eqn:thermo_consistency_semi_disc} \\
	\text{Projection:} & \quad \left(w_i, \, \trhok \, \frac{u^{k+1}_i - v^{k+1}_i}{\delta t}\right) 
	+ \frac{1}{2 We}\left(w_i, \, \pd{\left(p^{k+1} - p^{k}\right)}{x_i} \right) = 0, \label{eqn:projection_semi_disc} \\
	\text{Solenoidality:} & \quad 
	\left(q, \, \pd{u^{k+1}_{i}}{x_i}\right) = 0, 
	\quad \left(q, \, \pd{u^{k}_{i}}{x_i}\right) = 0, \quad \left(q, \, \pd{u^{k-1}_{i}}{x_i}\right) = 0,
	\label{eqn:cont_var_semi_disc}\\
	\text{Continuity:} & \quad
	\left(\frac{\rhokplusOne - \rhok}{\delta t}, \, q\right) + \left(\pd{\left(\trhok \huj^k \right)}{x_j}, \, q\right)-
	\left(\frac{1}{Pe} \hJj, \, \pd{q}{x_j}\right) = 0,
	\label{eqn:cont_consv_var_semi_disc}
	\end{align}
	$\forall \vec{w} \in \vec{H}^1_0(\Omega)$, $\forall q \in H^1(\Omega)$, given $\vec{u}^{k}, \vec{u}^{k-1} \in \vec{H}_0^1(\Omega)$, and $\phi^{k},\mu^{k} \in H^1(\Omega)$.
	\label{defn:variational_form_sem_disc}
\end{definition}
We redefine the pressure by absorbing the coefficient $1/We$ in its definition. Using the solenoidality of $\vec{u}^{k+1}$, \cref{eqn:projection_semi_disc} is implemented in a two step strategy as follows.
\begin{definition}
	Let $(\cdot,\cdot)$ be the standard $L^2$ inner product. The time-discretized projection variational
	problem can be stated as follows: find $\vec{u}^{k+1}(\vec{x}) \in \vec{H}_0^1(\Omega)$, $p^{k+1}(\vec{x})$ $\in {H}^1(\Omega)$ such that
	\begin{align}
	\text{Pressure Poisson:}& \qquad \left(\pd{w_i}{x_i}, \, \left(\frac{1}{\trhok}\pd{p^{k+1}}{x_i}\right)\;\right) = 
	-\frac{2}{\dt} \left(w_i, \pd{v_i^{k+1}}{x_i}\right) + \left(\pd{w_i}{x_i}, \, \left(\frac{1}{\trhok}\pd{p^{k}}{x_i}\right)\;\right), \label{eqn:press_poisson_semi_disc}\\
	\text{Velocity update:} & \qquad \left(w_i, \, \trhok u_i^{k+1}\right) + \frac{\dt}{2} \left(w_i, \, \pd{p^{k+1}}{x_i} \right) = 
	\left(w_i, \, \trhok v_i^{k+1}\right) + \frac{\dt}{2}\left(w_i, \,\pd{p^{k}}{x_i}\right), \label{eqn:vel_update_semi_disc}
	\end{align}
	$\forall \vec{w} \in \vec{H}^1_0(\Omega)$, given $\vec{v}^{n+1} \in \vec{H}_0^1(\Omega)$, and $p^{k} \in H^1(\Omega)$.
	\label{defn:hp_semi_disc}
\end{definition}
Here the pressure Poisson equation is derived by dividing~\cref{eqn:projection_semi_disc} by $\trhok$ and then taking the divergence.  Once the current pressure is known, we can update the velocity using~\cref{eqn:vel_update_semi_disc}.
\begin{remark}
	In~\cref{defn:variational_form_sem_disc} and \cref{defn:hp_semi_disc} for the momentum equations and projection equations, the boundary terms in the variational form are zero because the velocity and the basis functions live in $\vec{H}^1_0(\Omega)$. Also we use the no flux boundary condition for $\phi$ and $\mu$, which makes the boundary terms zero (i.e.,  $\left(q, {\pd{\tphi^{k}}{x_i}} \hat{n}_i\right)=0$ and $\left(q, {\pd{\mu^{k}}{x_i}} \hat{n}_i\right)=0$).  This choice of boundary conditions ensure that there are no boundary terms when the equations are weakened in the variational form.  
	Other boundary conditions suitable for energy stability analysis using this particular energy functional include closed systems with neutral contact line dynamics. Additionally, a slight generalization of these boundary conditions to periodic conditions is also suitable for an energy-stability analysis as there are no boundary terms for this case.  
	%We use these boundary conditions for all the proofs in this paper.  
	The numerical examples in the paper use these boundary conditions unless explicitly noted otherwise.
\end{remark}

\begin{remark} While $\phi \in [-1, 1]$ in the continuous equations, the discrete $\phi$ may violate these bounds. These bound violations may not change the dynamics of $\phi$  adversely, but they could lose the strict positivity of some quantities which depend on $\phi$  (e.g., mixture density $\rho(\phi)$ and viscosity $\eta(\phi)$). This effect is especially significant for high density and viscosity contrasts.
	%This causes an drift of the bulk phase density from the true specific density of that phase, with some locations exhibiting negative density (or viscosity). This effect is especially possible for very high density ratio between the two fluids, like in the case of a water-air system ($1:10^{-3}$). 
	We fix this issue by saturation scaling (i.e., we pull back the value of $\phi$ only for the calculation of density and viscosity). In particular, we define the following $\phi^*$ for the mixture density and viscosity calculations: 
	\begin{equation}
	\phi^* := 
	\begin{cases}
	\phi, &\quad \text{if} \;\; \abs{\phi} \leq 1, \\
	\mathrm{sign}(\phi), &\quad \text{otherwise.} 
	\end{cases}
	\label{eqn:phi_for_density}
	\end{equation} 
	\label{rmk:phi_pullback}
\end{remark}

\begin{remark}
	It is important to note here that we are using a block iteration technique.  Therefore, $\phi$ and $\mu$ are known when we are solving the momentum equations, and $v_i$ is known when we are solving the advective Cahn-Hilliard equation.  Also note that the mixture velocity used for advecting the phase field in \cref{eqn:phi_eqn_var_semi_disc} is $\tui$, which is weakly solenoidal instead of $\tvi$ which is not.  We design the solve strategy such that this is always satisfied. 
	%The theorems and proofs we present in the subsequent subsections all assume that we are using the block iterative technique. However, it is not difficult to extend these theorems and proofs for the case of a fully coupled implementation; the theorems of unconditional stability and existence presented here hold even in the fully coupled case. 
	%We now prove some propositions and lemmas needed to prove the energy estimate.
\end{remark}
We now state a couple of important properties of the scheme. 
\begin{proposition}[Mass conservation]
	The scheme of~\cref{eqn:phi_eqn_var_semi_disc} -- \cref{eqn:mu_eqn_var_semi_disc} with the following boundary conditions: 
	\begin{equation}
	\pd{\tmu}{x_i} \hat{n}_i\Bigl\rvert_{\partial \Omega} = 0, \quad \pd{\tphi}{x_i} \hat{n}_i\Bigl\rvert_{\partial \Omega} = 0, \quad \widetilde{\vec{u}}^{k}\Bigl\rvert_{\partial \Omega} = \vec{0},
	\end{equation}
	where $\hat{\vec{n}}$ is the outward normal to the boundary $\partial \Omega$,
	is globally mass conservative: 
	\begin{equation}
	\int_{\Omega} \phi^{k+1} \, \mathrm{d}\vec{x}  = \int_{\Omega} \phi^{k} \, \mathrm{d}\vec{x}.
	\end{equation}
	\label{prop:mass_conservation}
\end{proposition}
\begin{proof}
	The proposition can be easily proved using the test function in~\cref{eqn:phi_eqn_var_semi_disc} as 1, and subsequently using the divergence theorem along with the given boundary conditions prescribed by the function spaces. 
\end{proof} 
We verify this claim numerically in~\cref{subsec:manfactured_soln_result,subsec:single_rising_drop_2D}.  We recall the following result from~\citet{Khanwale2020}.  
\begin{lemma}[Weak equivalence of forcing]
	The forcing term due to Cahn-Hilliard in the momentum equation,~\cref{eqn:nav_stokes_var_semi_disc}, with the test function $w_i = \delta t \, \tvi^k$, becomes
	\begin{equation}
	\frac{Cn}{We}\left( \pd{}{x_j}\left({\pd{\tphi^k}{x_i}\pd{\tphi^{k}}{x_j}}\right), \delta t \, \tvi^k\right) = \frac{\delta t}{WeCn}\left({\tphi}^k\pd{\tmu^{k}}{x_i},\tvi^k\right) 
	+ \frac{\delta t~Cn}{We} \left(\frac{\tmu^{k} \, \tphi^{k}}{Cn^2}  -
	\frac{1}{2} \pd{\tphi^{k}}{x_j} \pd{\tphi^{k}}{x_j} - \frac{\tpsi}{Cn^2} , \, \pd{\tvi^k}{x_i}\right),
	\label{eqn:lemma_weak_equiv_forcing}
	\end{equation}
	$\forall \;\; \tphi^k$, $\tmu^{k} \in H^1(\Omega)$, and $\forall \;\; \widetilde{\vec{v}}^k \in  \vec{H}_{0}^1(\Omega)$, where $\vec{v}^k, \vec{v}^{k+1}, p^k, p^{k+1}, \phi^k, \phi^{k+1}, \mu^{k},\mu^{k+1}, $ satisfy \cref{eqn:nav_stokes_var_semi_disc} -- \cref{eqn:phi_eqn_var_semi_disc}.
	\label{lem:forcing_equivalent}
\end{lemma}
\begin{remark}
	In the literature, some variant of the forcing term ${\tphi}^k\pd{\tmu^{k}}{x_i}$, which is the first term on the right hand side of the expression above, is used effectively to construct semi-implicit schemes~\citep{Shen2010a, Shen2010b, Dong2012, Shen2015, Han2015, Chen2016, Dong2016, Guo2017, Zhu2019}. These forcing terms are equivalent to the thermodynamically consistent forcing term, $\pd{}{x_j}\left(\pd{\tphi^k}{x_i}\pd{\tphi^{k}}{x_j}\right)$, \textit{\textbf{only}} if the test function is divergence free. 
	%However, we want to emphasize that the aforementioned forcing is not equivalent to $\pd{}{x_j}\left(\pd{\tphi^k}{x_i}\pd{\tphi^{k}}{x_j}\right)$ which is the forcing derived for thermodynamic consistency (energy stability) if the test function is not divergence-free.
	The usual test functions used for a projection method with conforming Galerkin method are not divergence free. Thus, using a forcing term ${\tphi}^k\pd{\tmu^{k}}{x_i}$ results in methods that do not exhibit appropriate temporal convergence as conservation errors saturate.  We show this effect numerically in panel (b) of~\cref{fig:manufac_temporal_convergence} in~\cref{subsec:manfactured_soln_result}. 
	\label{rem:consv_forcing}
\end{remark}

We solve the spatially discretized version of variational problems in~\cref{defn:variational_form_sem_disc} and \cref{defn:hp_semi_disc} using a block iteration technique, i.e., we treat the velocity prediction equations (\cref{eqn:nav_stokes_var_semi_disc}), pressure Poisson equation (\cref{eqn:press_poisson_semi_disc}), velocity update equations (\cref{eqn:vel_update_semi_disc}), and the Cahn-Hilliard equations (\crefrange{eqn:phi_eqn_var_semi_disc}{eqn:mu_eqn_var_semi_disc}) as distinct sub-problems. Thus, three linear solvers of (1) velocity prediction, (2) pressure Poisson, and (3) velocity update are stacked together with a non-linear solver for Cahn-Hilliard equations inside the time loop. These solvers are each solved twice (two blocks) within every time step to preserve order of accuracy and self-consistency.  See \cref{fig:flowchart_block} for a  flowchart of the approach. We emphasize that a block iterative approach allows us to make the coupling variables from one equation constant in the other during each respective linear/non-linear solve.  For example, for the momentum equation, all the terms depending on $\phi$ (which is solved in the Cahn-Hilliard sub problem) are known. Similarly the mixture velocity used in the Cahn-Hilliard equation solve.  We choose to perform two-block iterations as for the timesteps we choose for accuracy concerns, we achieve convergence in block iteration in two blocks.  

We adopt a strategy here such that the mixture velocity used in the advection of the phase field (\cref{eqn:phi_eqn_var_semi_disc}) is always weakly solenoidal.  This means that the pressure Poisson and velocity update is performed two times as we run the block iteration twice.  This is a robust strategy, especially for the cases of high density ratios where the pressure gradients between the two phases can be high.  Another popular strategy, used for equal density ratios by~\citet{Han2015}, is to iterate between velocity prediction (\cref{eqn:nav_stokes_var_semi_disc}) and Cahn-Hilliard equations (\crefrange{eqn:phi_eqn_var_semi_disc}{eqn:mu_eqn_var_semi_disc}) until a consistency tolerance is satisfied, and then project the velocity to calculate pressure and solenoidal velocity in the last block iteration.  We do not use this strategy even though it is potentially cheaper (only requires pressure Poisson and velocity update to be performed once) due to the fact that we are not using divergence conforming elements unlike~\citet{Han2015} and we simulate large density ratios.

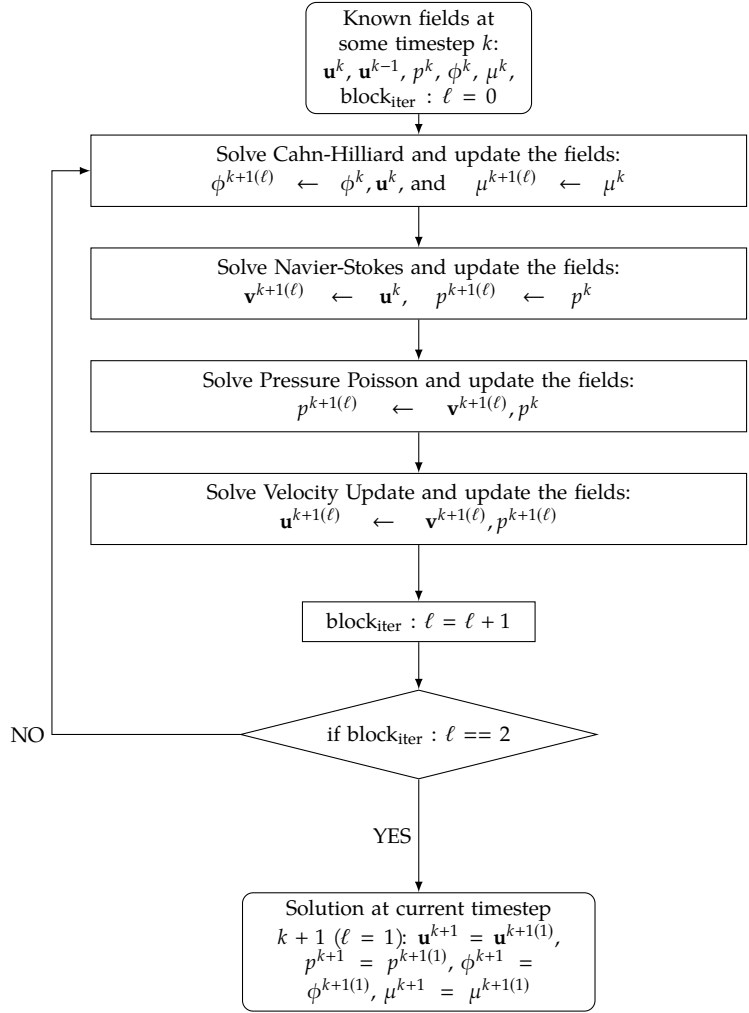
\begin{figure}
	\centering			
	\begin{tikzpicture}[%
	scale=1,transform shape,
	>=latex,              % Nice arrows; your taste may be different
	start chain=going right,    % General flow is top-to-bottom
	node distance=15mm and 15mm, % Global setup of box spacing
	every join/.style={norm},   % Default linetype for connecting boxes
	]
	\scriptsize
	
	\tikzset{
		start/.style={rectangle, draw, text width=3cm, text centered, rounded corners, minimum height=2ex},
		base/.style={draw, on chain, on grid, align=center, minimum height=2ex},
		proc/.style={base, rectangle, inner sep=4pt}, %, text width=8em},
		inout/.style={base,trapezium,trapezium left angle=70,trapezium right 
			angle=-70}, %, text width=10em},
		decision/.style={base, diamond, aspect = 4.0}, %, text width=10em},
		norm/.style={->, draw},
		%%  it/.style={font={\small\itshape}}
	}
	
	%\node [] () {}
	\node [start,text width=10em] (b0) {Known fields at some timestep $k$:  $\vec{u}^{k}$, $\vec{u}^{k-1}$, $p^{k}$, $\phi^{k}$, $\mu^{k}$, \\ $\text{block}_{\text{iter}}: \ell = 0$};
	\node [proc,below =of b0,yshift=0pt,xshift=0pt,text width=30em] (b1) {Solve Cahn-Hilliard and update the fields: \\
		$\phi^{k+1(\ell)} \leftarrow \phi^k,\vec{u}^{k}$, and \quad $\mu^{k+1(\ell)} \leftarrow \mu^k$};
	\node [proc,below =of b1,join,text width=30em] (b2) {Solve Navier-Stokes and update the fields: \\ $\vec{v}^{k+1(\ell)} \leftarrow \vec{u}^k$, \quad $p^{k+1(\ell)} \leftarrow p^k$};
	\node [proc,below =of b2,join,text width=30em] (b3) {Solve Pressure Poisson and update the fields: \\ $p^{k+1(\ell)} \leftarrow \vec{v}^{k+1(\ell)}, p^k$};
	\node [proc,below =of b3,join,text width=30em] (b4) {Solve Velocity Update and update the fields: \\ $\vec{u}^{k+1(\ell)} \leftarrow \vec{v}^{k+1(\ell)}, p^{k+1(\ell)}$};
	\node [proc,below =of b4,yshift=0pt,xshift=0pt,text width=10em] (b5) {$\text{block}_{\text{iter}}: \ell = \ell + 1$};
	%\node [proc,below =of b5,yshift=0pt,xshift=0pt,text width=30em] (b6) {$\ell^\text{th}$ block iteration Navier-Stokes: $\ell=\text{block}_{\text{iter}}$ \\ Solve Navier-Stokes and update the fields$:\vec{v}^{k+1(\ell)} \leftarrow \vec{v}^{k+1(\ell-1)}$, \quad $p^{k+1(\ell)} \leftarrow p^{k+1(\ell-1)}$};
	%\node [proc,below =of b6,yshift=0pt,xshift=0pt,text width=30em] (b7) {$\ell^\text{th}$ block Cahn-Hilliard: $\ell = \text{block}_{\text{iter}}$ \\ Solve Cahn-Hilliard and update the fields$:\phi^{k+1(\ell)} \leftarrow \phi^{k+1(\ell-1)}$, \quad $\mu^{k+1 (\ell)} \leftarrow \mu^{k+1 (\ell-1)}$};
	\node[decision,below =of b5,yshift=0pt,xshift=0pt,text width=10em](b6) {if $\text{block}_{\text{iter}} : \ell == 2$};
	%\node [proc,below =of b4,yshift=0pt,xshift=0pt,text width=16em] (b5) {Parse input data to setup local to global mapping (LM Array)};
	\node [start,below =of b6,yshift=0pt,xshift=0pt,text width=16em] (b7) {Solution at current timestep $k+1$ ($\ell = 1$):  $\vec{u}^{k+1} = \vec{u}^{k+1(1)}$, $p^{k+1} = p^{k+1(1)}$, $\phi^{k+1} = \phi^{k+1(1)}$, $\mu^{k+1} = \mu^{k+1(1)}$};

	\draw [->] (b0)--(b1);		
%	\draw [->] (b1)--(b2);
%	\draw [->] (b2)--(b3);
%	\draw [->] (b3)--(b4);
	\draw [->] (b4)--(b5);
	\draw [->] (b5) -- (b6);
	\draw [->] (b6.west) -| ++(-2.5,0) node[anchor=east] {NO} |- (b1.west);
	\draw [->] (b6) -- node[anchor=east] {YES}(b7);
	
	\end{tikzpicture}
	\caption{Flowchart for the block iteration technique as described
		in \cref{sec:numerical_tecniques}
		.}
	\label{fig:flowchart_block}		
\end{figure}

\subsection{Spatial discretization and the variational multiscale approach}
\label{subsec:space_scheme}
We discretize the unknowns: 
\begin{equation}
\left( \phi, \, \mu,  \, \vec{v}, \, \vec{u}, \, p \right),
\end{equation}
in space using conforming continuous Galerkin finite elements with piecewise polynomial approximations. However, approximating the velocity, $\vec{v}$, and the pressure, $p$, with the same polynomial order leads to numerical instabilities as this violates the discrete inf-sup condition  or Ladyzhenskaya-Babuska-Brezzi condition (e.g., see \citet[page 31]{ Volker2016}).  Even though the projection method decouples pressure and velocity, our use of equal interpolation order basis functions for pressure and velocity may lead to instabilities.  In order to overcome this difficulty, we use the variational multi-scale (VMS) method~\citep{ hughes2018multiscale}, which adds stabilization terms to the pressure Poisson equation that transform the inf-sup stability condition to a coercivity statement \citep{article:TezMitRayShi92}.  
%In particular, a large class of stabilization techniques derive from the so-called {\it variational multiscale} approach (VMS)~\citep{ hughes2018multiscale}, which generalizes the well-known 
 %SUPG/PSPG~\citep{ article:BrooksHughes1982} method in the context of large-eddy simulation (LES)~\citep{ Hughes1995}.  %This approach provides a stabilization mechanism such that the inf-sup stability condition is converted to a coercivity condition. 
Additionally, VMS provides a natural leeway into modeling high-Reynolds number flows~\citep{ Bazilevs2007} as it uses a projection based filter to decompose coarse and fine scales.        
  
The VMS approach uses a direct-sum decomposition of the function spaces as follows.  If $\vec{v} \in \vec{V}$, $p \in Q$, and $\phi \in Q$ then we decompose these spaces as: 
\begin{eqnarray}
\vec{V} = \vec{V}^{c} \oplus \vec{V}^f \qquad \text{and} \qquad Q = Q^{c} \oplus Q^f,
\end{eqnarray}
where $\vec{V}^{c}$ and $Q^{c}$ are the finite dimensional cG(r) subspaces of $\vec{V}$ and $Q$, respectively, and the superscript $f$ versions are the complements of the cG(r) subspaces in $\vec{V}$ and $Q$, respectively.
We decompose the velocity as follows: 
\begin{equation}
\vec{v} = \vec{v}^{c} + \vec{v}^{f}, \quad
%\phi = \overline{\phi} + \phi^{\prime}, \quad \text{and}
%\quad p = \overline{p} + p^{\prime},
\end{equation} 
where the {\it coarse scale} components are
  $\vec{v}^{c} \in \vec{V}^{c}$, and the {\it fine scale} components are
   $\vec{v}^{f} \in \vec{V}^f$.  We define a projection operator, $\mathscr{P}:\vec{V} \rightarrow \vec{V}^c$, such that
 $\vec{v}^{c} = \mathscr{P}\{\vec{v}\}$ and $\vec{v}^f = \vec{v} - \mathscr{P}\{\vec{v}\}$.   
Let 
\begin{equation}
\tvi^{c, k} = \frac{v^{c,k+1}_i  + u^{c,k}_i}{2}, \quad \tvi^{f, k} = \frac{v^{f,k+1}_i + u^{f,k}_i}{2}, \quad \text{and} \quad \tvi^{c, k} + \tvi^{f, k} = \tvi^{k}.
\end{equation}
Substituting this decomposition in the original variational form  
\cref{defn:variational_form_sem_disc} yields:
\begin{align}	
\begin{split}	
		\text{Velocity prediction:}& \quad \left(w_i,\trhok\, \frac{\left(v^{c,k+1}_i + v^{f,k+1}_i\right) - \left(u^{c,k}_i + u^{f,k}_i\right) %u^k_i
		}{\delta t}\right) + \left(w_i, \, \frac{\left(\rhokplusOne - \rhok\right)}{\delta t}\tvi^{f, k}\right)\\ &
		%+ \left(w_i, \, \frac{\left(\rhokplusOne - \rhok\right)}{\delta t}\tvi^{f, k}\right) \\ &
		%+ \frac{1}{2}\left(w_i, \, \frac{\left(\rhokplusOne - \rhok\right)}{\delta t}\tvi^{f, k}\right)\\ &
		+ \left(w_i, \, \trhok \, \huj^{k} \, \pd{\tvi^{c,k}}{x_j}\right) 
		+ \textcolor{blue}{\left(w_i, \, \trhok \, \huj^{k} \, \pd{\tvi^{f,k}}{x_j}\right)} 
		+ \left(w_i, \, \pd{\left(\trhok \huj^{k}\right)}{x_j}\,\tvi^{f,k}\right)\\ &
		%+ \frac{1}{2}\left(w_i, \, \pd{\left(\trhok \huj^k \right)}{x_j}\tvi^{c, k}\right)
		%+ \frac{1}{2}\left(w_i, \, \pd{\left(\trhok \huj^k \right)}{x_j}\tvi^{c, f}\right) \\ &
		%+ \left(w_i,\pd{\left(\rho(\phi)\overline{v_j}v_i^\prime\right)}{x_j}\right) 
		%+ \left(w_i,\pd{\left(\rho(\phi)v_j^\prime v_i^\prime\right)}{x_j}\right) 
		+ \frac{1}{Pe}\left(w_i, \, \hJj^{k} \, \pd{\tvi^{c, k}}{x_j}\right)
		+ \textcolor{blue}{\frac{1}{Pe}\left(w_i, \, \hJj^{k} \, \pd{\tvi^{f, k}}{x_j}\right)}
		+ \frac{1}{Pe}\left(w_i, \, \pd{\hJj^{k} }{x_j}\, \tvi^{f, k}\right) \\&
		%+ \frac{1}{2 Pe}\left(w_i, \,  \pd{\hJj}{x_j}\tvi^{c, k}\right)
		%+ \frac{1}{2 Pe}\left(w_i, \,  \pd{\hJj}{x_j}\tvi^{f, k}\right) \\&
		- \frac{Cn}{We} \left(\pd{wi}{x_j}, \,\left({\pd{\tphi^{h,k}}{x_i}\pd{\tphi^{h,k}}{x_j}}\right)\right)   
		+ \frac{1}{We}\left(w_i,\pd{p^{k}}{x_i}\right) \\&
		+ \frac{1}{Re}\left(\pd{w_i}{x_j},\tetak\,\pd{\left(\tvi^{c, k} + \tvi^{f, k}\right)}{x_j}\right) - \left(w_i,\frac{\trhok \, \hat{g_i}}{Fr}\right) = 0,   
		%\\&+  \left(q,\pd{\overline{v_i}}{x_i}\right) + \left(q,\pd{v_i^\prime}{x_i}\right)= 0, 
		\label{eqn:weak_VMS_ns}
		\end{split}\\
		\text{Pressure Poisson:}& \quad \left(\pd{w_i}{x_i}, \, \left(\frac{1}{\trhok}\pd{p^{k+1}}{x_i}\right)\;\right) = 
		-\frac{2}{\dt} \left(w_i, \pd{\left(v_i^{c, k+1} + v_i^{f, k+1}\right)}{x_i}\right) + \left(\pd{w_i}{x_i}, \, \left(\frac{1}{\trhok}\pd{p^{k}}{x_i}\right)\;\right),\\ 
		%- \frac{1}{PeCn} \left(\pd{q}{x_i}, \frac{\partial \left(m(\phi)\mu\right)}{\partial x_i}\right) = 0, \label{eqn:phi_eqn_var_decom}\\
		\text{Velocity Update:}& \quad \left(w_i, \, \trhok u_i^{k+1}\right) + \frac{\dt}{2} \left(w_i, \, \pd{p^{k+1}}{x_i} \right) = 
		\left(w_i, \, \trhok \left(v_i^{c, k+1} + v_i^{f, k+1}\right)\right) + \frac{\dt}{2}\left(w_i, \,\pd{p^{k}}{x_i}\right), \label{eqn:mu_eqn_var_decom}
\end{align} 
%where $\vec{w},\vec{\overline{v}}, \in \mathscr{P}\vec{H}^{1}(\Omega)$, $\overline{p},\phi \in \mathscr{P}H^1(\Omega), \vec{v}^\prime \in (\mathscr{I} - \mathscr{P})\vec{H}^{1}(\Omega)$,
%$p^\prime \in (\mathscr{I} - \mathscr{P})H^1(\Omega)$, and $\mu, q \in \mathscr{P}H^1(\Omega)$.
where $\vec{w},\vec{v}^{c, k+1}, \in \mathscr{P}\vec{H}^{r}(\Omega)$, $\vec{v}^{f, k+1} \in (\mathscr{I} - \mathscr{P})\vec{H}^{r}(\Omega)$.
Here $\mathscr{I}$ is the identity operator and $\mathscr{P}$ is the projection operator.  Note that we add the following equation to the original momentum equation (\cref{eqn:nav_stokes_var_semi_disc}) using mass conservation~\cref{eqn:cont_consv_var_semi_disc}:
\begin{equation}
\left(w_i, \, \frac{\left(\rhokplusOne - \rhok\right)}{\delta t}\tvi^{f, k}\right) + \left(w_i, \, \pd{\left(\trhok \huj^{k}\right)}{x_j}\,\tvi^{f,k}\right)+ \frac{1}{Pe}\left(w_i, \, \pd{\hJj^{k} }{x_j}\, \tvi^{f, k}\right) = 0.
\label{eqn:fine_scale_vel_multipliedbyConsv}
\end{equation}
This is important because the terms in blue in~\cref{eqn:weak_VMS_ns} are not in conservative form, and would require higher regularity of $\tvi^{f,k}$ to incorporate.  However, using~\cref{eqn:fine_scale_vel_multipliedbyConsv} along with integration-by-parts the terms in blue can be converted into conservative terms.  Consequently, we now have extra terms:
\begin{equation}
  \left(w_i,\trhok\, \frac{\left(v^{f,k+1}_i - u^{f,k}_i\right)}{\delta t}\right) + \left(w_i, \, \frac{\left(\rhokplusOne - \rhok\right)}{\delta t}\tvi^{f, k}\right).
  \label{eqn:temporal_fine_terms}
\end{equation} 
We can think of this addition as a second order discretization of $\left(w_i, \partial{\left(\rho v_i^{f}\right)}/\partial {t}\right)$. We assume, following~\citep{Bazilevs2007}, that the time derivative of fine scale momentum is zero.

We use the residual-based approximation proposed in~\citet{ Bazilevs2007} for the fine-scale components, applied to a two-phase system~\cite{Khanwale2020}, to close the equations:  
\begin{equation}
\trhok v_i^{f, k+1} = -\tau_m \mathcal{R}_m(\trhok,\vec{v}^{k+1},\vec{u}^{k},\vec{u}^{k-1}, p^k)
\label{eqn:resid_vms}
\end{equation}
with the following parameter values:
\begin{equation}
\begin{split}
\tau_m &= \left( \frac{4}{\Delta t^2}  + v_i^{c}G_{ij}v_j^{c} + \frac{1}{\trhok \, Pe}v_i^{c}G_{ij}J_j^c 
+ C_{I} \left(\frac{\tetak}{\trhok Re}\right)^2 G_{ij}G_{ij}\right)^{-1/2}.
\end{split}
\end{equation}
Here we set $C_{I}$ and $C_{\phi}$ for all our simulations to 6 and the residuals are given by
\begin{equation}
\begin{split}
\mathcal{R}_m\left(\trhok, v_i^{c,k+1},u_i^{k},u_i^{k-1},p^k\right) &= \trhok\frac{v_i^{c,k+1} - u_i^{k}}{\dt} + \trhok \huj^{k}\pd{\tvi^{c, k}\;}{x_j} + 
\frac{1}{Pe}\hJj\pd{\tvi^{c, k}}{x_j} \\ &
+ \frac{Cn}{We} \pd{}{x_j}\left({\pd{\phi^{h,k}}{x_i}\pd{\phi^{h,k}}{x_j}}\right) \\ &+ 
\frac{1}{We}\pd{p^k}{x_i} - \frac{1}{Re}\pd{}{x_j}\left({\tetak \pd{\tvi^{c, k}}{x_j}}\right) - \frac{\trhok\hat{g}}{Fr}.
\end{split}
\end{equation}
In the above expressions, we used the following notation:
\begin{equation}
\trhok := \trhok := \rho\left(\tphi^{h,k}\right) \qquad \text{and} \qquad
\eta^h := \eta\left(\tphi^{h,k}\right).
%\text{and} \quad
%m^h:=m\left(\phi^{h,k}\right),
\end{equation}
%It is important to note that because we are using block iterative method, the momentum equations,  \cref{eqn:weak_VMS_ns}, and the Cahn-Hilliard equations, \cref{eqn:phi_eqn_var_decom} and \cref{eqn:mu_eqn_var_decom},
% are solved as two different nonlinear sub-problems. We use conforming Galerkin based finite elements, and replace the continuous spaces with their discrete counterparts; notice that as we only solve for course scale components, the trial functions and the basis functions are in the same space. Then we can write a discrete variational formulation can we written as follows.
Now, we replace the infinite-dimensional function spaces by their finite dimensional counterparts using conforming Galerkin finite elements where the the trial and test functions are taken from the same spaces.
Note that we only solve for the coarse-scale components. For notational simplicity we do not add the conventional superscript $h$ denoting finite dimensional conforming Galerkin approximations.  All the functions spaces in the definition below are finite dimensional.  The resulting discrete variational formulation can then be defined as follows.
\begin{definition}[VMS setting for velocity]  
	Find $\vec{v}^{c,k+1} \in \mathscr{P}\vec{H}^{r}_0(\Omega)$ such that
	\begin{align}
	\begin{split}
	% line 1
	\text{Velocity prediction:}&  \quad \left(w_i,\trhok\, \frac{v^{c,k+1}_i - u^{k}_i}{\delta t}\right) 
	%+ \frac{1}{2}\left(w_i, \, \frac{\left(\rhokplusOne - \rhok\right)}{\delta t}\tvi^{c, k}\right) 
	%+ \frac{1}{2}\left(w_i, \, \frac{\left(\rhokplusOne - \rhok\right)}{\delta t}\tvi^{f, k}\right)\\ &
	+ \frac{1}{2}\left(w_i, \, \trhok \, \huj^{k} \, \pd{v_i^{c,k+1}}{x_j}\right)
	+ \frac{1}{2}\left(w_i, \, \trhok \, \huj^{k} \, \pd{u_i^{k}}{x_j}\right)  \\ &
	+ \frac{1}{2}\left(\pd{w_i}{x_j}, \, \huj^{k} \, \tau_m\mathcal{R}_m(\trhok,v_i^{c,k+1},u_i^{k},u_i^{k-1}, p^{k})\right) \\ &
	%+ \frac{1}{2}\left(w_i, \, \pd{\left(\trhok \huj^k \right)}{x_j}\tvi^{c, k}\right)
	%+ \frac{1}{2}\left(w_i, \, \pd{\left(\trhok \huj^k \right)}{x_j}\tvi^{c, f}\right) \\ &
	%+ \left(w_i,\pd{\left(\rho(\phi)\overline{v_j}v_i^\prime\right)}{x_j}\right) 
	%+ \left(w_i,\pd{\left(\rho(\phi)v_j^\prime v_i^\prime\right)}{x_j}\right) 
	+ \frac{1}{2Pe}\left(w_i, \, \hJj^{k} \, \pd{v_i^{c, k+1}}{x_j}\right) 
	+ \frac{1}{2Pe}\left(w_i, \, \hJj^{k} \, \pd{u_i^{k}}{x_j}\right)\\ &
	+ \frac{1}{2Pe}\left(\pd{w_i}{x_j}, \, \hJj^{k} \, \frac{1}{\trhok}\tau_m\mathcal{R}_m(\trhok,v_i^{k+1},u_i^{k},u_i^{k-1}, p^k)\right) \\&
	%+ \frac{1}{2 Pe}\left(w_i, \,  \pd{\hJj}{x_j}\tvi^{c, k}\right)
	%+ \frac{1}{2 Pe}\left(w_i, \,  \pd{\hJj}{x_j}\tvi^{f, k}\right) \\&
	- \frac{Cn}{We} \left(\pd{w_i}{x_j},\left({\pd{\tphi^{h,k}}{x_i}\pd{\tphi^{h,k}}{x_j}}\right)\right)   
	+ \frac{1}{We}\left(w_i,\pd{p^{k}}{x_i}\right) \\&
	+ \frac{1}{Re}\left(\pd{w_i}{x_j},\tetak\,\pd{\left(\tvi^{c, k} + \tvi^{f, k}\right)}{x_j}\right) - \left(w_i,\frac{\trhok \, \hat{g_i}}{Fr}\right) = 0,
	\label{eqn:nav_stokes_eqn_var_decom}
	\end{split}
	\end{align}
	$\forall \vec{w} \in \mathscr{P}\vec{H}^{r,h}_0(\Omega)$, given $\vec{u}^{k}, \vec{u}^{k-1} \in \mathscr{P}\vec{H}^{r,h}_0(\Omega)$ and $p^k, \phi^h, \mu^h \in \mathscr{P}H^{r,h}(\Omega)$.
	\label{defn:weak_VMS_disc}
\end{definition}
\begin{remark}
	Note that $\hui^{k}$ is not decomposed here. The velocities at older times have their fine scale terms added to them (see the boxed term in \cref{eqn:vel_update_eqn_var_decom}) in the velocity update step in the previous time step and are completely known. Therefore, for ease of implementation in the code we do not need to decompose them.  When we convert the fine scale velocities in the convective form again to get a complete velocity at time $k$ (see term 3 and term 5), we get extra fine scale terms 
	$-\left(w_i,\, \trhok \left(u_i^{f,k}/dt\right)  \right) + 1/2 \left(w_i,\, \left(\left(\rhokplusOne - \rhok \right)/\dt\right) u_i^{f,k} \right)$. Our experimental results suggest that these terms are very small, and consequently we set them to zero.  We emphasize here that \cref{eqn:nav_stokes_eqn_var_decom} is a linear equation. We deploy a semi-implicit time marching scheme that turns non-linear terms into a linear  algebraic equation 
	that must be solved at each time step. The viscous terms are considered to be implicit.  This type of discretization improves efficiency as a full Newton iteration need not be solved while maintaining stability and robustness as viscous terms are still implicit.  Note that this type of discretization is different from IMEX schemes \cite{article:IMEX_AscherRuuthWetton1995}, where the non-linear terms are completely explicit.
	\label{rmk:weak_vms}
\end{remark}

Now, let us define the variational VMS setting for the projection and update steps. 
\begin{definition}[VMS setting for pressure Poisson]  
	Find $p^{k+1} \in H^{r}(\Omega)$ such that
	\begin{align}
	\begin{split}
	% line 1
	\text{Pressure Poisson:}&\; \left(\pd{w_i}{x_i}, \, \left(\frac{1}{\trhok}\pd{p^{k+1}}{x_i}\right)\;\right) = 
	-\frac{2}{\dt} \left(w_i, \pd{v_i^{c, k+1}}{x_i}\right) 
	\\ & \quad\quad\quad\quad\quad\quad\quad\quad\quad\quad
	\boxed{- \frac{2}{\dt} \left(\pd{w_i}{x_i}, \frac{1}{\trhok}\tau_m \mathcal{R}_m\left(\trhok,v_i^{c,k+1},u_i^{k},u_i^{k-1}, p^{k}\right)\right)}
	\\ & \quad\quad\quad\quad\quad\quad\quad\quad\quad\quad
	+ \left(\pd{w_i}{x_i}, \, \left(\frac{1}{\trhok}\pd{p^{k}}{x_i}\right)\;\right),
	\label{eqn:pp_eqn_var_decom}
	\end{split}
	\end{align}
	$\forall \vec{w} \in \mathscr{P}\vec{H}^{r,h}_0(\Omega)$ and
	$\forall q \in \mathscr{P}H^{r,h}(\Omega)$.
	\label{defn:weak_VMS_disc_pp_vel}
\end{definition}
\begin{remark}
The boxed terms in \cref{defn:weak_VMS_disc_pp_vel} is very important for stability.  If that term is not included we get checkerboard instabilities in pressure as we are using equal order interpolating basis functions for both velocity and pressure.  The boxed term is a manifestation of the splitting of $v_i^{k+1}$ into its coarse scale ($v_i^{c,k+1}$) and fine scale components ($v_i^{f,k+1}$) combined with the use of residual based VMS approximation in \cref{eqn:resid_vms}.
\end{remark}
As the pressure is calculated now, we can write the variational setting for the velocity update which transforms the non-solenoidal velocity $\vec{v}^{c,k+1}$ to a solenoidal velocity $u^{k+1}$ at the current timestep. 
\begin{definition}[VMS setting for velocity update]  
	Find $p^{k+1} \in H^{r}(\Omega)$ such that
	\begin{align}
	\begin{split}
	% line 1
	\text{Velocity update:}\; \quad \left(w_i, \, \trhok u_i^{k+1}\right) + \frac{\dt}{2} \left(w_i, \, \pd{p^{k+1}}{x_i} \right) &= 
	\left(w_i, \, \trhok v_i^{c, k+1}\right)
	\\ & \quad
	\boxed{- \left(w_i, \, \tau_m \mathcal{R}_m\left(\trhok,v_i^{c,k+1},u_i^{k},u_i^{k-1}, p^{k}\right)\right)}
	\\ & \quad
	+ \frac{\dt}{2}\left(w_i, \,\pd{p^{k}}{x_i}\right),
	\label{eqn:vel_update_eqn_var_decom}
	\end{split}
	\end{align}
	$\forall \vec{w} \in \mathscr{P}\vec{H}^{r,h}_0(\Omega)$ given $\vec{v}^{c,k+1} \in \mathscr{P}\vec{H}^{r,h}_0(\Omega)$, and $p^{k+1}, p^{k} \in \mathscr{P}H^{r,h}(\Omega)$,
	\label{defn:weak_VMS_disc_vel_update}
\end{definition}
Here we have again used the splitting of $v_i^{k+1}$ into its coarse scale ($v_i^{c,k+1}$) and fine scale components ($v_i^{f,k+1}$) combined with the use of residual based VMS approximation in \cref{eqn:resid_vms}.
We can now write the fully discrete Cahn-Hilliard variational form as,
\begin{definition} [discrete Cahn-Hilliard equations]
	Find $\phi^h, \mu^h \in H^{r,h}(\Omega)$ such that such that
	\begin{align}
	\begin{split}
	\text{Cahn-Hilliard:}& \quad \left(q, \frac{\phi^{h, k+1} - \phi^{h,k}}{\dt}\right)
	- \left(\pd{q}{x_i}, \tui^{h, k} \tphi^{h, k} \right) 
	- \frac{1}{PeCn} \left(\pd{q}{x_i}, \frac{\partial \left(m^h\tmu^{h,k}\right)}{\partial x_i}\right) = 0, \label{eqn:phi_eqn_var_decom_rb}
	\end{split}\\
	\text{Potential:}& \quad -\left(q,\tmu^{h,k}\right) + \left(q, \psi^\prime\left(\tphi^{h, k}\right)\right) + Cn^2 \left(\pd{q}{x_i},{\pd{\tphi^{h,k}}{x_i}}\right)  = 0,
	\label{eqn:mu_eqn_var_decom_disc}
	\end{align}
	$\forall \,\, \widetilde{\vec{u}}^k \in \vec{H}^{r,h}_0(\Omega)$ and
	$\forall q \in \mathscr{P}H^{r,h}(\Omega)$.
	\label{defn:weak_VMS_disc_CH}
\end{definition}

%Finally, we note that in the above expressions the time derivative is still continuous.  In the fully discrete numerical method we replace the time-derivatives in the momentum and phase field equations using
%the trapezoidal rule in the form of the scheme presented in \cref{eqn:disc_nav_stokes_semi} -- \cref{eqn:disc_time_phi_eqn_semi}.
In addition to assumptions made in \cref{eqn:temporal_fine_terms} and \cref{rmk:weak_vms}, we reiterate the other assumptions made in getting these variational problems. 
\begin{enumerate}
	\item $\vec{v}^{f} = 0$ on the boundary $\partial\Omega$; similarly $\phi^{\prime} = 0$ on  $\partial\Omega$.
	\item $\left(\pd{w_i}{x_j},\tetak \pd{\tvi^{f, k}}{x_j}\right) = 0$ from the orthogonality condition of the projector.  The projector utilizes the inner product that comes from the bilinear form of the viscous terms~\citep{Hughes2007,Bazilevs2007}.
	\item We use the pulled back
	$\phi^{*}$ (see \cref{rmk:phi_pullback}) for this calculation, which regularizes 
	$\phi$ by smoothing out overshoots and undershoots.  
	The pull back ensures $\phi^{*} \in H^{r}$, and that is projection on the mesh $\phi^{*,h} \in H^{r,h}$ as required for the cG formulation. 
\end{enumerate}
\begin{remark}
	The above formulation is written for a generic order ($r$) for the interpolating polynomials (basis functions). Nevertheless, we restrict our attention to $r = 1$ for the numerical experiments in this paper.  
\end{remark}
\begin{remark}
    Here we use an incremental pressure projection scheme. We show with detailed numerical experiments this scheme produces good results.  See \citet{Guermond2006} for an excellent review of pressure projection schemes in the weak setting, including some schemes which use rotational identities to achieve improvement.  We note that the design and efficient HPC implementation of schemes incorporating rotational identities on octree based meshes is not straightforward. This is an active area of research in our groups.
\end{remark}

\section{Octree based domain decomposition}
\label{sec:octree_mesh}
%\textcolor{red}{Masado's section}

% \paragraph{octree structure}
% \paragraph{storage}
% \paragraph{ordering}
% \paragraph{domain decomposition}
% \paragraph{mesh generation algorithms: construction and 2:1 balancing}
% \paragraph{traversal with neighbour data}

%Octrees are widely used in the computational sciences
%to represent dynamically-adapted hierarchical meshes ~\cite{ SundarSampathBiros08, BursteddeWilcoxGhattas11, Fernando2018_GR, Fernando:2017}; this
% is largely due to their conceptual simplicity and their 
% ability to scale across a large number of processors. %For simplicity, in this work we use axis-aligned octrees during the mesh generation process. 
% Adaptivity is crucial in the computational sciences, where in many cases it reduces the overall degrees of freedoms (problem size), making these simulations feasible on currently available computers. The use of adaptive discretizations can introduce additional challenges, especially in distributed computing, such as load-balancing, low-cost scalable mesh generation, and mesh partitioning. Thus, we use Dendro, a highly scalable parallel octree library, to generate full adaptive quasi-structured octree based meshes and partitions. In the following sections, we summarize how Dendro allows us to perform our numerical simulations. The reader can find a detailed account of the algorithms used in Dendro in~\cite{ Fernando:2017, Fernando2018_GR}.
%
We use \dendroKT, a highly scalable parallel octree library, to generate, manage and traverse adaptive octree-based meshes in distributed memory. In this section, we summarize the core data-structures and algorithms used by \dendroKT. Additional details on these algorithms can be found in \cite{ishii2019solving}.  The \dendroKT~interface is the next version of the \dendroFive~interface used in~\citet{Khanwale2021}.  We present the improvements and critical changes in the octree framework compared to~\dendroFive.

\subsection{Distributed memory data structures}
One of the distinguishing aspects of \dendroKT~ compared with other adaptive mesh refinement (AMR) frameworks is that it does not store any element-to-node maps. The overall memory footprint of the adaptive mesh is kept to a minimum, storing only the coordinates of the nodes. All other information are deduced on-the-fly during traversals of the $k$D-tree. Besides reducing the memory-footprint, this also avoids indirect memory accesses (via element-to-node maps) that are extremely inefficient on modern HPC architectures. In the distributed memory setting, this choice has implications for the inter-process and intra-process data structures.

%\paragraph{Inter-process}
A minimal data structure is maintained to send {\em owned} nodal data to $k$D neighboring processes: a list of the MPI processes owning neighboring partitions, and the number of, and list of local indices of, owned nodes to replicate to each neighboring process. We call this data structure a scatter-map, and is similar to the information that is typically used to exchange data to create {\em ghost} or {\em halo} regions. When ghost data is received from remote processes, it is aligned with a list of remotely-owned nodal coordinates, allowing ghost data to be intermingled, sorted, and traversed alongside locally-owned data without special handling.

%\paragraph{Intra-process}
Once the tree is constructed (described in the next section) all tree-related information is discarded, except the coordinates of the leaf octants and coordinates of the nodal points. Direct neighborhood access is not possible because of the lack of neighborhood maps, instead efficient traversal methods are provided (\cref{subsec:traversal}) to enable FEM computations. Because of the lack of traditional maps used by unstructured mesh libraries, we refer to this approach as a {\em mesh-free} approach. 

\subsection{Octree construction and 2:1 balancing} 
\dendroKT~refines an octant based on user-specified criteria proceeding in a top-down fashion. The user defines the refinement criteria by a function that takes the coordinates of the octant, and returns \texttt{true} or \texttt{false}. Since the refinement happens locally to the element, this step is embarrassingly parallel. In distributed-memory machines, the initial top-down tree construction enables an efficient partitioning of the domain across an arbitrary number of processes. All processes start at the root node (i.e., the cubic bounding box for the entire domain). We perform redundant computations on all processes to avoid communication during the refinement stage. Starting from the root node, all processes refine (similar to a sequential implementation) until the process produces at least $\mathcal{O}(p)$ octants requiring further refinement. The procedure ensures that upon partitioning across $p$ processors, each processor gets at least one octant. Then using a space-filling-curve (SFC) based partitioning, we partition the octants across $p$ partitions~\cite{ Fernando:2017}. Once the algorithm completes this partitioning, we can restrict the refinement criterion to a processor's partition, which we can re-distribute to ensure load-balancing.  with $P$, where $p< P$ partitions, according to the adaptive grid generated. 
We enforce a condition in our distributed octrees that no two neighbouring octants differ in size by more than a factor of two (2:1 balancing). This ratio makes subsequent operations simpler without affecting the adaptive properties. Our balancing algorithm uses a variant of \textsc{TreeSort}~\cite{ Fernando:2017} with top-down and bottom-up traversal of octrees which is different from existing approaches~\cite{ bern1999parallel, BursteddeWilcoxGhattas11, SundarSampathAdavaniEtAl07}. As noted above, only the leaf octant coordinates and node coordinates (that lie on the constructed octree) are retained. Nodes that do not represent independent degrees of freedom, such as the so-called {\em hanging nodes} are also discarded.

\subsection{Identification of unique, independent nodes}
Before distributed nodal vectors can be defined and traversed,
the nodes incident on the leaf octants must be enumerated and partitioned.
We refer to the result of enumerating and partitioning the nodes
as a {\em DA}, for {\em distributed array},
as it contains the coordinates of local and remotely accessed nodes,
as well as the process-level scattermap needed for ghost exchange.
It is straightforward to emit the coordinates of the nodes appearing at each element separately.
However, most of these nodes are shared with neighbouring elements, which may reside in other processes,
and some of the nodes are dependent, or {\em hanging}, on an the face of a larger neighbouring element.
The situation is additionally complicated by the possibility of a carved-out domain,
which increases the number of ways that a node can be shared by several neighbouring elements.

In this work the DA construction can be outlined as follows.
Given the input of a distributed list of leaf octants,
two intermediate graphs are produced,
where edges represent dependencies from elements to nodes,
The second graph is used to derive node ownership and the scattermap.

In more detail:
For each leaf octant in the local mesh partition,
the coordinates of candidate nodes are emitted.
Anticipating the coordinates of potential hanging nodes on the element exterior,
a set of {\em cancellation nodes} is also emitted,
at one level deeper in the octree.
Flags on the nodal coordinates indicate whether they are regular or cancellation nodes.
Also, each node is emitted in a pairing with the emitting element,
representing a dependency from the element to the node.

At this point there exists a distributed list of edges,
ordered by element coordinates in the SFC.
The next step is to identify hanging nodes using the cancellation nodes.
The distributed list of edges is re-sorted and -partitioned,
using nodal coordinate as the \textsc{DistTreeSort} key instead of the element.
Thus all edges to a given node are stored contiguously.
\begin{itemize}
  \item If a regular node is found at the same coordinate as a cancellation node,
        then it is hanging.
        There is a bijection between the nodes of a child face that are not shared by its parent,
        and the nodes of the parent face that are not shared with the child.
        Using this bijection, hanging nodes are mapped to non-hanging nodes,
        and the resulting edges are re-emitted.
        Effectively, the original (element, node) dependencies are corrected
        so that a hanging element is able to request the nodes needed for interpolation.
  \item If a regular node is found with no cancellation node at the same coordinate,
        the corresponding edge is emitted as-is.
  \item Edges of cancellation nodes are discarded.
\end{itemize}
After the nodal keys are updated, the second graph is realized
by resorting the edges by the new nodal coordinates.

On each local partition of the graph, the set of elements is \textsc{TreeSort}-ed
against the globally-known splitters of the mesh partition,
producing a map from edges to process rank.
Effectively the graph has been aggregated into a set of (process, node) dependencies.
For each nodal coordinate, one of the connected process ranks is chosen
to be the \textit{owner} of that node.
The other connected process ranks become \textit{borrowers} of the node.
The nodes are sent to the owners with a list of the borrowers,
and to each of the borrowers with a tag specifying the owner.
After this exchange, each process has received the coordinates
of its owned nodes and remotely accessed nodes,
along with the information to be stored in the scattermap.

% \subsection{Traversal-based mesh-free operations}

% \textcolor{red}{I will wrap this up tomorrow. Are you using assembly or matrix-free ?}
% We are using matrix based version.

\subsection{Traversal-based mesh-free operations}
\label{subsec:traversal}
\dendroKT\ supports both matrix and matrix-free computations. Since, we do not store any map data-structure, such an operation is achieved by performing top-down and bottom-up traversals of the tree. We briefly describe the key ideas of the approach here for performing matrix and vector assembly, as well as the inter-grid transfers needed AMR. Additional details on mesh-free traversals can be found in ~\citet{ishii2019solving}.

% \paragraph{Matrix free:\;}
% The global matrix is defined as a summation of local elemental matrices,
% where the summation is due to common nodal points being shared by neighbouring elements.
% We are able to apply the global operator to a grid vector without explicitly assembling the global matrix.
% Instead, we perform a series of elemental matrix-vector multiplications,
% and use the octree structure to compose the results.

\paragraph{Traversal-based Vector Assembly:}
(see \cref{fig:traversal})
The elemental vector couples elemental nodes in a global input grid vector with equivalent elemental nodes in a global output grid vector.
Within a grid vector, the nodes pertaining to a particular element are generally not stored contiguously.
If one were to read and write to the elemental nodes using an element-to-node map, the memory accesses would require indirection:
$v_{glob}[map[e*npe + i]] += v_{loc}$.
Not only do element-to-node maps cause indirect memory accesses;
the maps become complicated to build if the octree is incomplete due to complex geometry.
We take an alternative approach that obviates the need for element-to-node maps.
Instead, through top-down and bottom-up traversals of the $kD$-tree,
we ensure that elemental nodes are stored contiguously in a leaf,
and there apply the elemental operations. 

%%%\textbf{top-down}
The idea of the top-down phase is to selectively copy nodes
from coarser to finer levels until the leaf level,
wherein the selected nodes are exactly the elemental nodes.
Starting at the root of the tree, we have all the nodes in the grid vector.
We create buckets for all child subtrees.
Looping through the nodes, a node is copied into a bucket if the node is incident on the child subtree corresponding to that bucket.
A node that is incident on multiple child subtrees will be duplicated.
By recursing on each child subtree and its corresponding bucket of incident nodes, we eventually reach the leaf level.

%%%\textbf{leaf}
Once the traversal reaches a leaf octant, the elemental nodes have been copied into a contiguous array.
The elemental vector is computed directly, without the use of an element-to-node map.
The result is stored in a contiguous output buffer the same size as the local elemental input vector.

%%%\textbf{bottom-up}
After all child subtrees have been traversed, the bottom-up phase returns results from a finer to a coarser level.
The parent subtree nodes are once again bucketed to child subtrees,
but instead of the parent values being copied, the values of nodes from each child are accumulated into a parent output array.
That is, for any node that is incident on multiple child subtrees, the values from all node instances are summed to a single value.
The global vector is assembled after the bottom-up phase executes at the root of the octree.

Distributed memory is supported by two slight augmentations.
Firstly, the top-down and bottom-up traversals operate on ghosted vectors.
Therefore ghost exchanges are required before and after each local traversal.
Secondly, the traversals are restricted to subtrees containing the owned octants.
The list of owned octants is bucketed top-down, in conjunction with the bucketing of nodal points.
A child subtree is traversed recursively only if one or more owned octants are bucketed to it. Note that because the traversal path is restricted by a list of existing octants,
the traversal-based operation gracefully handles incomplete octrees without special treatment.

 \begin{figure}[t!]
 \centering
  \begin{tikzpicture}[scale=0.2,every node/.style={scale=0.6} ]
	\tikzstyle{edge from parent}=[black,->,shorten <=1pt,>=stealth',semithick,draw]
	\tikzstyle{level 1}=[sibling distance=8mm]
	\tikzstyle{level 2}=[sibling distance=6mm]
	\tikzstyle{level 3}=[sibling distance=4mm]
		
		\begin{scope}[shift={(0,0)}]
		\draw[step=10] (0,0) grid +(10,10);
		%\draw[step=2.5] (5,0) grid +(5,5);
		%\draw[step=2.5] (0,5) grid +(5,5);
		\def \r{0.12}
		\foreach \x in {0,2.5,5}{
			\foreach \y in {0,2.5,5}{
				\draw[red,fill=red] (\x,\y) circle (\r);
			}
		}
		\foreach \x in {5,7.5,10}{
			\foreach \y in {5,7.5,10}{
				\draw[red,fill=red] (\x,\y) circle (\r);
			}
		}
		\foreach \x in {0,1.25,2.5,3.75}{
			\foreach \y in {6.25,7.5,8.75,10}{
				\draw[blue,fill=blue] (\x,\y) circle (\r);
			}
		}	
		\foreach \x in {6.25,7.5,8.75,10}{
			\foreach \y in {0,1.25,2.5,3.75}{
				\draw[blue,fill=blue] (\x,\y) circle (\r);
			}
		}
		\end{scope}
		
		\begin{scope}[shift={(15,0)}]
		\draw[step=5] (0,0) grid +(10,10);
		%\draw[step=2.5] (5,0) grid +(5,5);
		%\draw[step=2.5] (0,5) grid +(5,5);
		\def \r{0.1}
		\foreach \x in {0.4,2.5,4.6,5.4}{
			\foreach \y in {0.4,2.5,4.6,5.4}{
				\draw[red,fill=red] (\x,\y) circle (\r);
			}
		}
		\foreach \x in {4.6,5.4,7.5,9.6}{
			\foreach \y in {4.6,5.4,7.5,9.6}{
				\draw[red,fill=red] (\x,\y) circle (\r);
			}
		}
		\foreach \x in {0,1.25,2.5,3.75}{
			\foreach \y in {6.25,7.5,8.75,10}{
				\draw[blue,fill=blue] (\x,\y) circle (\r);
			}
		}	
		\foreach \x in {6.25,7.5,8.75,10}{
			\foreach \y in {0,1.25,2.5,3.75}{
				\draw[blue,fill=blue] (\x,\y) circle (\r);
			}
		}
		\end{scope}
		
		\begin{scope}[shift={(30,0)}]
		\draw[step=5] (0,0) grid +(10,10);
	    \draw[step=2.5] (5,0) grid +(5,5);
		\draw[step=2.5] (0,5) grid +(5,5);
		\def \r{0.1}
		
		\foreach \x in {0.4,2.5,4.6,5.4}{
			\foreach \y in {0.4,2.5,4.6,5.4}{
				\draw[red,fill=red] (\x,\y) circle (\r);
			}
		}
		\foreach \x in {4.6,5.4,7.5,9.6}{
			\foreach \y in {4.6,5.4,7.5,9.6}{
				\draw[red,fill=red] (\x,\y) circle (\r);
			}
		}
		
		\foreach \x in {0.4,1.25,2.1,2.9,3.35,4.15}{
			\foreach \y in {6.25,7.1,7.9,8.75,9.6}{
				\draw[blue,fill=blue] (\x,\y) circle (\r);
			}
		}	
		\foreach \x in {6.25,7.1,7.9,8.75,9.6}{
			\foreach \y in {0.4,1.25,2.1,2.9,3.75}{
				\draw[blue,fill=blue] (\x,\y) circle (\r);
			}
		}
		
		\end{scope}
		\draw[->] (11,7.5) -- node[above] {top} node[below] {down} (14,7.5);
		\draw[->] (14,2.5) -- node[above] {bottom} node[below] {up} (11,2.5);	
		
		\draw[->] (26,7.5) -- node[above] {top} node[below] {down} (29,7.5);
		\draw[->] (29,2.5) -- node[above] {bottom} node[below] {up} (26,2.5);	
		
		\end{tikzpicture} 
		%\vfill
		
		\begin{tikzpicture}[scale=1.2, level distance=8mm,emph/.style={edge from parent/.style={red,->,shorten <=1pt,>=stealth',very thick,draw}},norm/.style={edge from parent/.style={black,->,shorten <=1pt,>=stealth',semithick,draw}}]
    	\tikzstyle{edge from parent}=[black,->,shorten <=1pt,>=stealth',semithick,draw]
    	\tikzstyle{level 1}=[sibling distance=8mm]
    	\tikzstyle{level 2}=[sibling distance=6mm]
    	\tikzstyle{level 3}=[sibling distance=4mm]
    	
    	\node[fill=red!30,rounded corners] {root}
    	child[emph] { node[fill=blue!30,rounded corners] {nl}
    		child[norm] { node[fill=green!30,rounded corners] {a}}
    		child[emph] { node[fill=green!30,rounded corners] {b}
    			%child[norm] { node[fill=green!30,rounded corners] {b}}
    			%child[norm] { node[fill=green!30,rounded corners] {c}}
    			%child[emph] { node[draw=red,very thick,fill=green!30,rounded corners] {d}}
    			%child[norm] { node[fill=green!30,rounded corners] {e}}
    		}
    		child[norm] { node[fill=green!30,rounded corners] {f}}
    		child[norm] { node[fill=green!30,rounded corners] {g}}
    	}
    	child { node[fill=green!30,rounded corners] {h}}
    	child { node[fill=green!30,rounded corners] {i}}
    	child { node[fill=blue!30,rounded corners] {nl}
    		child { node[fill=green!30,rounded corners] {j}}
    		child { node[fill=green!30,rounded corners] {k}}
    		child { node[fill=green!30,rounded corners] {l}}
    		child { node[fill=green!30,rounded corners] {m}}
    	};
    	
    	%% Draw levels axis ...
    	\draw[->] (-3,0) -- (-3,-2);		
    	\draw[snake=ticks,segment length=0.95cm] (-3,0) -- (-3,-2);
    	
    	\draw (-2.8,0) node {$0$}
    	(-2.8,-0.8) node {$1$}
    	(-2.8,-1.6) node {$2$};
    	%(-2.8,-2.4) node {$3$};
    	\draw (-3.3,-1.6) node [rotate=90] {$level$};		
    	%\draw (3,-2.9) node {$\mathcal{R}$};
	
	    \end{tikzpicture}
		
	\caption{Illustration of top-down \& bottom-up tree traversals for a $2D$ tree with quadratic element order. The leftmost figure depicts the unique shared nodes (nodes are color-coded based on level), as we perform top-down traversal nodes shared across children of the parent get duplicated for each bucket recursively, once leaf node is reached it might be missing elemental local nodes, which can be interpolated from immediate parent (see the rightmost figure). After elemental local node computations, bottom-up traversal performed while merging the nodes duplicated in the top-down traversal.
	    \label{fig:traversal}
	}
\end{figure}
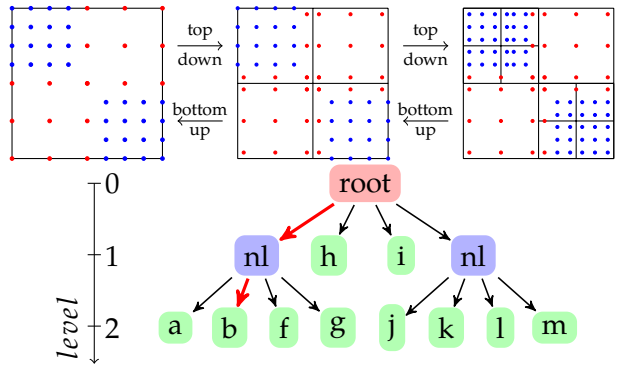
% \paragraph{Extension to incomplete trees}
% All the child subtrees in a complete octree overlap automatically with the domain,
% and therefore all child subtrees would be traversed recursively.
% In \citep{ishii2019solving}, this assumption allowed the authors to
% discard the list of octants underpinning the FEM grid,
% and still perform traversal-based {\mvec}s on the set of nodal points.
% For incomplete trees, however, the explicit octant list must be retained.
% The octant list is bucketed top-down, in conjunction with the bucketing of nodal points.
% A child subtree is traversed recursively only if one or more subdomain octants are bucketed to it.

\paragraph{Traversal-based Matrix Assembly:} 
To implement matrix assembly, we have leveraged PETSc interface~\citep{petsc-web-page,petsc-efficient},
that only requires a sequence of entries
$(\text{id}_\text{row}, \text{id}_\text{col}, \text{val})$,
and can be configured to add entries with duplicate indices~\citep{petsc-user-ref}. Note that any other distributed sparse-matrix library can be supported in a similar fashion.

The remaining task is to associate the correct global node indices
with the rows and columns of every elemental matrix.
We use an octree traversal to accomplish this task.
Similar to the traversal-based vector assembly,
nodes are selectively copied from coarser to finer levels, recursively,
until reaching the leaf, wherein the elemental nodes are contiguous.
Note that integer node ids are copied instead of floating-point values from a grid vector.
At the leaf, an entry of the matrix is emitted
for every row and column of the elemental matrix,
using the global row and column indices instead of the elemental ones.
No bottom-up phase is required for assembly,
as PETSc handles the merging of multi-instanced entries.

\paragraph{Traversal-based Intergrid Transfers}

An important aspect of AMR is the ability to effectively and efficiently transfer information from the current mesh to the newly refined/coarsened mesh---the intergrid transfer. \dendroKT\ supports efficient traversal-based intergrid transfers, and regions can be locally refined or coarsened in different regions of the mesh. For efficiency, we limit the local refinement or coarsening to only be by one level\footnote{Additional levels of coarsening or refinement will need multiple AMR steps, and are supported.}. In order to transfer information from mesh $A$ (current) to mesh $B$ (new), we simultaneously traverse both meshes in a synchronized fashion, i.e., we keep track of the points/data within a bucket for both meshes. Since $A$ and $B$ are (locally) off by one level, when we reach the leaf level on either $A$ or $B$, three cases are possible,
\begin{itemize}
    \item The buckets of $A$ and $B$ are at the same level and are leaves:  we simply copy the finite-element values associated with this bucket (element),
    \item The bucket for $A$ is at the leaf level, but the bucket for $B$ is not: in this case, $B$ is locally refined and we need to interpolate values from the element of $A$ to its children.  The elemental
interpolation is done by evaluating the local shape functions at the locations of the child nodes. 
    \item The bucket for $B$ is at the leaf level, but the bucket for $A$ is not: in this case, $B$ is locally coarsened and we need to project values from the elements of $A$ to its parent. 
\end{itemize}

Once the leaf-level transfer is done, the upward traversal is the same as for  vector assembly.

\section{Numerical experiments}
\label{sec:num_exp}

%\subsection{Convergence via the method of manufactured solutions}
%We performed convergence analysis via the method manufactured solutions for a case with quad elements in 2D.  Because these results are not in 3D, we present them in \ref{sec:manufac_sol}.

\subsection{2D simulations: manufactured solutions}
\label{subsec:manfactured_soln_result}

We use the method of manufactured solutions to assess the convergence properties of our method.  We select an input ``solution'' which is solenoidal, and substitute it in the full set of governing equations. We then use the residual as a body force on the right-hand side of~\crefrange{eqn:nav_stokes_var_semi_disc}{eqn:phi_eqn_var_semi_disc}.
%The idea of this approach  is to input a ``solution'' that satisfies solenoidality, but not necessarily the full set of evolution equations. 
%Instead, the residual from plugging this ``solution'' into the full Cahn-Hilliard Navier-Stokes system becomes a forcing term on the right-hand side of \cref{eqn:nav_stokes_var_semi_disc} -- \cref{eqn:phi_eqn_var_semi_disc}.
We choose the following ``solution'' with appropriate body forcing terms: 
\begin{equation}
\begin{split}
\vec{v} &= \left( \pi \sin^2(\pi x_1)\sin(2 \pi x_2)\sin(t), \, -\pi\sin(2\pi x_1)\sin^2(\pi x_2)\sin(t), \, 0 \right), \\
p &= \cos(\pi x_1)\sin(\pi x_2)\sin(t), \quad 
\phi = \mu = \cos(\pi x_1)\cos(\pi x_2)\sin(t).
% \mu &= \cos(\pi x_1)\cos(\pi x_2)\sin(t).
\end{split}
\label{eq:manufac_exact}
\end{equation}
Our numerical experiments use the following non-dimensional parameters: $Re = 10$, $We = 1$, $Cn = 1.0$, $Pe = 3.0$, and $Fr = 1.0$. The density ratio is set to $\rho_{-}/\rho_{+} = 0.85$. 

For the first experiment we use a 2D uniform mesh with $512 \times 512$ bilinear elements (quads).  Panel~(a) of~\cref{fig:manufac_temporal_convergence} shows the temporal convergence of the $L^2$ errors (numerical solution compared with the manufactured solution) calculated at $t = \pi$ to allow for one complete time period. For panel (a) we show the convergence for the forcing term of $\frac{Cn}{We} \pd{}{x_j}\left({\pd{\phi}{x_i}\pd{\phi}{x_j}}\right)$. The figure shows the evolution of the error versus time-step on a log-log scale. The errors are decreasing with a slope of two for the phase-field function $\phi$, thereby demonstrating second-order convergence. For velocity the slope on the log-log scale is initially about 2.0, but then
tapers off for smaller time-steps; we expect this tapering off at smaller time-steps due to the fact that we used a fixed mesh and at some point the spatial errors dominate. As the scheme only uses pressure at previous timestep, the expected order for pressure here is 1 and in panel (a) we see the expected slope of 1 for pressure.  In panel (b) we repeat the temporal convergence study in panel (a) with a forcing of $\frac{1}{CnWe} \phi\pd{\mu}{x_i}$.  We observe that the velocity errors saturate at much larger timesteps compared to panel (a) and then they never decrease.  Recall \cref{rem:consv_forcing} where we show that the forcing used in panel (a) and panel (b) are not equivalent unless the basis functions are divergence which is not the case for our conforming Galerkin finite element method.  We use $\frac{Cn}{We} \pd{}{x_j}\left({\pd{\phi}{x_i}\pd{\phi}{x_j}}\right)$ forcing for all the other cases in this work as it shows correct 2nd order convergence for velocity. %This can also be seen in \cref{tab:temporal_error_table} which shows the errors and rate of change of errors for varying time step for velocity and $\phi$.

We next conduct a spatial convergence study. We fix the time step at $\delta t = 5\times10^{-4}$, and vary the spatial mesh resolution. Panel~(c) of~\cref{fig:manufac_temporal_convergence} shows the spatial convergence of $L^2$ errors (numerical solution compared with the manufactured solution) at $t = \pi$. %\Cref{tab:spatial_error_table} shows the errors and the rate of change of errors for varying spatial grid spacing for velocity and $\phi$; 
We observe second order convergence for velocity, $\phi$ and pressure as expected for linear basis functions.
%We observe second order convergence for $\phi$, and a slope of around 2.0 for velocity.

Panel~(d) of~\cref{fig:manufac_temporal_convergence} shows mass conservation for an intermediate resolution simulation with $\delta t=5\times10^{-4}$ and $256 \times 256$ elements. We plot mass drift:
\begin{equation}
\int_{\Omega} \phi\left(\vec{x},t\right) \, d\vec{x} 
- \int_{\Omega}\phi\left(\vec{x},t=0\right) \, d\vec{x},
\end{equation}
and expect this value to be close to zero as per the theoretical prediction of~\cref{prop:mass_conservation}. We observe excellent mass conservation with fluctuations of the order of $10^{-7}$, which is to be
expected in double precision arithmetic. We detail the command-line arguments along with tolerances for iterative solvers used for the Algebraic Multigrid Method from~\petsc~in~\ref{subsec:app_manufac_sol}.
%Here we used a relative tolerance of $10^{-7}$ for the Newton iteration of Cahn-Hilliard solves.  For the linear solves of Navier-Stokes, pressure Poisson, and velocity update we used a relative tolerance of $10^{-7}$.  
 
\begin{figure}
	\centering
	\begin{tikzpicture}
	\begin{loglogaxis}[width=0.4\linewidth, scaled y ticks=true,xlabel={timestep},ylabel={$\norm{f - f_{exact}}_{L^2}$},legend entries={$v_1$,$v_2$,$p$,$\phi$, $\mu$, $slope = 1.0$, $slope = 2$},
	legend style={at={(0.5,-0.25)},anchor=north, nodes={scale=0.65, transform shape}}, 
	%legend pos=outer south, 
	legend columns=2,
	title={\footnotesize(a)$h = 1/2^9$, lin-2-blocks($\rho_H/\rho_L = 1.25$)~ $\frac{Cn}{We} \pd{}{x_j}\left({\pd{\phi}{x_i}\pd{\phi}{x_j}}\right)$},
	ymin=1e-8,ymax=2e-0,
	xmin=2.5e-3,xmax=1e-0,
	%cycle list name=exotic
	% initialize
	cycle list/Set1,
	% combine it with 'mark list*':
	cycle multiindex* list={
		mark list*\nextlist
		Set1\nextlist
	},
	]
	\addplot table [x={timestep},y={L2U},col sep=comma] {Figures/convergence_mms/linearNS_projection_nonLinearCH_customTheta_kt/temporal_theta_linear2blocks_level9_dphidphi_unequalDensity.csv};
	\addplot table [x={timestep},y={L2V},col sep=comma] {Figures/convergence_mms/linearNS_projection_nonLinearCH_customTheta_kt/temporal_theta_linear2blocks_level9_dphidphi_unequalDensity.csv};
	\addplot table [x={timestep},y={L2Press},col sep=comma] {Figures/convergence_mms/linearNS_projection_nonLinearCH_customTheta_kt/temporal_theta_linear2blocks_level9_dphidphi_unequalDensity.csv};
	\addplot +[mark=triangle*] table [x={timestep},y={L2Phi},col sep=comma] {Figures/convergence_mms/linearNS_projection_nonLinearCH_customTheta_kt/temporal_theta_linear2blocks_level9_dphidphi_unequalDensity.csv};
	\addplot +[mark=diamond*] table [x={timestep},y={L2Mu},col sep=comma] {Figures/convergence_mms/linearNS_projection_nonLinearCH_customTheta_kt/temporal_theta_linear2blocks_level9_dphidphi_unequalDensity.csv};
	%slope = 1
	\addplot +[mark=none, red, dashed] [domain=0.001:1]{1*x^1.0};
	%slope = 2
	\addplot +[mark=none, blue, dashed] [domain=0.001:1]{0.001*x^2};
	\addplot +[mark=none, blue, dashed] [domain=0.001:1]{0.075*x^2};
	\addplot +[mark=none, blue, dashed] [domain=0.001:1]{2.5*x^2};
	\end{loglogaxis}
	\end{tikzpicture}
	\begin{tikzpicture}
	\begin{loglogaxis}[width=0.4\linewidth, scaled y ticks=true,xlabel={timestep},ylabel={$\norm{f - f_{exact}}_{L^2}$},legend entries={$v_1$,$v_2$,$p$,$\phi$, $\mu$, $slope = 1.0$, $slope = 2$},
	legend style={at={(0.5,-0.25)},anchor=north, nodes={scale=0.65, transform shape}}, 
	%legend pos=outer south, 
	legend columns=2,
	title={\footnotesize(b)$h = 1/2^9$, lin-2-blocks($\rho_H/\rho_L = 1.25$)~ $\frac{1}{CnWe} \phi\pd{\mu}{x_j}$},
	ymin=1e-8,ymax=2e-0,
	xmin=1e-2,xmax=1e-0,
	%cycle list name=exotic
	% initialize
	cycle list/Set1,
	% combine it with 'mark list*':
	cycle multiindex* list={
		mark list*\nextlist
		Set1\nextlist
	},
	]
	\addplot table [x={timestep},y={L2U},col sep=comma] {Figures/convergence_mms/linearNS_projection_nonLinearCH_customTheta_kt/temporal_theta_linear2blocks_level9_phidmu_unequalDensity.csv};
	\addplot table [x={timestep},y={L2V},col sep=comma] {Figures/convergence_mms/linearNS_projection_nonLinearCH_customTheta_kt/temporal_theta_linear2blocks_level9_phidmu_unequalDensity.csv};
	\addplot table [x={timestep},y={L2Press},col sep=comma] {Figures/convergence_mms/linearNS_projection_nonLinearCH_customTheta_kt/temporal_theta_linear2blocks_level9_phidmu_unequalDensity.csv};
	\addplot +[mark=triangle*] table [x={timestep},y={L2Phi},col sep=comma] {Figures/convergence_mms/linearNS_projection_nonLinearCH_customTheta_kt/temporal_theta_linear2blocks_level9_phidmu_unequalDensity.csv};
	\addplot +[mark=diamond*] table [x={timestep},y={L2Mu},col sep=comma] {Figures/convergence_mms/linearNS_projection_nonLinearCH_customTheta_kt/temporal_theta_linear2blocks_level9_phidmu_unequalDensity.csv};
	%slope = 1
	\addplot +[mark=none, red, dashed] [domain=0.001:1]{1*x^1.0};
	%slope = 2
	\addplot +[mark=none, blue, dashed] [domain=0.001:1]{0.001*x^2};
	\addplot +[mark=none, blue, dashed] [domain=0.001:1]{0.075*x^2};
	\addplot +[mark=none, blue, dashed] [domain=0.001:1]{2.5*x^2};
	\end{loglogaxis}
	\end{tikzpicture}
	
	\begin{tikzpicture}
	\begin{loglogaxis}[width=0.4\linewidth, scaled y ticks=true,xlabel={Element size, $h$ (-)},ylabel={$\norm{f - f_{exact}}_{L^2(\Omega)}$},
	legend entries={$v_1$,$v_2$,$p$,$\phi$, $\mu$, $slope = 2$},
	legend style={at={(0.5,-0.25)},anchor=north, nodes={scale=0.65, transform shape}},legend columns=-1,
	xtick = {0.00390625, 0.0078125, 0.015625, 0.03125, 0.0625},
	xticklabel={
		\pgfkeys{/pgf/fpu=true}
		\pgfmathparse{exp(\tick)}%
		\pgfmathprintnumber[fixed relative, precision=2]{\pgfmathresult}
		\pgfkeys{/pgf/fpu=false}
	},
	scaled x ticks=true, 
	legend columns=2,
	title={\footnotesize(c) $\Delta t = 0.0005$, lin-2-blocks($\rho_H/\rho_L = 1.25$)~ $\frac{Cn}{We} \pd{}{x_j}\left({\pd{\phi}{x_i}\pd{\phi}{x_j}}\right)$},
	ymin=1e-8,ymax=2e-0,
	xmin=3.5e-3,xmax=7e-2,
	%cycle list name=exotic
	% initialize
	cycle list/Set1,
	% combine it with 'mark list*':
	cycle multiindex* list={
		mark list*\nextlist
		Set1\nextlist
	},
	]
	\addplot table [x={h},y={L2U},col sep=comma] {Figures/convergence_mms/linearNS_projection_nonLinearCH_customTheta_kt/spatial_theta_linear2blocks_ts0dot0005_dphidphi_unequalDensity.csv};
	\addplot table [x={h},y={L2V},col sep=comma] {Figures/convergence_mms/linearNS_projection_nonLinearCH_customTheta_kt/spatial_theta_linear2blocks_ts0dot0005_dphidphi_unequalDensity.csv};
	\addplot table [x={h},y={L2Press},col sep=comma] {Figures/convergence_mms/linearNS_projection_nonLinearCH_customTheta_kt/spatial_theta_linear2blocks_ts0dot0005_dphidphi_unequalDensity.csv};
	\addplot +[mark=triangle*] table [x={h},y={L2Phi},col sep=comma] {Figures/convergence_mms/linearNS_projection_nonLinearCH_customTheta_kt/spatial_theta_linear2blocks_ts0dot0005_dphidphi_unequalDensity.csv};
	\addplot +[mark=diamond*] table [x={h},y={L2Mu},col sep=comma] {Figures/convergence_mms/linearNS_projection_nonLinearCH_customTheta_kt/spatial_theta_linear2blocks_ts0dot0005_dphidphi_unequalDensity.csv};
	%slope = 2
	\addplot +[mark=none, blue, dashed] [domain=0.001:1]{0.001*x^2};
	\addplot +[mark=none, blue, dashed] [domain=0.001:1]{0.075*x^2};
	\addplot +[mark=none, blue, dashed] [domain=0.001:1]{2.5*x^2};
	\end{loglogaxis}
	\end{tikzpicture}
	\begin{tikzpicture}
	\begin{axis}[width=0.45\linewidth,scaled y ticks=true, 
	xlabel={Time (-)},ylabel={$\int_{\Omega} \phi (x_i)\mathrm{d}x_i - \int_{\Omega} \phi_{0} (x_i)\mathrm{d}x_i$},legend style={nodes={scale=0.65, transform shape}}, 
	ymin=-5e-7,ymax=5e-7,%1.0001, 
	xmin=0, xmax=3.2, 
	, title={(d) Mass conservation}]
	\addplot[line width=0.35mm, color=blue]table [x={time},y={TotalPhiMinusInitial},col sep=comma] {Figures/convergence_mms/linearNS_projection_nonLinearCH_customTheta_kt/Energy_data_dt0d0005_level8.csv};
	\end{axis}
	\end{tikzpicture}	
	
	\caption{\textit{Manufactured Solution Examples-chns-block-iteration-kt:} (a) Temporal convergence of the numerical scheme for the case of manufactured solutions for conservative forcing; (b) Temporal convergence of the numerical scheme for the case of manufactured solutions for non-conservative forcing; (c) Spatial convergence of the numerical scheme for the case of manufactured solutions for conservative forcing; (d) the mass conservation for the case of manufactured solutions using $256 \times 256$ elements with time step of $5\times10^{-4}$. Here $f$ denotes variables of interest which are velocities ($v_1$ and $v_2$), pressure $p$, phase field $\phi$, and chemical potential $\mu$. } 
	\label{fig:manufac_temporal_convergence}
\end{figure}
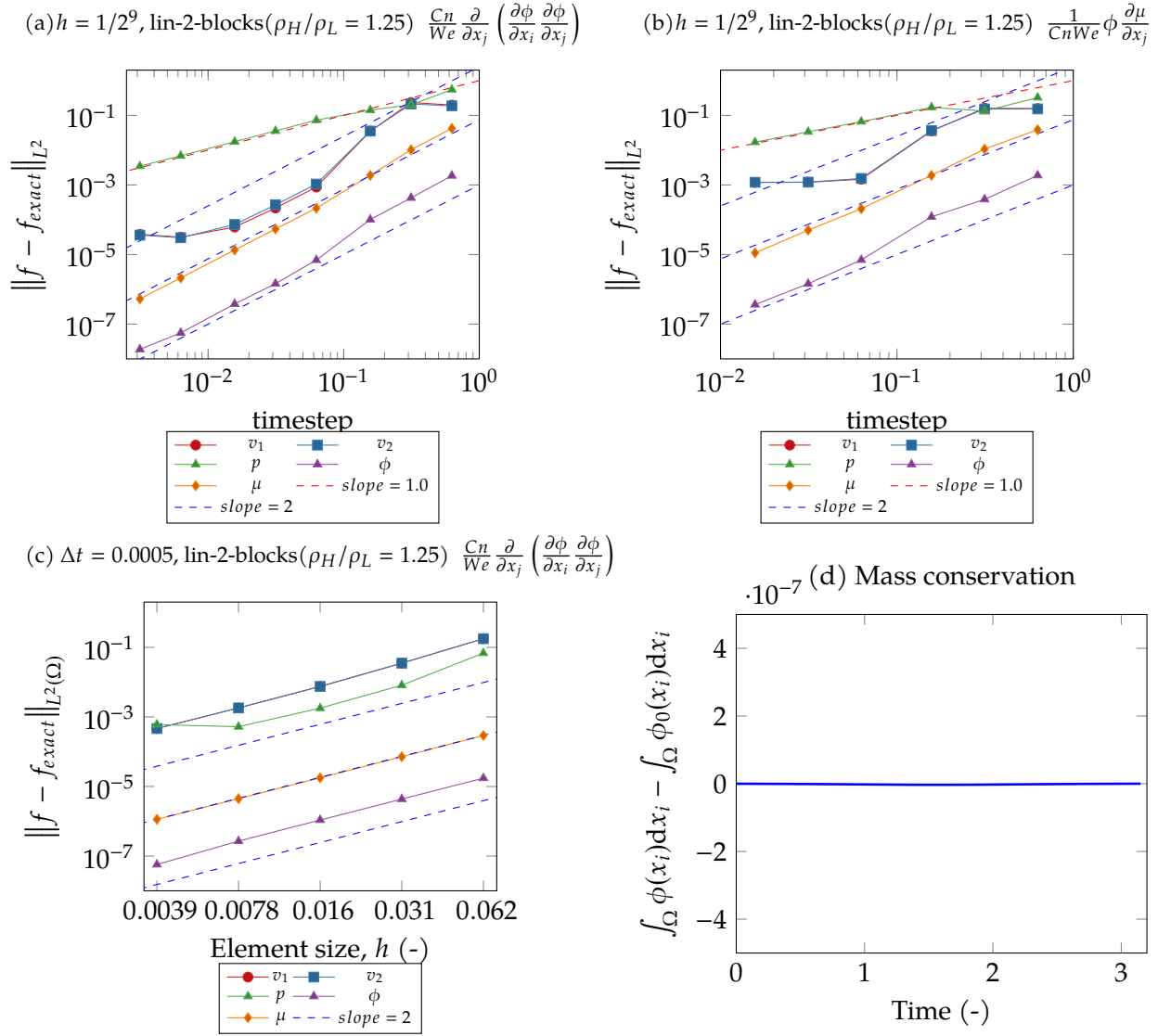

\subsection{2D simulations: single rising bubble}
\label{subsec:single_rising_drop_2D}

To validate the framework, we consider benchmark cases for a single rising bubble in a quiescent water channel \citep{ Hysing2009, Aland2012, Yuan2017}.  
%We next illustrate the framework using a canonical case of a single air bubble rising in a quiescent channel of water. 
%There have been several benchmark studies published~\citep{ Hysing2009, Aland2012, Yuan2017} for this case.  
%We start with selecting appropriate scales to non-dimensionalize the problem.  
We set the Froude number ($Fr = u^2/(gD)$) to 1, which fixes the non-dimensional velocity scale to $u = \sqrt{gD}$, where $g$ is the gravitational acceleration, and $D$ is the diameter of the bubble.  This scaling gives a Reynolds number of $\rho_c g^{1/2}D^{3/2}/\mu_c$, where $\rho_c$ and $\mu_c$ are the specific density and specific viscosity of the continuous fluid, respectively. 
The Archimedes number,  $Ar = \rho_c g^{1/2}D^{3/2}/\mu_c$, 
scales the diffusion term in the momentum equation. 
The Weber number here becomes $We = \rho_c g D^2/\sigma$.  We use the density of the continuous fluid to non-dimensionalize: $\rho_{+} = 1$.  The density and viscosity ratios are $\rho_{+}/\rho_{-}$ and $\nu_{+}/\nu_{-}$, respectively. We present results for two standard benchmark cases.  

\Cref{tab:physParam_bubble_rise_2D_benchmarks} shows the parameters and the corresponding non-dimensional numbers for the two cases simulated in this work.  The bubble is centered at $(1,1)$, and since our scaling length scale is the bubble diameter, the bubble diameter for our simulations is 1.  The domain is %selected to be 
$[0,2]\times[0,4]$.  %This choice is determined by the domain selected in the benchmark studies.  
Following the benchmark studies in the literature, we choose the top and bottom wall to have no slip boundary conditions and the side walls to have boundary conditions: $v_1=0$ ($x$-velocity) and $\frac{\partial v_2}{\partial x}=0$ ($y$-velocity). 
%We use the biCGstab (bcgs) linear solver from the PETSc suite along with the Additive Schwarz (ASM) preconditioner for the linear solves in the Newton iterations (see~\cref{subsec:newton_iter}). 
We detail the command-line arguments along with tolerances for iterative solvers used for the Algebraic Multigrid Method from~\petsc~in~\ref{subsec:app_bubble_rise}. 
We use a time step of $\num{2.5e-3}$ for test case 1 and $\num{1e-3}$ for test case 2. It is important to note that this is a substantially higher time step compared to fully linear block schemes presented in~\citep{Shen2010a,Shen2010b,Shen2015,Chen2016,Zhu2019} and on par with fully implicit schemes such as~\citep{Guo2017,Khanwale2020,Khanwale2021} even though the velocity prediction, pressure Poisson, and velocity update steps are linear.  This is the advantage of using a non-linear time-discretization for Cahn-Hilliard (CH).  Our scheme is not limited by interfacial relaxation timescales. Interestingly, in practice, the non-linear iteration converges in a single Newton iteration, thus becoming computationally equivalent to a linear scheme. Finally, we notice that this scheme allows for a larger range of timesteps.  %The convergence criterion for both test cases uses a relative tolerance of $10^{-6}$ for Newton iteration and a relative tolerance of $10^{-7}$ for the linear solves within each Newton iteration.  

\begin{table}[H]
	\centering\normalsize\setlength\tabcolsep{5pt}
	\begin{tabular}{@{}|c|c|c|c|c|c|c|c|c|c|c|c|@{}}
		\toprule
		Test Case  & $\rho_{c}$  & $\rho_{b}$  & $\mu_{c}$  & $\mu_{b}$ & $\rho_{+}/\rho_{-}$ & $\nu_{+}/\nu_{-}$  & $g$ & $\sigma$ & $Ar$ & $We$ & $Fr$ \\
		\midrule
		\midrule
		{$1$}  & {1000}  & {100}  & {10}  &{1.0} & {10}     &{10} & {0.98}    & {24.5}  &  {35}     & {10} & {1.0}    \\
		{$2$}  & {1000}  & {1.0}  & {10}  &{0.1} & {1000}     &{100} & {0.98}    & {1.96}  &  {35}     & {125} & {1.0}    \\
		\bottomrule
	\end{tabular}
	\caption{Physical parameters and corresponding non-dimensional numbers for the 2D single rising drop  benchmarks considered
	in \cref{subsec:single_rising_drop_2D}.}
	\label{tab:physParam_bubble_rise_2D_benchmarks}                            
	%\end{adjustwidth}
\end{table}

\subsubsection{Test case 1}
\label{subsubsec:single_rising_drop_2D_t1}
This test case considers the effect of higher surface tension, and consequently less deformation of the bubble as it rises. %A higher surface tension corresponds to a lower Weber (We) number. 
We compare the bubble shape in~\cref{fig:test_case1} with benchmark quantities presented in three previous studies~\citep{ Hysing2009, Aland2012, Yuan2017}.  We take $Cn=\num{1e-2}$ for this case.  Panel~(a) of~\cref{fig:test_case1} shows a shape comparison against benchmark studies in the literature for the case with Adaptive Mesh Refinement (AMR) vs Uniform Mesh. We see an excellent agreement in the shape of the bubble.  Panel~(b) of~\cref{fig:test_case1} shows a comparison of centroid locations with respect to time against benchmark studies in the literature; again, we see an excellent agreement for both AMR and Uniform Mesh. We can see from the magnified inset besides of~panel (a) of \cref{fig:test_case1} that both the cases with AMR and uniform meshes show excellent agreement with benchmark studies.  %We see an almost exact overlap between the benchmark and cases with $h = 2/400$ and $h = 2/600$, where $h$ is the size of the element, demonstrating spatial convergence. 

We next check whether the numerical method follows the theoretical energy stability~\footnote{The energy function in~\cref{eqn:energy_functional} should decrease as a function of time}.  
We present the evolution of the energy functional defined in \cref{eqn:energy_functional} for test case 1. 
Panel~(c) of~\cref{fig:test_case1}~shows that the energy is decreasing in accordance with the energy stability condition for both the cases of AMR and uniform meshes. The energy profile using the AMR mesh is nearly identical to the uniform mesh, suggesting that adaptive meshes can capture the energy dissipation correctly. 

Finally, we check the mass conservation. Panel~(d) shows the total mass of the system minus the initial mass. For the uniform mesh the change in the total mass is of the order of $10^{-8}$, even after 1600 time steps.  As expected, we see more drift for the AMR mesh compared to the uniform mesh.  This is because the interpolation strategy used in adaptive mesh refinement to interpolate fields between adapted meshes is \textit{not mass conserving}. Mass conserving interpolations for conforming Galerkin methods is an open question and out of the scope of the current work. Excellent mass conservation with the uniform mesh show that the numerical method delivers excellent mass conservation for long time horizons.  

It is useful to compare the computational cost of using an AMR mesh versus using a uniform fine mesh. The uniform mesh has $\sim 2$ million elements.% and takes 6 hours on 2 node on TACC~\Frontera~ i.e 12 node hours.
The AMR mesh on the other hand has a mesh count of about $54,000$ elements.% and takes 2 hours with 2 nodes on TACC~\Frontera~ i.e 4 node hours.  We get a speed up of 3 times with the AMR mesh.  
It is important to note that for this test case the $We$ number is small and the shape of the bubble does not change much therefore, one does not require very small $Cn$ to resolve the interfacial dynamics.  Thus, in this case the resolution for interface is not very high and we do not anticipate a large efficiency gain.

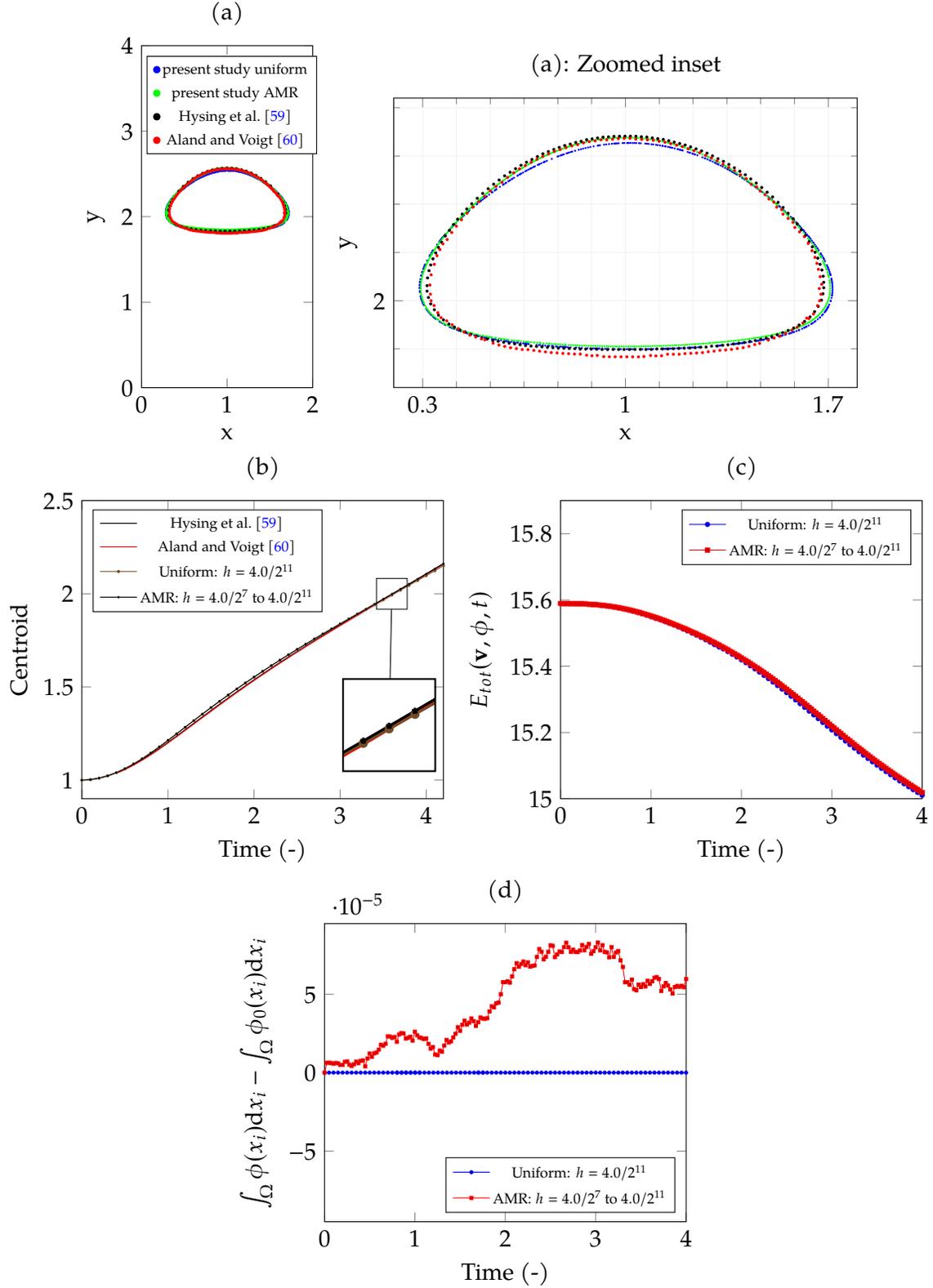
\begin{figure}[H]
	\centering
	\tikzexternaldisable
	\begin{tikzpicture}
	\begin{axis}[width=0.5\linewidth,scaled y ticks=true,xlabel={$\mathrm{x}$},ylabel={$\mathrm{y}$},legend style={nodes={scale=0.65, transform shape}}, ymin=0.0, ymax=4.0, ytick distance=1.0,  xtick={0.0, 1.0, 2}, title={(a)},
	legend style={nodes={scale=0.95, transform shape}, row sep=2.5pt},
	legend entries={
		present study uniform,
		present study AMR,
		\citet{Hysing2009}, \citet{Aland2012}},
	legend pos= north west,
	legend image post style={scale=1.0},
	unit vector ratio*=1 1 1,
	xmin=0, xmax=2, 
	legend image post style={scale=3.0},
	]
	\addplot [only marks,mark size = 0.5pt,color=blue, each nth point=5, filter discard warning=false, unbounded coords=discard] table [x={x},y={y},col sep=comma] {Figures/Bubble_rise_benchmarks/Re35We10/bubbleShapeRe35We10.csv};
	\addplot [only marks,mark size = 0.5pt,color=green, each nth point=5, filter discard warning=false, unbounded coords=discard] table [x={x},y={y},col sep=comma] {Figures/Bubble_rise_benchmarks/Re35We10/bubbleShapeRe35We10AMR.csv};
	\addplot [only marks,mark size = 0.5pt,color=black,each nth point=3, filter discard warning=false, unbounded coords=discard] table [x={x},y={y},col sep=comma] {Figures/Bubble_rise_benchmarks/Re35We10/Hysing_Re35We10.csv};
	\addplot [only marks,mark size = 0.5pt,color=red,each nth point=1, filter discard warning=false, unbounded coords=discard] table [x={x},y={y},col sep=comma] {Figures/Bubble_rise_benchmarks/Re35We10/Aland_Voight_Re35We10_shape.csv};
	\end{axis}
	\end{tikzpicture}
	\tikzexternalenable
	\begin{tikzpicture}
	\begin{axis}[width=0.55\linewidth,scaled y ticks=true,xlabel={$\mathrm{x}$},ylabel={$\mathrm{y}$},legend style={nodes={scale=0.65, transform shape}}, ymin=1.70, ymax=2.7, ytick distance=1.0,  xtick={0.3, 1.0, 1.7}, title={(a): Zoomed inset},
	legend style={nodes={scale=0.95, transform shape}, row sep=2.5pt},
	legend pos= north west,
	legend image post style={scale=1.0},
	unit vector ratio*=1 1 1,
	xmin=0.2, xmax=1.8, 
	legend image post style={scale=3.0},
	grid=both,
	grid style={line width=.1pt, draw=gray!10},
	minor tick num=5
	]
	\addplot [only marks,mark size = 0.25pt,color=blue, each nth point=4, filter discard warning=false, unbounded coords=discard] table [x={x},y={y},col sep=comma] {Figures/Bubble_rise_benchmarks/Re35We10/bubbleShapeRe35We10.csv};
	\addplot [only marks,mark size = 0.25pt,color=green, each nth point=4, filter discard warning=false, unbounded coords=discard] table [x={x},y={y},col sep=comma] {Figures/Bubble_rise_benchmarks/Re35We10/bubbleShapeRe35We10AMR.csv};
	\addplot [only marks,mark size = 0.5pt,color=black,each nth point=1, filter discard warning=false, unbounded coords=discard] table [x={x},y={y},col sep=comma] {Figures/Bubble_rise_benchmarks/Re35We10/Hysing_Re35We10.csv};
	\addplot [only marks,mark size = 0.5pt,color=red,each nth point=1, filter discard warning=false, unbounded coords=discard] table [x={x},y={y},col sep=comma] {Figures/Bubble_rise_benchmarks/Re35We10/Aland_Voight_Re35We10_shape.csv};
	\end{axis}
	\end{tikzpicture}
	
	\tikzexternaldisable
	\begin{tikzpicture}[spy using outlines={rectangle, magnification=3, size=1.5cm, connect spies}]
	\begin{axis}[width=0.45\linewidth,scaled y ticks=true,xlabel={Time (-)},ylabel={Centroid},legend style={nodes={scale=0.65, transform shape}}, xmin=0, xmax=4.2, ymin=0.9, ymax=2.5, ytick distance=0.5,  xtick={0.0, 1.0, 2, 3, 4}, title={(b)},
	legend style={nodes={scale=0.95, transform shape}, row sep=2.5pt},
	legend entries={\citet{Hysing2009}, 
		\citet{Aland2012}, 
		%\citet{Yuan2017}, 
		Uniform: $h = 4.0/2^{11}$,
		AMR: $h = 4.0/2^{7}$ to $4.0/2^{11}$ 
	},
	legend pos= north west,
	legend image post style={scale=1.0},
	]
	\addplot [line width=0.15mm, color=black] table [x={time},y={centroid},col sep=comma] {Figures/Bubble_rise_benchmarks/Re35We10/Hysing_centroid_Re35We10.csv};
	%%%%%%%%%%%%%%
	\addplot [line width=0.15mm, color=red] table [x={time},y={centroid},col sep=comma] {Figures/Bubble_rise_benchmarks/Re35We10/Aland_Voight_centroid_Re35We10.csv};
	%%%%%%%%%%%%%%
	%\addplot [line width=0.15mm, color=ForestGreen] table [x={time},y={centroid},col sep=comma] {Figures/Bubble_rise_benchmarks/Re35We10/Yuan_centroid_Re35We10.csv};
	%%%%%
	\addplot+[mark size = 0.5pt]table [x={time},y={centroid},col sep=comma, each nth point=1, filter discard warning=false, unbounded coords=discard] {Figures/Bubble_rise_benchmarks/Re35We10/centerOfMass_uniformlevel11.csv};	
	%%%%%
	\addplot+[mark size = 0.5pt]table [x={time},y={centroid},col sep=comma, each nth point=1, filter discard warning=false, unbounded coords=discard] {Figures/Bubble_rise_benchmarks/Re35We10/centerOfMassRe35We10AMR.csv};
	%%%%%
	%\addplot+[mark size = 0.5pt]table [x={time},y={centroid},col sep=comma, each nth point=1, filter discard warning=false, unbounded coords=discard] {Figures/Bubble_rise_benchmarks/Re35We10/centerOfMass_600_1200.csv};
	\coordinate (a) at (axis cs:3.6,2.0);
	\end{axis}
	\spy [black] on (a) in node  at (5,1.2);
	\end{tikzpicture}
	\hskip 5pt
	\begin{tikzpicture}
	\begin{axis}[width=0.45\linewidth,scaled y ticks=true,xlabel={Time (-)},ylabel={$E_{tot}(\vec{v},\phi,t)$},legend style={nodes={scale=0.65, transform shape}}, xmin=0, xmax=4, xtick={0,1,2,3,4},  title={(c)},
	legend style={nodes={scale=0.95, transform shape}, row sep=2.5pt},
	legend entries={
		Uniform: $h = 4.0/2^{11}$,
		AMR: $h = 4.0/2^{7}$ to $4.0/2^{11}$
	},
	legend pos= north east,
	legend image post style={scale=1.0},
	ymin=15,ymax=15.9
	]
	\addplot +[mark size = 1pt, each nth point=20, filter discard warning=false, unbounded coords=discard] table [x={time},y={TotalEnergy},col sep=comma] {Figures/Bubble_rise_benchmarks/Re35We10/Energy_dataRe35We10Uniform.csv};
	\addplot +[mark size = 1pt, each nth point=10, filter discard warning=false, unbounded coords=discard] table [x={time},y={TotalEnergy},col sep=comma] {Figures/Bubble_rise_benchmarks/Re35We10/Energy_dataRe35We10AMR.csv};
	%\addplot +[mark size = 1pt, each nth point=20, filter discard warning=false, unbounded coords=discard] table [x={time},y={TotalEnergy},col sep=comma] {Figures/Bubble_rise_benchmarks/Re35We10/Energy_data_600_1200.csv};
	\end{axis}
	\end{tikzpicture}

	\begin{tikzpicture}
	\begin{axis}[width=0.45\linewidth,scaled y ticks=true,xlabel={Time (-)},ylabel={$\int_{\Omega} \phi (x_i)\mathrm{d}x_i - \int_{\Omega} \phi_{0} (x_i)\mathrm{d}x_i$},legend style={nodes={scale=0.65, transform shape}}, xmin=0, xmax=4, xtick={0,1,2,3,4}, title={(d)},
	legend style={nodes={scale=0.95, transform shape}, row sep=2.5pt},
	legend entries={
		Uniform: $h = 4.0/2^{11}$,
		AMR: $h = 4.0/2^{7}$ to $4.0/2^{11}$
	},
	legend pos= south east,
	legend image post style={scale=1.0},
	ymin=-9.5e-5,ymax=9.5e-5]
	\addplot+[mark size = 0.75pt, each nth point=20, filter discard warning=false, unbounded coords=discard]table [x={time},y={TotalPhiMinusInit},col sep=comma] {Figures/Bubble_rise_benchmarks/Re35We10/Energy_dataRe35We10Uniform.csv};
	\addplot+[mark size = 0.75pt, each nth point=10, filter discard warning=false, unbounded coords=discard]table [x={time},y={TotalPhiMinusInit},col sep=comma] {Figures/Bubble_rise_benchmarks/Re35We10/Energy_dataRe35We10AMR.csv};
	%\addplot+[mark size = 0.75pt, each nth point=20, filter discard warning=false, unbounded coords=discard]table [x={time},y={TotalPhiNorm},col sep=comma] {Figures/Bubble_rise_benchmarks/Re35We10/Energy_data_600_1200.csv};
	\end{axis}
	\end{tikzpicture}
	\caption{ \textit{2D Single Rising Drop Test Case 1}
	(\cref{subsubsec:single_rising_drop_2D_t1}). Shown in the panels are (a) comparisons of the computed bubble shape against results
	from the literature at non-dimensional time $T = 4.2$; (b) comparisons of the rise of the bubble centroid against results from the literature; (c) decay of the energy functional illustrating energy stability; and (d) total mass conservation (integral of total $\phi$).}
	\label{fig:test_case1}
\end{figure}

\subsubsection{Test case 2}
\label{subsubsec:single_rising_drop_2D_t2}
This test case considers a lower surface tension resulting in high deformations of the bubble as it rises. %for which surface tension is lower compared to test case 2.  A lower surface tension corresponds a higher Weber (We) number consequently .  
As before, we compare the bubble shape in~\cref{fig:test_case2} with benchmark quantities presented in~\citep{ Hysing2009, Aland2012, Yuan2017}.  Panel~(a) shows the shape comparison with benchmark studies in the literature. We see an excellent agreement in the shape of the bubble. All simulations (our results and benchmarks) exhibit a skirted bubble shape. We see an excellent match in the overall bubble shape with minor differences in the dynamics of its tails between the AMR and uniform meshes. Specifically, we see that the tails of the bubble in our case pinch-off to form satellite bubbles\footnote{Such instabilities require a low $Cn$ number, as only a  thin interface can capture the dynamics of the thin tails of the bubble.}.  The slight difference between the shape and location of the satellite drops between the AMR and uniform meshes is expected as there is some difference in the convection of the satellite drops due to lower resolution in the bulk.  We performed this simulation with a $Cn=0.005$ for both the AMR and uniform meshes.  We can see in panel~(a) of~\cref{fig:test_case2} that our simulation captures this filament pinch-off in the tails.  The works of~\citet{ Aland2012, Yuan2017} did not observe these thin tails and pinch-offs, while~\citet{ Hysing2009} described pinch-off of the tails and satellite bubbles.  %The dynamics of this bubble tail is highly dependent on the numerical method used. % We think that these small differences can be attributed to the difference in the numerical method.

Panel~(b) of~\cref{fig:test_case2} compares the centroid location evolution with time. Again we see an excellent agreement with all three previous benchmark studies.  We can see from the magnified inset in this panel that for both AMR and uniform meshes the plot approaches the benchmark studies and we see an almost exact overlap between the benchmark. Next, we report the evolution of the energy functional defined in~\cref{eqn:energy_functional} for test case 2.  Panel~(c) of~\cref{fig:test_case2}~shows the decay of the total energy functional in accordance with the energy stability condition for both AMR and uniform meshes. Finally, panel~(d) of~\cref{fig:test_case2} shows the total mass of the system in comparison with the total initial mass of the system. We can see that for all spatial resolutions, the change in the total mass against the initial total mass is of the order of \num{1e-8}, which illustrates that the numerical method satisfies mass conservation over long time horizons.  As expected, for the AMR case just like in the test case 1 we see higher drift from mass conservation with AMR. As stated earlier, we attribute this to a non-mass conserving coarsening interpolation.

It is worth noticing the advantage of deploying AMR in this case. The uniform mesh has $8.4$ million elements and takes 6 hours on 20 TACC~\Frontera~ node (120 node hours).  The AMR mesh, on the other hand, has a maximum mesh count of about $0.1$ million elements at peak refinement and takes 4 hours on 2 TACC~\Frontera~ nodes on (8 node hours).  We get a much higher speed up for this test case compared to test case 1. High $We$ results a lot in the deformation of bubble shape and therefore, requires very small $Cn$ to resolve the interfacial dynamics.  Therefore, the minimum resolution required in this case is significantly higher compared to the test case 1. %Hence, we see a much higher speed up with AMR meshes compared to test case 1.  

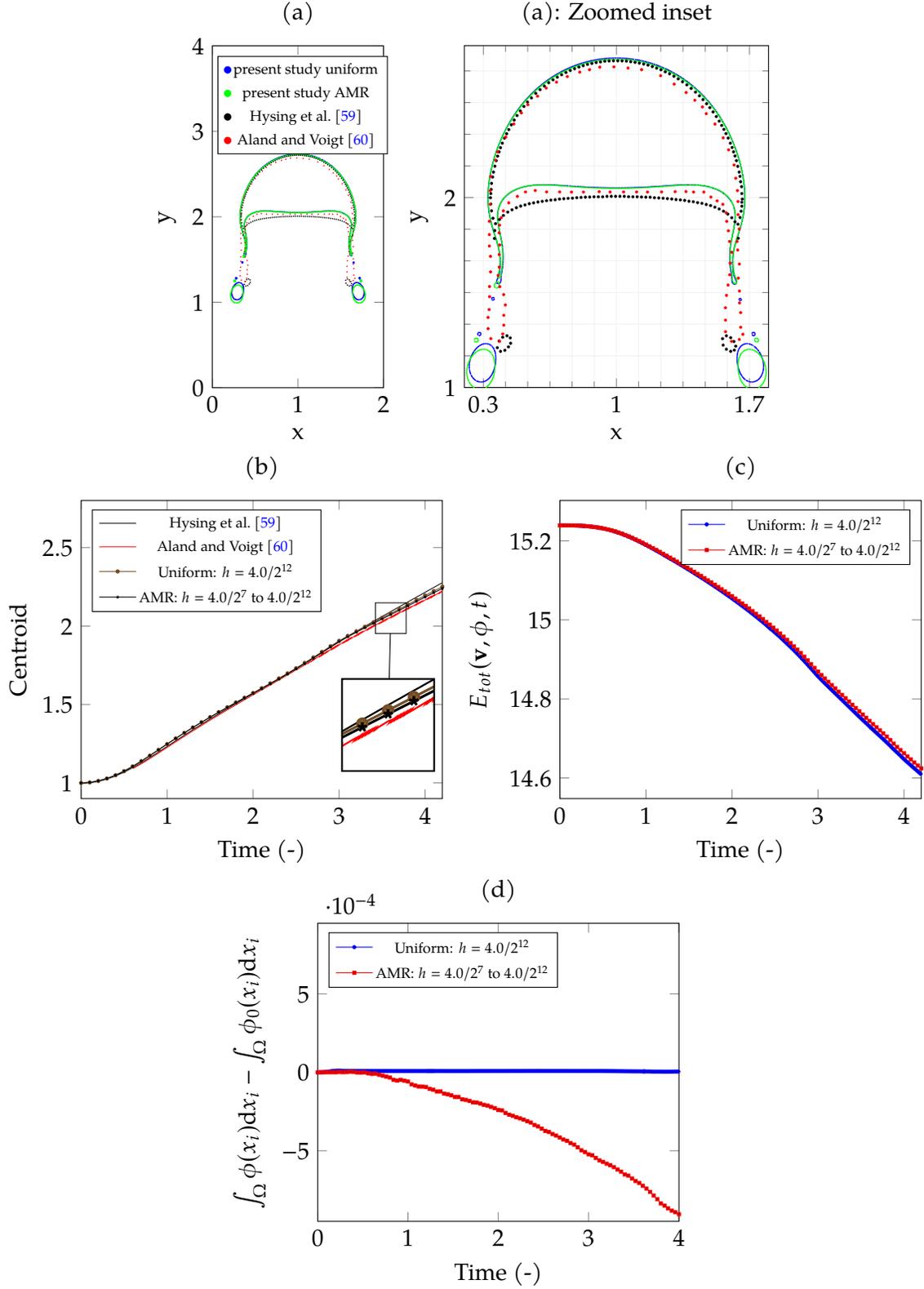
\begin{figure}[H]
	\centering
	\tikzexternaldisable
	\begin{tikzpicture}
	\begin{axis}[width=0.5\linewidth,scaled y ticks=true,xlabel={$\mathrm{x}$},ylabel={$\mathrm{y}$},legend style={nodes={scale=0.65, transform shape}}, ymin=0, ymax=4, ytick distance=1.0,  xtick={0.0, 1.0, 2}, title={(a)},
	legend style={nodes={scale=0.95, transform shape}, row sep=2.5pt},
	legend entries={
		present study uniform,
		present study AMR, 
		\citet{Hysing2009}, \citet{Aland2012}, 
		%\citet{Yuan2017}
	},
	legend pos= north west,
	legend image post style={scale=5.0},
	unit vector ratio*=1 1 1,
	xmin=0, xmax=2, 
	legend image post style={scale=3.0}
	]
	\addplot [only marks,mark size = 0.1pt,color=blue, each nth point=10, filter discard warning=false, unbounded coords=discard] table [x={x},y={y},col sep=comma] {Figures/Bubble_rise_benchmarks/Re35We125/bubbleShapeRe35We125.csv};
	\addplot [only marks,mark size = 0.1pt,color=green, each nth point=10, filter discard warning=false, unbounded coords=discard] table [x={x},y={y},col sep=comma] {Figures/Bubble_rise_benchmarks/Re35We125/bubbleShapeRe35We125AMR.csv};
	\addplot [only marks,mark size = 0.1pt,color=black,each nth point=1, filter discard warning=false, unbounded coords=discard] table [x={x},y={y},col sep=comma] {Figures/Bubble_rise_benchmarks/Re35We125/Hysing_Re35_We125_bubble_shape.csv};
	\addplot [only marks,mark size = 0.1pt,color=red,each nth point=1, filter discard warning=false, unbounded coords=discard] table [x={x},y={y},col sep=comma] {Figures/Bubble_rise_benchmarks/Re35We125/Aland_Voight_Re35_We125_bubble_shape.csv};
	%\addplot [only marks,mark size = 0.1pt,color=ForestGreen,each nth point=1, filter discard warning=false, unbounded coords=discard] table [x={x},y={y},col sep=comma] {Figures/Bubble_rise_benchmarks/Re35We125/Yuan_Re35_We125_bubble_shape.csv};
	\end{axis}
	\end{tikzpicture}
	\tikzexternalenable
	\begin{tikzpicture}
	\begin{axis}[width=0.5\linewidth,scaled y ticks=true,xlabel={$\mathrm{x}$},ylabel={$\mathrm{y}$},legend style={nodes={scale=0.65, transform shape}}, ymin=1.0, ymax=2.8, ytick distance=1.0,  xtick={0.3, 1.0, 1.7}, title={(a): Zoomed inset},
	legend style={nodes={scale=0.95, transform shape}, row sep=2.5pt},
	legend pos= north west,
	legend image post style={scale=1.0},
	unit vector ratio*=1 1 1,
	xmin=0.2, xmax=1.8, 
	legend image post style={scale=3.0},
	grid=both,
	grid style={line width=.1pt, draw=gray!10},
	minor tick num=5
	]
	\addplot [only marks,mark size = 0.1pt,color=blue, each nth point=6, filter discard warning=false, unbounded coords=discard] table [x={x},y={y},col sep=comma] {Figures/Bubble_rise_benchmarks/Re35We125/bubbleShapeRe35We125.csv};
	\addplot [only marks,mark size = 0.1pt,color=green, each nth point=6, filter discard warning=false, unbounded coords=discard] table [x={x},y={y},col sep=comma] {Figures/Bubble_rise_benchmarks/Re35We125/bubbleShapeRe35We125AMR.csv};
	\addplot [only marks,mark size = 0.5pt,color=black,each nth point=1, filter discard warning=false, unbounded coords=discard] table [x={x},y={y},col sep=comma] {Figures/Bubble_rise_benchmarks/Re35We125/Hysing_Re35_We125_bubble_shape.csv};
	\addplot [only marks,mark size = 0.5pt,color=red,each nth point=1, filter discard warning=false, unbounded coords=discard] table [x={x},y={y},col sep=comma] {Figures/Bubble_rise_benchmarks/Re35We125/Aland_Voight_Re35_We125_bubble_shape.csv};
	%\addplot [only marks,mark size = 0.5pt,color=ForestGreen,each nth point=1, filter discard warning=false, unbounded coords=discard] table [x={x},y={y},col sep=comma] {Figures/Bubble_rise_benchmarks/Re35We125/Yuan_Re35_We125_bubble_shape.csv};
	\end{axis}
	\end{tikzpicture}
	
	\tikzexternaldisable
	\begin{tikzpicture}[spy using outlines={rectangle, magnification=3, size=1.5cm, connect spies}]
	\begin{axis}[width=0.45\linewidth,scaled y ticks=true,xlabel={Time (-)},ylabel={Centroid},legend style={nodes={scale=0.65, transform shape}}, xmin=0, xmax=4.2, ymin=0.9, ymax=2.8, ytick distance=0.5,  xtick={0.0, 1.0, 2, 3, 4}, title={(b)},
	legend style={nodes={scale=0.95, transform shape}, row sep=2.5pt},
	legend entries={
		\citet{Hysing2009}, 
		\citet{Aland2012}, 
		%\citet{Yuan2017}, 
		Uniform: $h = 4.0/2^{12}$,
		AMR: $h = 4.0/2^{7}$ to $4.0/2^{12}$
		%, $h = 2.0/800$, $h = 2.0/1000$
	},
	legend pos= north west,
	legend image post style={scale=1.0}
	]
	\addplot [line width=0.1mm, color=black] table [x={time},y={centroid},col sep=comma] {Figures/Bubble_rise_benchmarks/Re35We125/Hysing_centroid_Re35We125.csv};
	%%%%%%%%%%%%%%
	\addplot [line width=0.1mm, color=red] table [x={time},y={centroid},col sep=comma] {Figures/Bubble_rise_benchmarks/Re35We125/aland_voight_centroid_Re35We125.csv};
	%%%%%%%%%%%%%%
	%\addplot [line width=0.1mm, color=ForestGreen] table [x={time},y={centroid},col sep=comma] {Figures/Bubble_rise_benchmarks/Re35We125/Yuan_centroid_Re35We125.csv};
	%%%%%
	\addplot+[mark size = 0.75pt]table [x={time},y={centroid},col sep=comma, each nth point=1, filter discard warning=false, unbounded coords=discard] {Figures/Bubble_rise_benchmarks/Re35We125/centerOfMassRe35We125_uniformLevel12.csv};	
	%%%%%
	\addplot+[mark size = 0.75pt]table [x={time},y={centroid},col sep=comma, each nth point=1, filter discard warning=false, unbounded coords=discard] {Figures/Bubble_rise_benchmarks/Re35We125/centerOfMassRe35We125AMR.csv};
	%%%%%
	%\addplot+[mark size = 0.75pt]table [x={time},y={centroid},col sep=comma, each nth point=1, filter discard warning=false, unbounded coords=discard] {Figures/Bubble_rise_benchmarks/Re35We125/centerOfMass_1000_2000.csv};
	\coordinate (a) at (axis cs:3.6,2.05);
	\end{axis}
	\spy [black] on (a) in node  at (5,1.2);
	\end{tikzpicture}
	\hskip 5pt
	\begin{tikzpicture}
	\begin{axis}[width=0.45\linewidth,scaled y ticks=true,xlabel={Time (-)},ylabel={$E_{tot}(\vec{v},\phi,t)$},legend style={nodes={scale=0.65, transform shape}}, xmin=0, xmax=4.2, xtick={0,1,2,3,4},  title={(c)},
	legend style={nodes={scale=0.95, transform shape}, row sep=2.5pt},
	legend entries={
	Uniform: $h = 4.0/2^{12}$,
	AMR: $h = 4.0/2^{7}$ to $4.0/2^{12}$
	},
	legend pos= north east,
	legend image post style={scale=1.0}]
	\addplot +[mark size = 0.75pt, each nth point=20, filter discard warning=false, unbounded coords=discard] table [x={time},y={TotalEnergy},col sep=comma] {Figures/Bubble_rise_benchmarks/Re35We125/Energy_dataRe35We125Uniform.csv};
	\addplot +[mark size = 0.75pt, each nth point=40, filter discard warning=false, unbounded coords=discard] table [x={time},y={TotalEnergy},col sep=comma] {Figures/Bubble_rise_benchmarks/Re35We125/Energy_dataRe35We125AMR.csv};
	%\addplot +[mark size = 0.75pt, each nth point=20, filter discard warning=false, unbounded coords=discard] table [x={time},y={TotalEnergy},col sep=comma] {Figures/Bubble_rise_benchmarks/Re35We125/Energy_data_1000_2000.csv};
	%\addplot +[mark size = 0.75pt, each nth point=20, filter discard warning=false, unbounded coords=discard] table [x={time},y={TotalEnergy},col sep=comma] {Figures/Bubble_rise_benchmarks/Re35We125/Energy_data_1200_2400.csv};
	\end{axis}
	\end{tikzpicture}
	
	\begin{tikzpicture}
	\begin{axis}[width=0.45\linewidth,scaled y ticks=true,xlabel={Time (-)},ylabel={$\int_{\Omega} \phi (x_i)\mathrm{d}x_i - \int_{\Omega} \phi_{0} (x_i)\mathrm{d}x_i$},legend style={nodes={scale=0.65, transform shape}}, xmin=0, xmax=4.0, xtick={0,1,2,3,4}, title={(d)},
	legend style={nodes={scale=0.95, transform shape}, row sep=2.5pt},
	legend entries={
	Uniform: $h = 4.0/2^{12}$,
	AMR: $h = 4.0/2^{7}$ to $4.0/2^{12}$
	},
	legend pos= north west,
	legend image post style={scale=1.0},
	ymin=-9.5e-4,ymax=9.5e-4]
	\addplot+[mark size = 0.75pt, each nth point=20, filter discard warning=false, unbounded coords=discard]table [x={time},y={TotalPhiMinusInit},col sep=comma] {Figures/Bubble_rise_benchmarks/Re35We125/Energy_dataRe35We125Uniform.csv};
	\addplot+[mark size = 0.75pt, each nth point=40, filter discard warning=false, unbounded coords=discard]table [x={time},y={TotalPhiMinusInit},col sep=comma] {Figures/Bubble_rise_benchmarks/Re35We125/Energy_dataRe35We125AMR.csv};
	%\addplot+[mark size = 0.75pt, each nth point=20, filter discard warning=false, unbounded coords=discard]table [x={time},y={TotalPhiNorm},col sep=comma] {Figures/Bubble_rise_benchmarks/Re35We125/Energy_data_1000_2000.csv};
	%\addplot+[mark size = 0.75pt, each nth point=20, filter discard warning=false, unbounded coords=discard]table [x={time},y={TotalPhiNorm},col sep=comma] {Figures/Bubble_rise_benchmarks/Re35We125/Energy_data_1200_2400.csv};
	\end{axis}
	\end{tikzpicture}
	\caption{\textit{2D Single Rising Drop Test Case 2} 
	(\cref{subsubsec:single_rising_drop_2D_t2}). Shown in the panels are (a)~comparisons of the computed bubble shape against results from the literature at non-dimensional time $T = 4.2$; (b)~comparisons of the rise of the bubble centroid against results from the literature; (c)~decay of the energy functional illustrating energy stability; and (d)~total mass conservation (integral of total $\phi$).}
	\label{fig:test_case2}
\end{figure}
\subsection{2D simulations: Rayleigh-Taylor instability}
\label{subsec:rayleigh_taylor_2D}
We now demonstrate the performance of the numerical framework with large deformation in the interface and chaotic regimes (high Reynolds numbers) with a moderately high density ratio.  While the bubble rise case in the previous sub-section is an interplay between surface tension and buoyancy, buoyancy dominates the evolution of the Rayleigh-Taylor instability (RTI). Here the choice of non-dimensional numbers ensures that the surface tension effect is small (high Weber number). In contrast, other studies switch off the surface tension forcing terms in the momentum equations (see, e.g.,~\citep{ Xie2015, Tryggvason1990, Li1996, Guermond2000}).  We choose a finite but small surface tension because it ensure more stable filaments which are generated when the interfacial dynamics become more chaotic.  This represents a more difficult case to simulate compared to the zero surface tension counterpart as in the chaotic regime of the RTI the thin filaments and satellite drops generated need to be resolved with a very high resolution. Additionally, one has to take smaller timesteps to resolve pinch offs during the breakup of these filaments and satellite drops.  

The setup is as follows: the heavier fluid is on top of lighter fluid and the interface is perturbed. The heavier fluid on top penetrates the lighter fluid and buckles, which generates instabilities.  This interface motion is challenging to track due to large changes in its topology. Additionally, the long term RTI generally encompasses turbulent conditions that require resolving finer scales to capture the complete dynamics.  We non-dimensionalize the problem by selecting the width of the channel as the characteristic length scale and the density of the heavier fluid as the characteristic specific density.  Just as in the bubble rise case we use buoyancy-based scaling, setting the Froude number ($Fr = u^2/(gD)$) to 1, which fixes the non-dimensional velocity scale to be $u = \sqrt{gD}$, where $g$ is the gravitational acceleration, and $D$ is the width of the channel.  Using this velocity to calculate the Reynolds number, we get $Re = \rho_L g^{1/2}D^{3/2}/\mu_L$, where $\rho_L$ and $\mu_L$ are the specific density and specific viscosity of the light fluid, respectively. We set the Reynolds number to $Re = 3000$.  These choices lead to a Weber number of $We = \rho_c g D^2/\sigma$.  To compare our results with previous studies, we simulate with the same initial conditions as presented in~\citet{Xie2015, Khanwale2021}.  The $We$ number is selected to be 100, so that the effect of surface tension is small on the evolution of the interface but still finite.  

The Atwood number ($At$) is often used to parameterize the dependence on density ratio, with $At = \left(\rho_{+} - \rho_{-}\right)/\left(\rho_{+} + \rho_{-}\right)$.  For the density ratio of 0.1, the Atwood numbers is $At = 0.82$.
We chose specific density of the heavy fluid to non-dimensionalize, therefore $\rho_{+} = 1.0$, and $\rho_{-} = 0.33$ for $At = 0.5$, while $\rho_{-} = 0.1$ for $At = 0.82$.  $\nu_{+}/\nu_{-}$ the viscosity ratio is set to 1. We use a no-slip boundary condition for velocity on all the walls along with no flux boundary conditions for $\phi$ and $\mu$.  The no-flux boundary condition for $\phi$ and $\mu$ inherently produce a 90 degree wetting angle (neutral contact angle) for both the fluids.  

%\begin{remark}
	%Notice, that even though the $We$ number set to a high value, it is not zero.  This is in contrast with the studies done in the literature, where surface tension is chosen as zero.  
Weak surface tension reduces vortex roll-up in the simulations of immiscible systems, especially at lower $At$. \citet{ Waddell2001} experimentally showed different vortex roll-up amounts for miscible and immiscible systems. This difference is analogous to the difference between zero surface tension simulations and finite surface tension simulations.  This effect is irrelevant to compare front locations against the literature (short time horizons). Nevertheless, it is crucial to accurately track the long time dynamics (as smaller filaments are more stable in the non-zero surface tension case).  A finite surface tension would generate different long term mixing and breakup of the interface. 
%\end{remark}     

We run numerical experiment for a long term RTI with $Cn = 0.00125$\footnote{Note that this is a very small $Cn$ number and requires a very fine mesh. This corresponds to the finest resolution selected in~\citep{Khanwale2021} for the same case.} with an adaptive mesh where the refinement varies from $8/2^{10}$ to $8/2^{14}$, for the Atwood number of $At = 0.82$.  We use a variable timestep for this simulation.  We use a timestep of \num{1.25e-4} till non-dimensional time $t = 0.6$ as we have linear behavior of the interface.  We then decrease the timestep to \num{5e-5} till $t = 1.4$ as the interface starts accelerating, and then we further reduce it to \num{2.5e-5} at $t = 1.4$ in anticipation of the shedding and breakup events near the interface. 
A carefully tuned algebraic multi-grid linear solver with successive over-relaxation is setup for the velocity, Cahn-Hilliard and pressure Poisson solvers.  
We detail the command-line arguments along with the tolerances used in \ref{subsec:app_rt2d}.  Note that the resolution here is finer than the resolution used in~\citep{Khanwale2021} where we run the same case with an uniform mesh of $800 \times 6400$ elements.  

\Cref{fig:rt2d} shows the snapshots of the interface shape along with the corresponding vorticity generated as it evolves in time for $At = 0.82$ and $Cn = 0.00125$. We observe the usual evolution of Rayleigh Taylor instability where the heavier fluid penetrates the light one, causing the lighter fluid to rise near the wall.  The penetrating plume of the heavier fluid sheds small filaments at a non-dimensional time of around $t^\prime = 1.9$.  The penetrating plume is symmetric at early times, with symmetry breaking occurring at longer times due to increasingly chaotic behavior. The instability further proceeds to a periodic flapping. At longer times, the instability transitions to a chaotic mixing stage.  Notice that there are strong vortex roll-ups generated near the thin filaments (e.g see vorticity in panels (f), (g), (h)) which influence the onset and progression of chaotic behavior in RTI.  As this is a finite surface tension case these filaments are more stable and require high resolution to resolve these fine scale vortical structures near the interface.  

We observe from~\cref{fig:rt2d} that the development of the instability (at longer times) depends on the resolution of shedding filaments. Therefore, the long-time dynamics depend on the interface thickness that the $Cn$ number controls.  \citet{Khanwale2021} presented analysis of how $Cn$ numbers change the long term mixing behavior of RTI.  They performed RTI simulations for $Cn = 0.005, 0.0025, 0.00125$ and showed that for $Cn = 0.00125$ the mixing in RTI is substantially different at long times.  In the present work we choose the finest $Cn = 0.00125$ which  requires a very high resolution mesh.  However, in the present work we use an adaptive mesh to capture such a fine $Cn$ interface representation.  First we compare the results of short term with literature in panel (a) of~\cref{fig:RT_2D_comparison} which shows the comparison of locations of the bottom and top front of the interface.  We see an excellent agreement for short times with the literature. Note that for short times the effect of the finite surface tension is not noticeable and the front locations match very well with the literature. This is because for shorter times the dynamics are dominated by large scale flow structures.  We also compare our interface fronts with \citet{Khanwale2021} where RTI simulations were performed with a uniform mesh with finite surface tension (same $We$ number as the present study) and we also see excellent match even for longer times.  Panel~(b) from~\cref{fig:RT_2D_comparison} shows the decay of the energy functional with respect to time demonstrating energy stability for $At = 0.82$.  Panel~(c) from~\cref{fig:RT_2D_comparison} shows the drift of mass which is expected to be as close to zero as possible.  However, we see that for long term behavior the mass drift increases, we attribute this behavior to the interpolation operations during coarsening of the adaptive mesh refinement, this is similar to behavior in panel (d) of \cref{fig:test_case1} and \cref{fig:test_case2}.  

To quantify vortex dynamics at long times in RTI we perform vortex identification using Q-criterion~\citep{Hunt1988}.  \Cref{fig:rt2d_q} shows turbulent instability in RTI.  Panel (a) shows the complex mixing structure of the interface in RTI in the chaotic regime, panel (b) of \cref{fig:rt2d_q} shows the vorticity in the system after a very long time ($t^\prime = t\sqrt{At} = 4.962$), panel (c) shows the Q-criterion at the same time.  To identify vortices we threshold the Q-criterion between $Q=1.9$ and $Q = 2$ which is shown in panel (d) of \cref{fig:rt2d_q}.  It can be clearly seen that a myriad of small scale structures generated near the filaments due to shear instability, and are resolved in our simulation. Note that as this is a 2D case, therefore there is no vortex stretching, instead we have a redistribution of vorticity through the non-linear convection.  It is well known \cite{Kraichnan1967b} that there is a strong inverse cascade in the case of 2D turbulence and small scale structures combine to form large scale vortical structures.  Therefore, it is even more important to resolve these near interface vortical structures which eventually feed in energy to the larger scales to correctly capture the large scale mixing in the system.  This fine scale resolution has a profound influence on the higher-order statistics of Rayleigh-Taylor instability.  We defer a detailed analysis of higher-order statistics to future work. %\footnote{These high-resolution simulations of Rayleigh Taylor instability over long time horizons could serve as benchmarks. This data will be made publicly available.} 

To demonstrate the resolution of the adaptive mesh that we are using for this case we show a progressive zoomed insets of the mesh overlayed on top of $\phi$ in \cref{fig:rt2d_mesh}.  Notice that we use an adaptive mesh where the refinement varies from element size of $8/2^{10}$ to $8/2^{14}$, with higher refinement near the interface. This resulted in a maximum problem size of about 0.8 million nodes, which ran on TACC~\Frontera~ with 16 nodes for 40 hours.  The corresponding uniform mesh with a resolution of $8/2^{14}$ resulted in about 33.5 million dof which would be substantially more expensive to run.  

% \begin{remark}
% 	This non-linear iteration for CH reduces to linear equivalent as Newton iteration converges in 1 iteration for the small timestep used for the Rayleigh-Taylor instability. All the other solver in the block are already linear. Therefore, for small timesteps needed for Rayleigh-Taylor is fully linear and block iterative and second order and equivalent to schemes in~\citep{Shen2010a,Shen2010b,Shen2015,Chen2016,Zhu2019} in terms of computational expenditure. 
% \end{remark}

\begin{figure}[H]
	\centering	
	\begin{tabular}{p{0.18\textwidth}p{0.18\textwidth}p{0.18\textwidth}p{0.18\textwidth}p{0.18\textwidth}}
		%%%%%%%%%%%%%%%%%%%%%% 1 %%%%%%%%%%%%%%%%%%%%%%%%%%%%%%%%%%%
		\subfigure [$t^\prime = 0.0$] {
			\includegraphics[width=\linewidth]{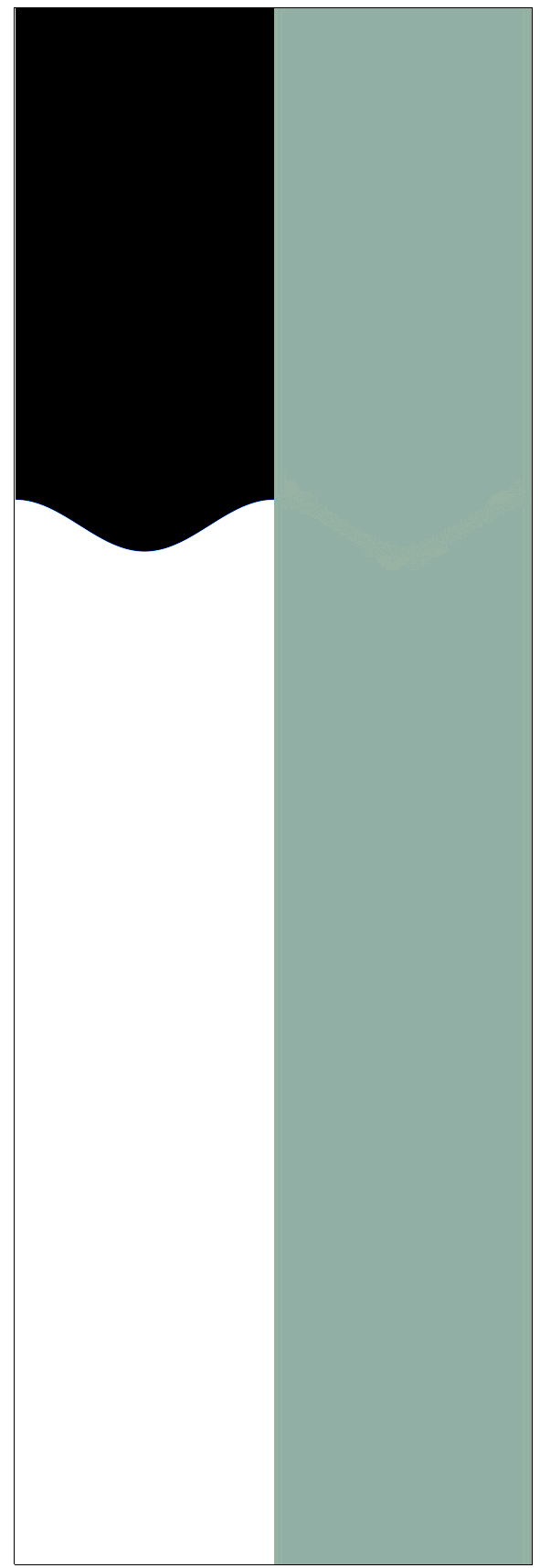}
			\label{subfig:rt_snap_1}
		} &
		%%%%%%%%%%%%%%%%%%%%%% 2 %%%%%%%%%%%%%%%%%%%%%%%%%%%%%%%%%%%
		\subfigure [$t^\prime = 1.385$] {
			\includegraphics[width=\linewidth]{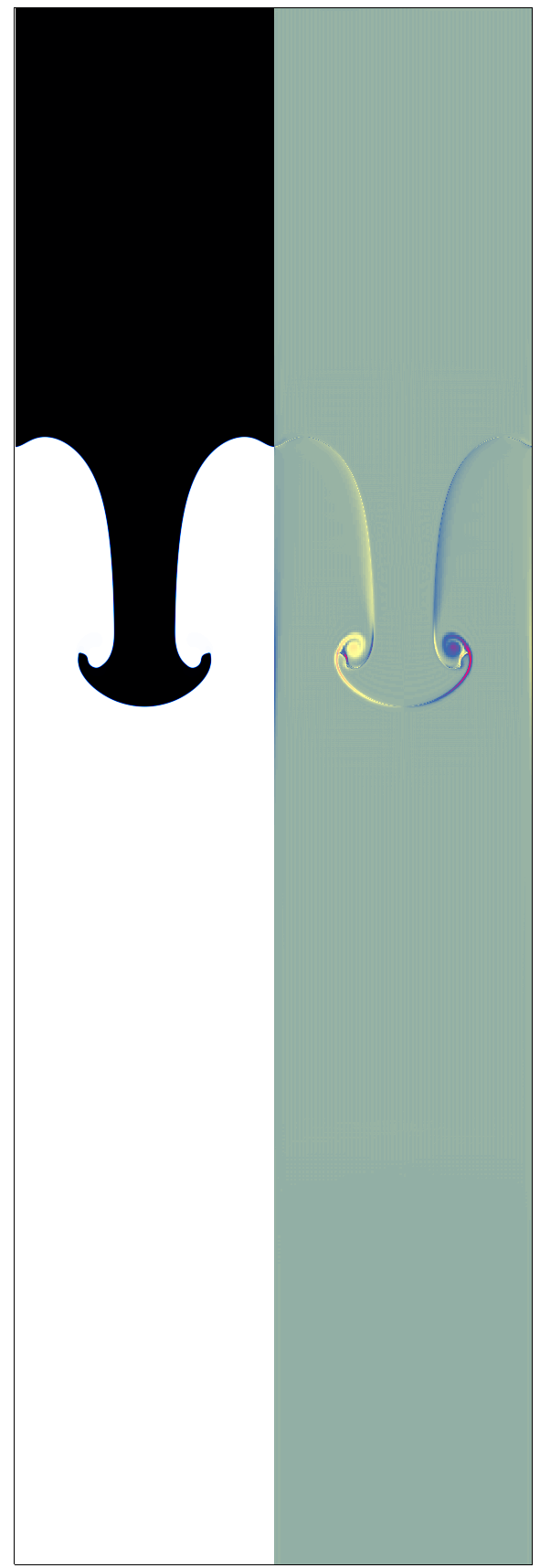}
			\label{subfig:rt_snap_2}
		} & 
		%%%%%%%%%%%%%%%%%%%%%% 3 %%%%%%%%%%%%%%%%%%%%%%%%%%%%%%%%%%%
		\subfigure [$t^\prime = 1.566$] {
			\includegraphics[width=\linewidth]{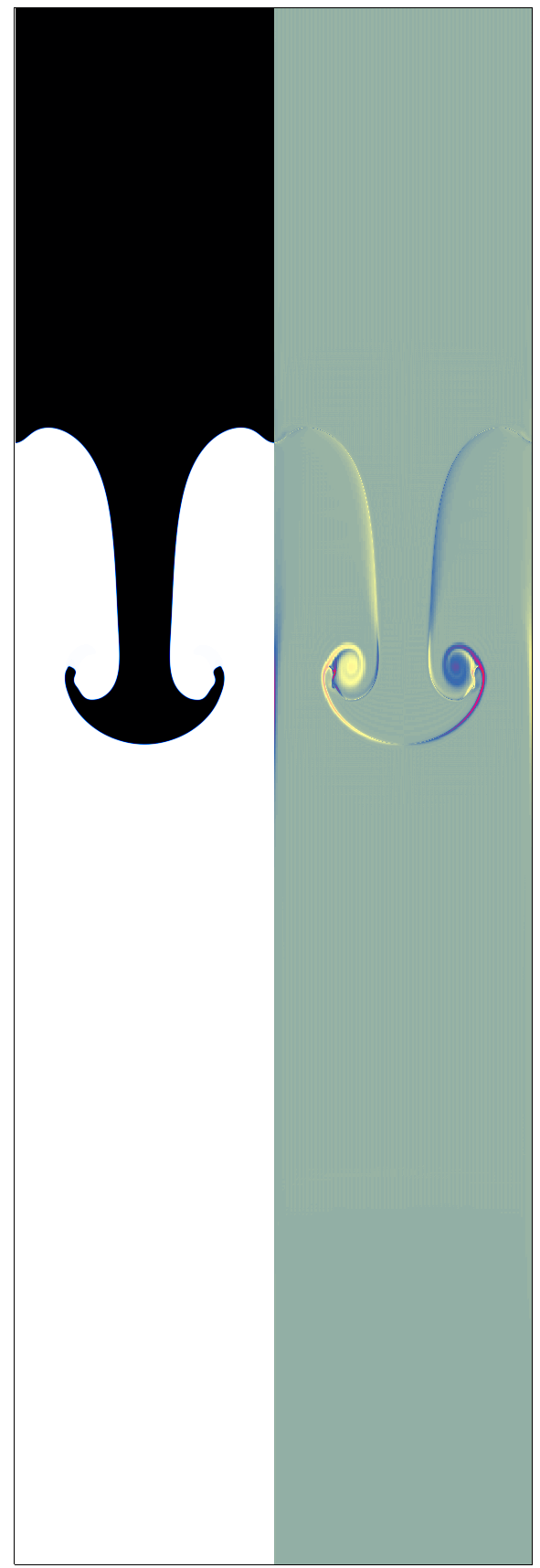}
			\label{subfig:rt_snap_3}
		} &
		
		%%%%%%%%%%%%%%%%%%%%%% 4 %%%%%%%%%%%%%%%%%%%%%%%%%%%%%%%%%%%
		\subfigure [$t^\prime = 1.847$] {
			\includegraphics[width=\linewidth]{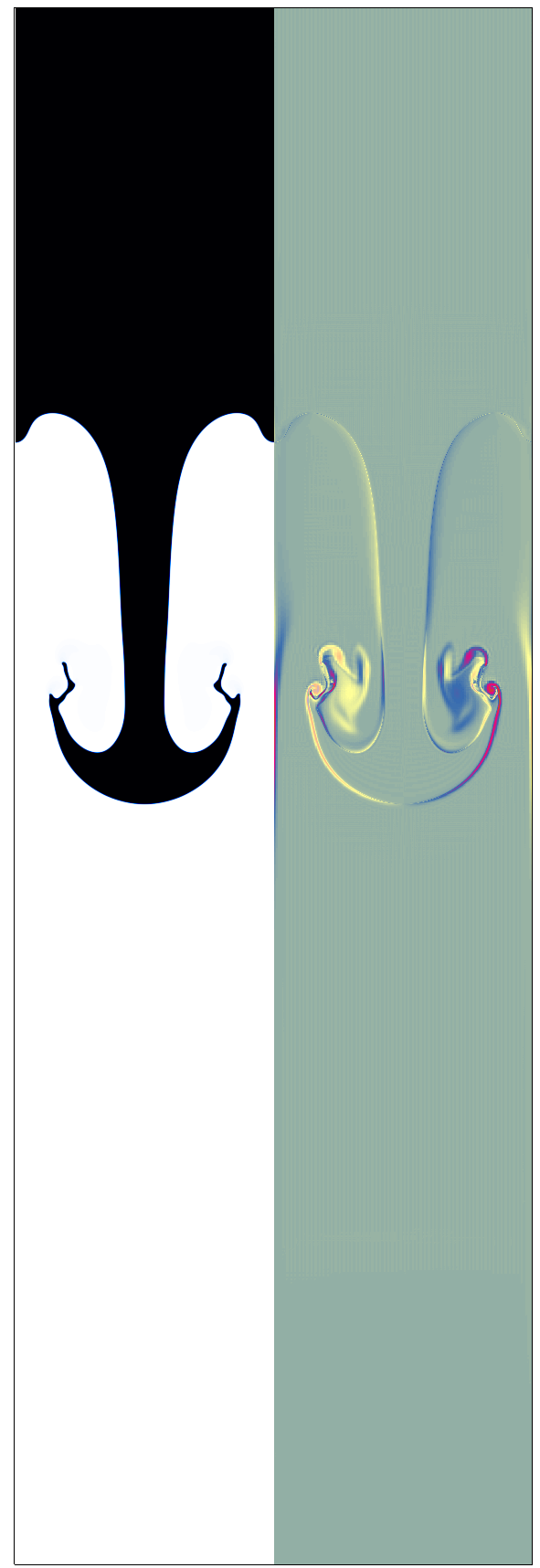}
			\label{subfig:rt_snap_4}
		} &
		%%%%%%%%%%%%%%%%%%%%%% 5 %%%%%%%%%%%%%%%%%%%%%%%%%%%%%%%%%%%
		\subfigure [$t^\prime = 2.100$] {
			\includegraphics[width=\linewidth]{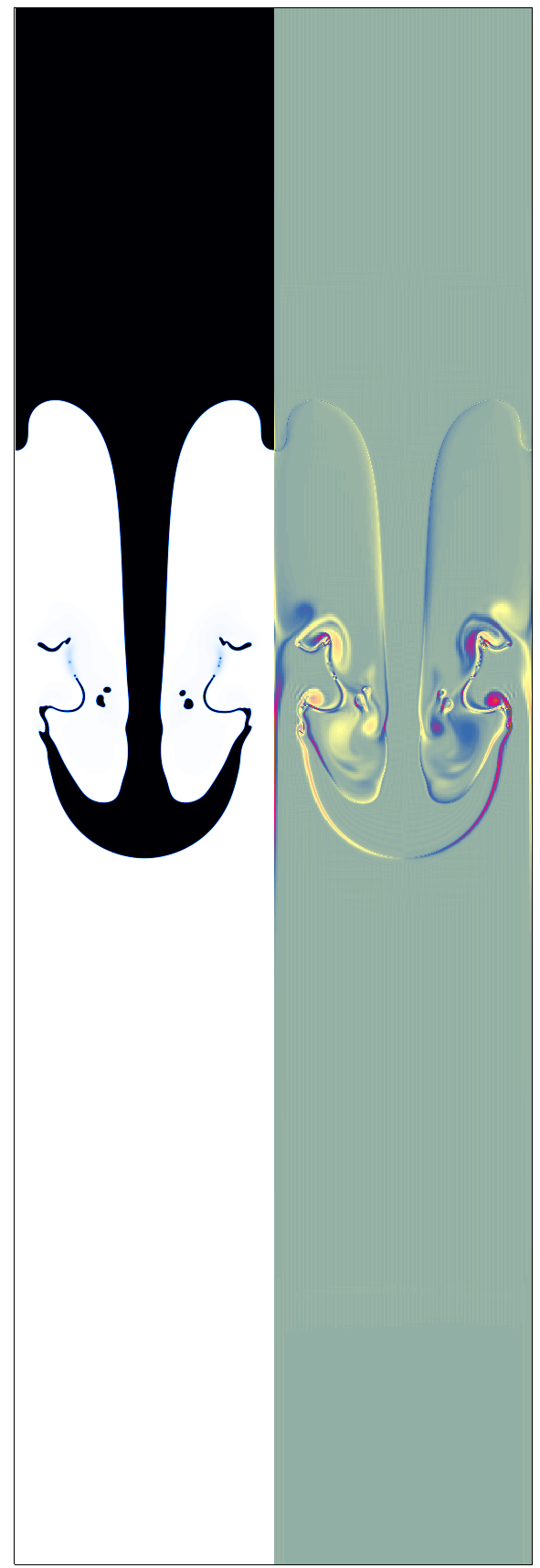}
			\label{subfig:rt_snap_5}
		} \\ 
		%%%%%%%%%%%%%%%%%%%%%% 6 %%%%%%%%%%%%%%%%%%%%%%%%%%%%%%%%%%%
		\subfigure [$t^\prime = 2.644$] {
			\includegraphics[width=\linewidth]{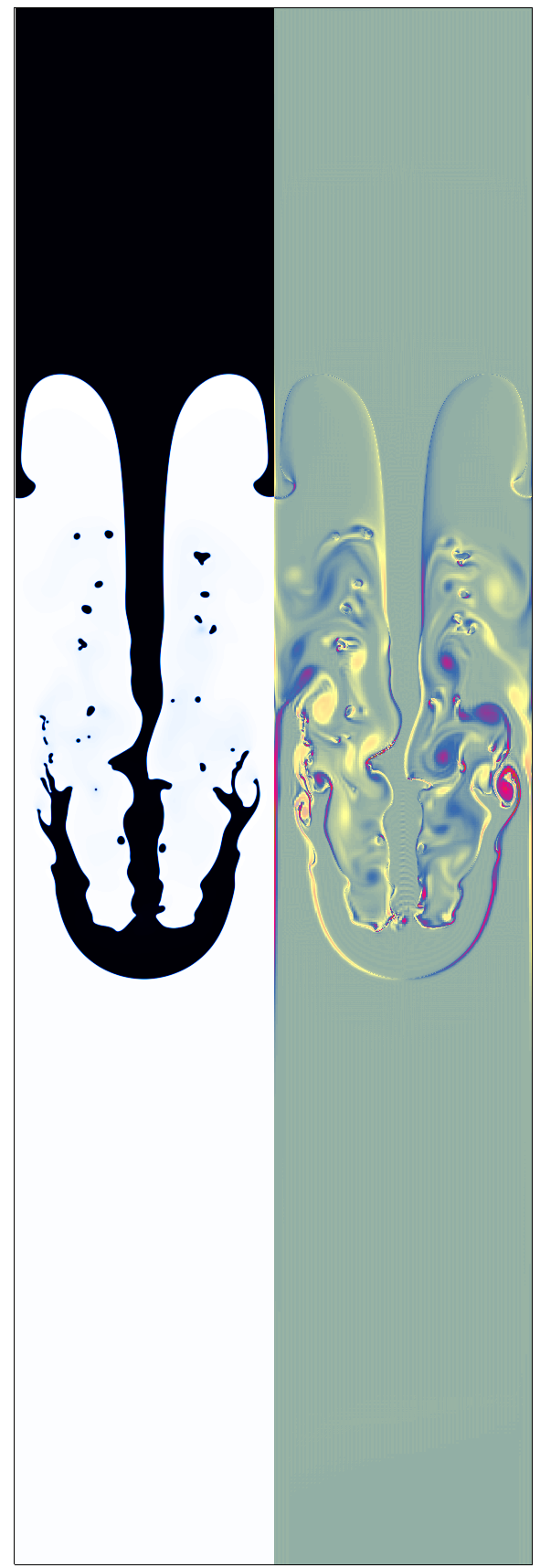}
			\label{subfig:rt_snap_6}
		} &
		
		%%%%%%%%%%%%%%%%%%%%%% 7 %%%%%%%%%%%%%%%%%%%%%%%%%%%%%%%%%%%
		\subfigure [$t^\prime = 3.006$] {
			\includegraphics[width=\linewidth]{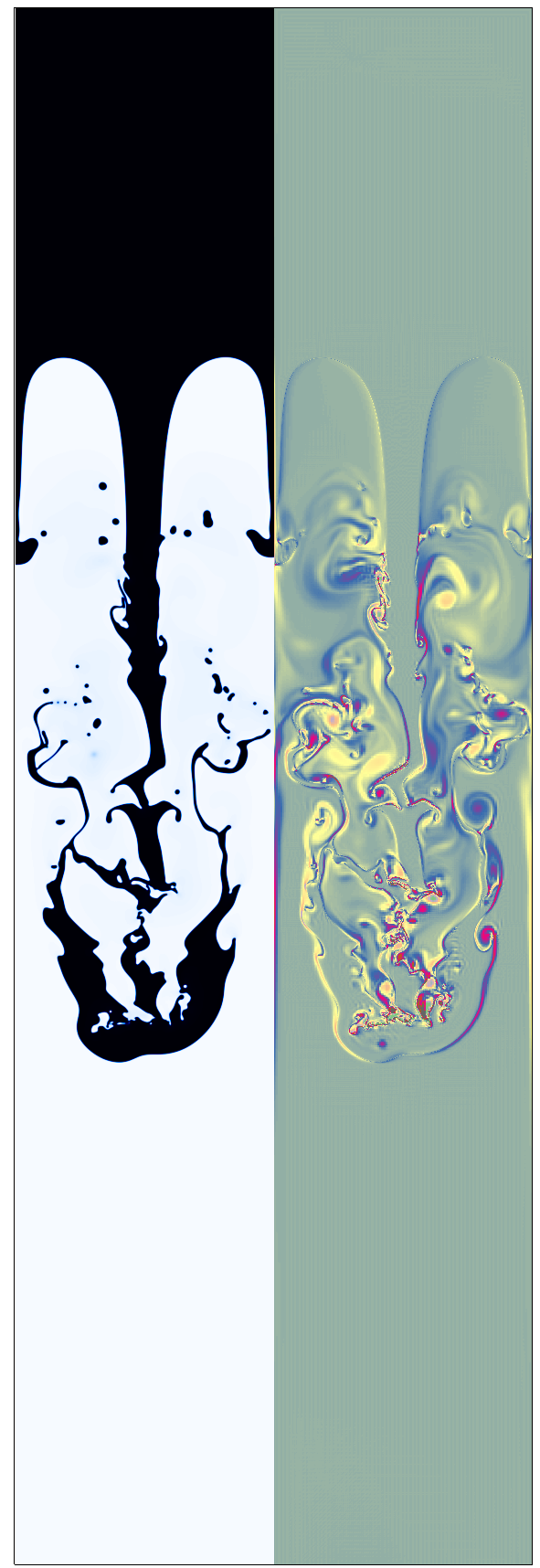}
			\label{subfig:rt_snap_7}
		} &
		%%%%%%%%%%%%%%%%%%%%%% 8 %%%%%%%%%%%%%%%%%%%%%%%%%%%%%%%%%%%
		\subfigure [$t^\prime = 3.549$] {
			\includegraphics[width=\linewidth]{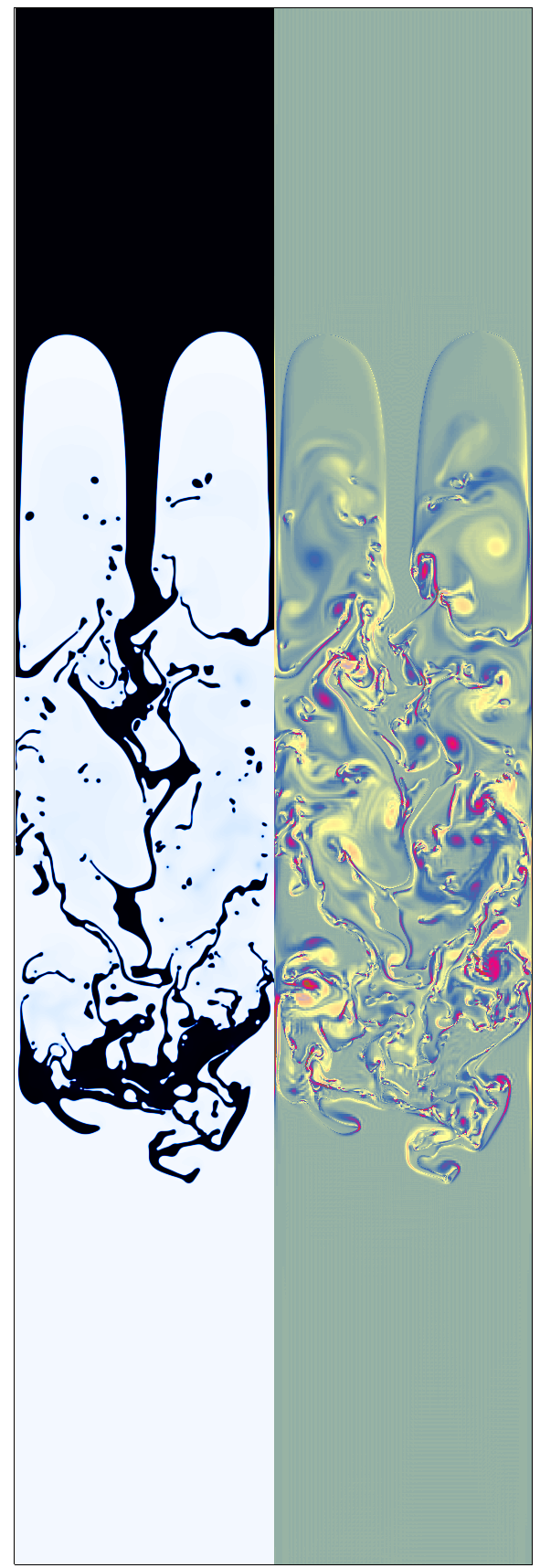}
			\label{subfig:rt_snap_8}
		} & 
		%%%%%%%%%%%%%%%%%%%%%% 9 %%%%%%%%%%%%%%%%%%%%%%%%%%%%%%%%%%%
		\subfigure [$t^\prime = 4.238$] {
			\includegraphics[width=\linewidth]{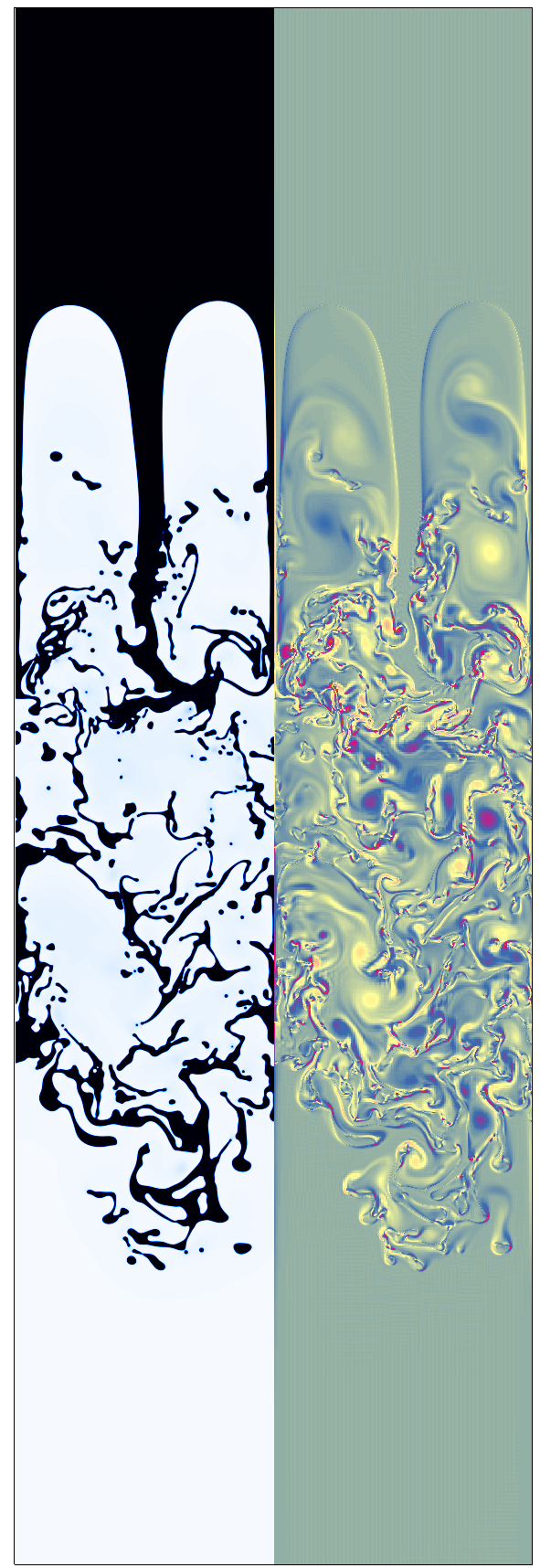}
			\label{subfig:rt_snap_9}
		} &
		%%%%%%%%%%%%%%%%%%%%%% 10 %%%%%%%%%%%%%%%%%%%%%%%%%%%%%%%%%%%
		\subfigure [$t^\prime = 4.962$] {
			\includegraphics[width=\linewidth]{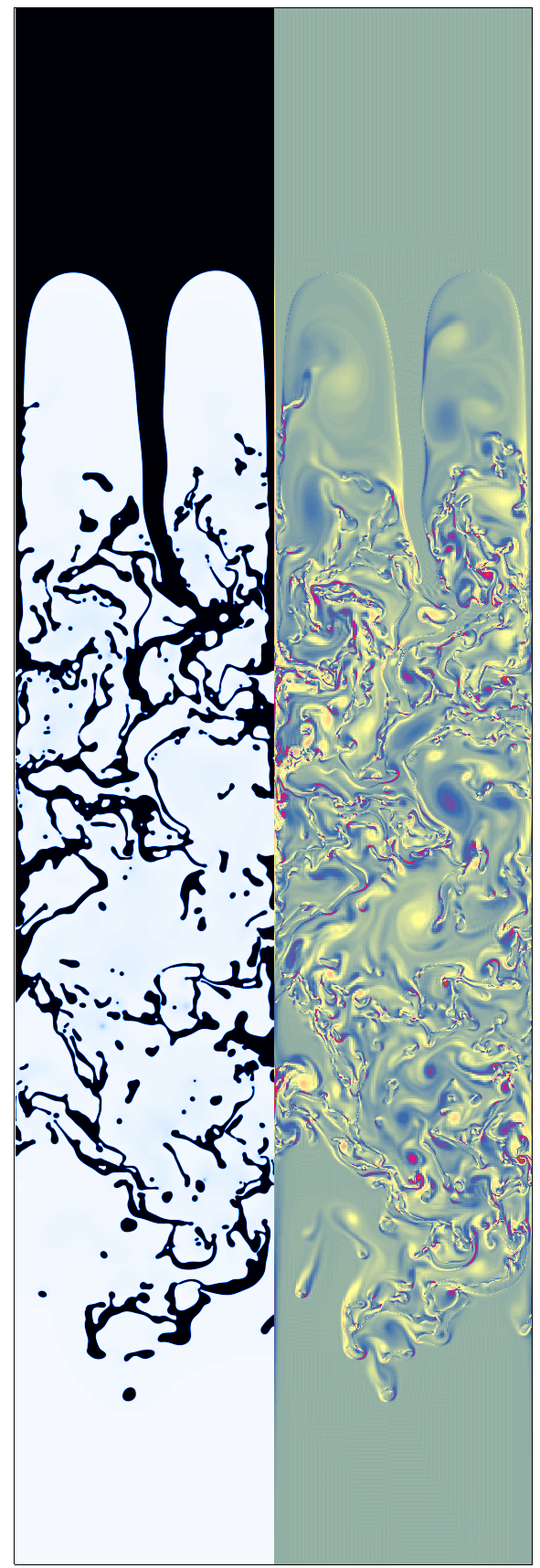}
			\label{subfig:rt_snap_10}
		}  
	\end{tabular}
	\caption{\textit{Rayleigh-Taylor instability in 2D} (\cref{subsec:rayleigh_taylor_2D}): Dynamics of the interface as a function of time for $At = 0.82$ (density ratio of 0.1).  In each panel the left plot illustrates the interface, and the right plot shows corresponding vorticity. Here normalised time $t' = t \sqrt{At}$}
	\label{fig:rt2d}
\end{figure}

%\subsubsection{Influence of $Cn$ on long time dynamics}

\begin{figure}[H]
	\centering
	\tikzexternaldisable
	\begin{tikzpicture}[]
	\begin{axis}[width=0.5\linewidth,scaled y ticks=true,xlabel={Normalised time $t' = t \sqrt{At}$(-)},ylabel={Position(-)},legend style={nodes={scale=0.65, transform shape}}, xmin=0, xmax=3, ymin=-3.5, ymax=2.0, ytick distance=1.0,  xtick={0.0, 0.5, 1.0, 1.5, 2, 2.5, 3.0}, title={(a)},
	legend style={nodes={scale=0.95, transform shape}, row sep=2.0pt},
	legend entries={
		%\citet{Ding2007}~($At = 0.5$), \citet{Guermond2000}~($At = 0.5$), \citet{Tryggvason1988}~($At = 0.5$), 
		%present study bottom front ($At = 0.5$~$Cn = 0.005$), 
		%present study top front ($At = 0.5$~$Cn = 0.005$),
		%present study bottom front ($At = 0.5$~$Cn = 0.0025$), 
		%present study top front ($At = 0.5$~$Cn = 0.0025$),
		%present study bottom front ($At = 0.5$~$Cn = 0.00125$), 
		%present study top front ($At = 0.5$~$Cn = 0.00125$), 
		\citet{Xie2015}~($At = 0.82$), 
		%present study bottom front ($At = 0.82$~$Cn = 0.005$), 
		%present study top front ($At = 0.82$~$Cn = 0.005$),		
		%present study bottom front ($At = 0.82$~$Cn = 0.0025$), 
		%present study top front ($At = 0.82$~$Cn = 0.0025$),
		\citet{Khanwale2021}~($At = 0.82$~$Cn = 0.00125$~Uniform Mesh),
		\citet{Khanwale2021}~($At = 0.82$~$Cn = 0.00125$~Uniform Mesh),
		present study bottom front ($At = 0.82$~$Cn = 0.00125$~AMR), 
		present study top front ($At = 0.82$~$Cn = 0.00125$)~AMR},
	legend pos= outer north east,%north west,
	legend image post style={scale=1.0},
	legend style={font=\footnotesize}
	]
	\addplot [only marks,mark size = 1.5pt, mark=triangle*, color=black, each nth point=3, filter discard warning=false, unbounded coords=discard] table [x={time},y={position},col sep=comma] {Figures/RT_2D/xie_combined_At0dot82.csv};
	%%%%%
	%\addplot [line width=0.25mm, color = blue, dotted]table [x={time},y={bottomFront},col sep=comma, each nth point=3, filter discard warning=false, unbounded coords=discard] {Figures/RT_2D/extentsRTinstability_At0dot82_Cn0dot005.csv};
	%%%%%
	%\addplot [line width=0.25mm, color = blue, dotted]table [x={time},y={topFront},col sep=comma, each nth point=3, filter discard warning=false, unbounded coords=discard] {Figures/RT_2D/extentsRTinstability_At0dot82_Cn0dot005.csv};
	%%%%%
	%\addplot [line width=0.25mm, color = blue, dashed]table [x={time},y={bottomFront},col sep=comma, each nth point=3, filter discard warning=false, unbounded coords=discard] {Figures/RT_2D/extentsRTinstability_At0dot82_Cn0dot0025.csv};
	%%%%%
	%\addplot [line width=0.25mm, color = blue, dashed]table [x={time},y={topFront},col sep=comma, each nth point=3, filter discard warning=false, unbounded coords=discard] {Figures/RT_2D/extentsRTinstability_At0dot82_Cn0dot0025.csv};
	%%%%%
	\addplot [only marks,mark size = 1.0pt, mark=square, color=black, each nth point=1, filter discard warning=false, unbounded coords=discard]table [x={time},y={bottomFront},col sep=comma] {Figures/RT_2D/extentsRTinstability_At0dot82_Cn0dot00125.csv};
	%%%%%
	\addplot [only marks,mark size = 1.0pt, mark=square, color=black, each nth point=1, filter discard warning=false, unbounded coords=discard]table [x={time},y={topFront},col sep=comma] {Figures/RT_2D/extentsRTinstability_At0dot82_Cn0dot00125.csv};
	%%%%
	\addplot [line width=0.25mm, color = blue, solid]table [x={scaledtime},y={bottomFront},col sep=comma, each nth point=1, filter discard warning=false, unbounded coords=discard] {Figures/RT_2D/extentsRTinstability_At0dot82_Cn0dot00125_AMR.csv};
	%%%%
	\addplot [line width=0.25mm, color = blue, solid]table [x={scaledtime},y={topFront},col sep=comma, each nth point=1, filter discard warning=false, unbounded coords=discard] {Figures/RT_2D/extentsRTinstability_At0dot82_Cn0dot00125_AMR.csv};
	\end{axis}
	\end{tikzpicture}
	
	\tikzexternalenable
	\includegraphics{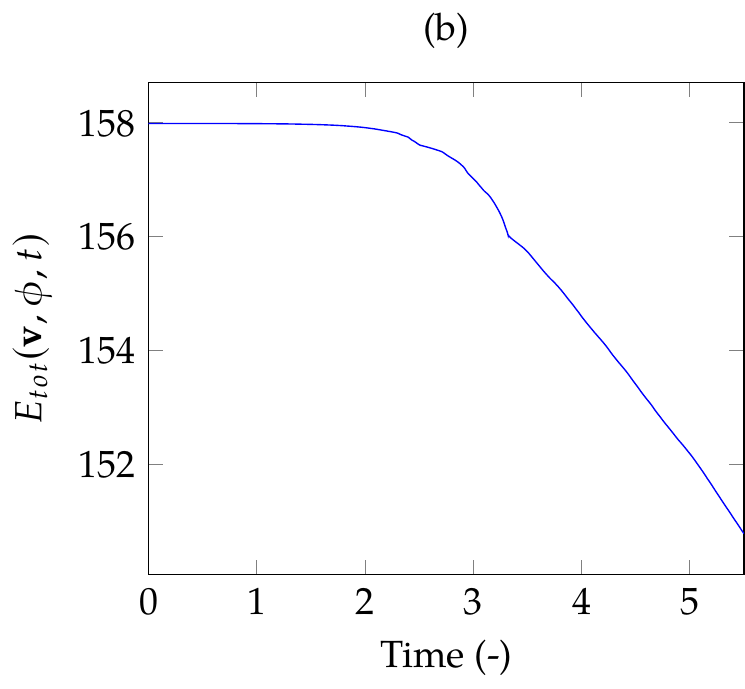}
	\hskip 5pt
	\includegraphics{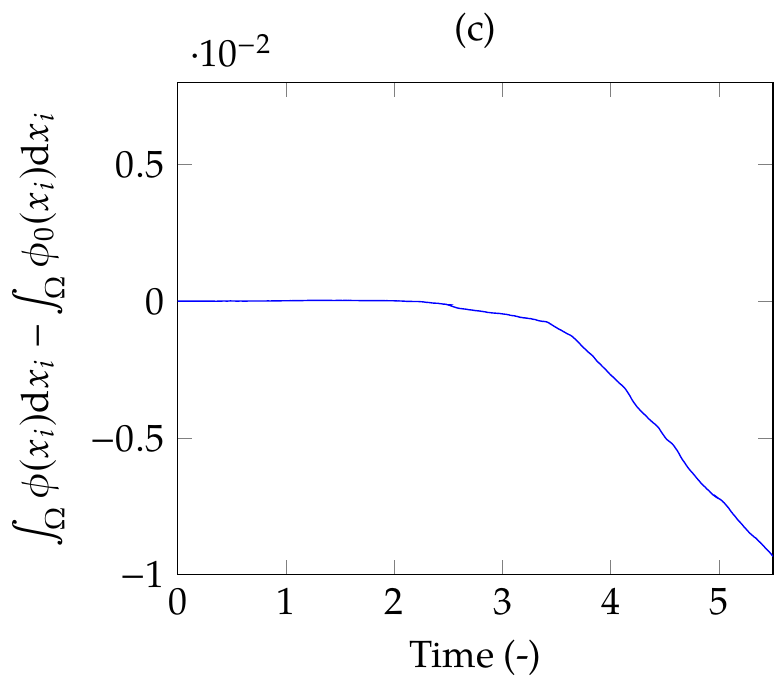}
	\caption{\textit{Rayleigh-Taylor instability 2D} (\cref{subsec:rayleigh_taylor_2D}): (a) Comparison of positions of top and bottom front of the interface with literature; (b) decay of the energy functional illustrating Energy Stability for $At = 0.82$; (c) total mass conservation (integral of total $\phi$ - initial mass) for $At = 0.82$. Note the excellent mass conservation across $\sim 260,000$ time steps. As stated in the text, most of this deviation is attributed to the non-mass conserving interpolation schemes used to map solutions between two consecutive meshes.}
	\label{fig:RT_2D_comparison}
\end{figure}
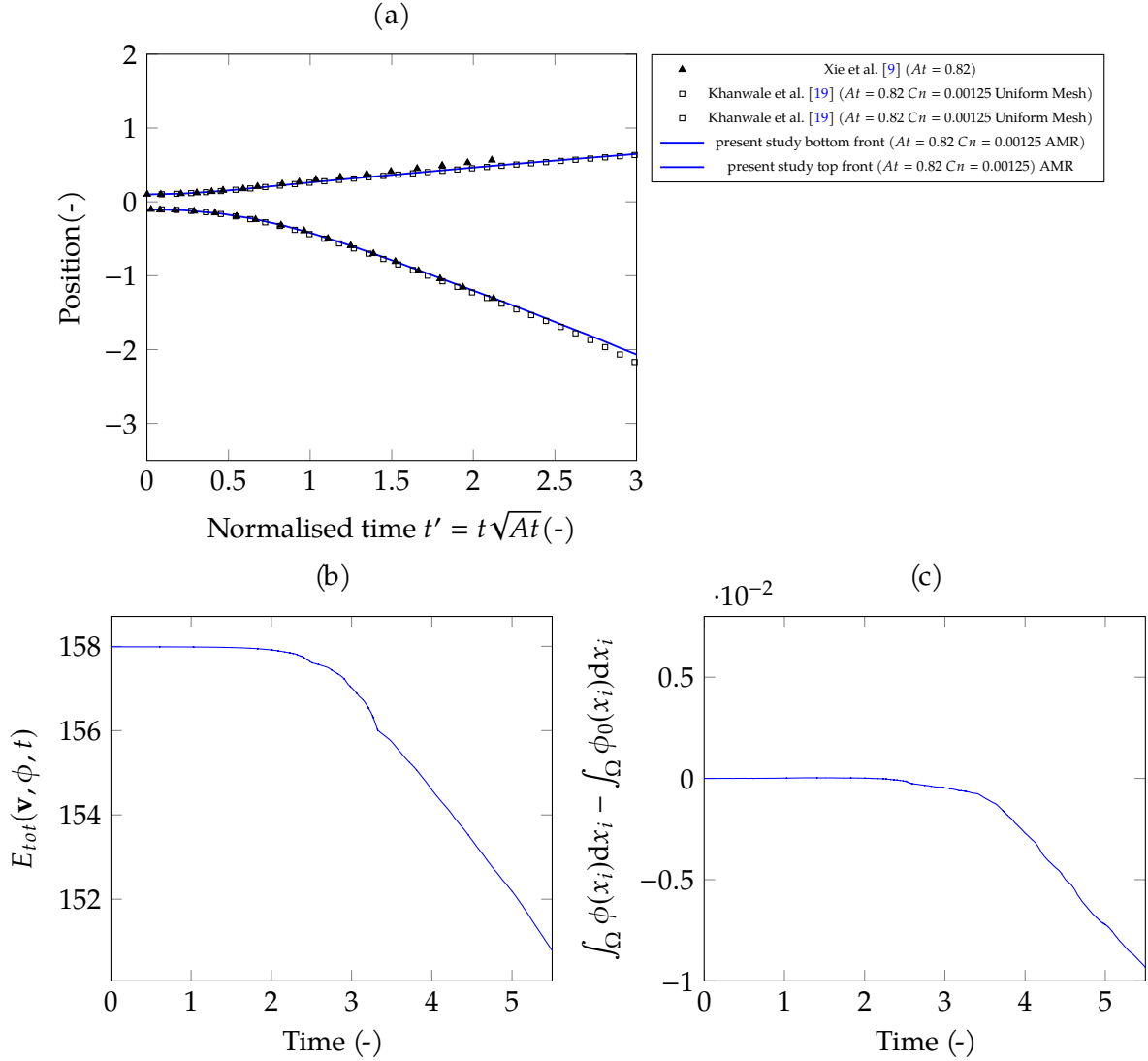

\begin{figure}[H]
	\centering	
	\includegraphics[width=\linewidth]{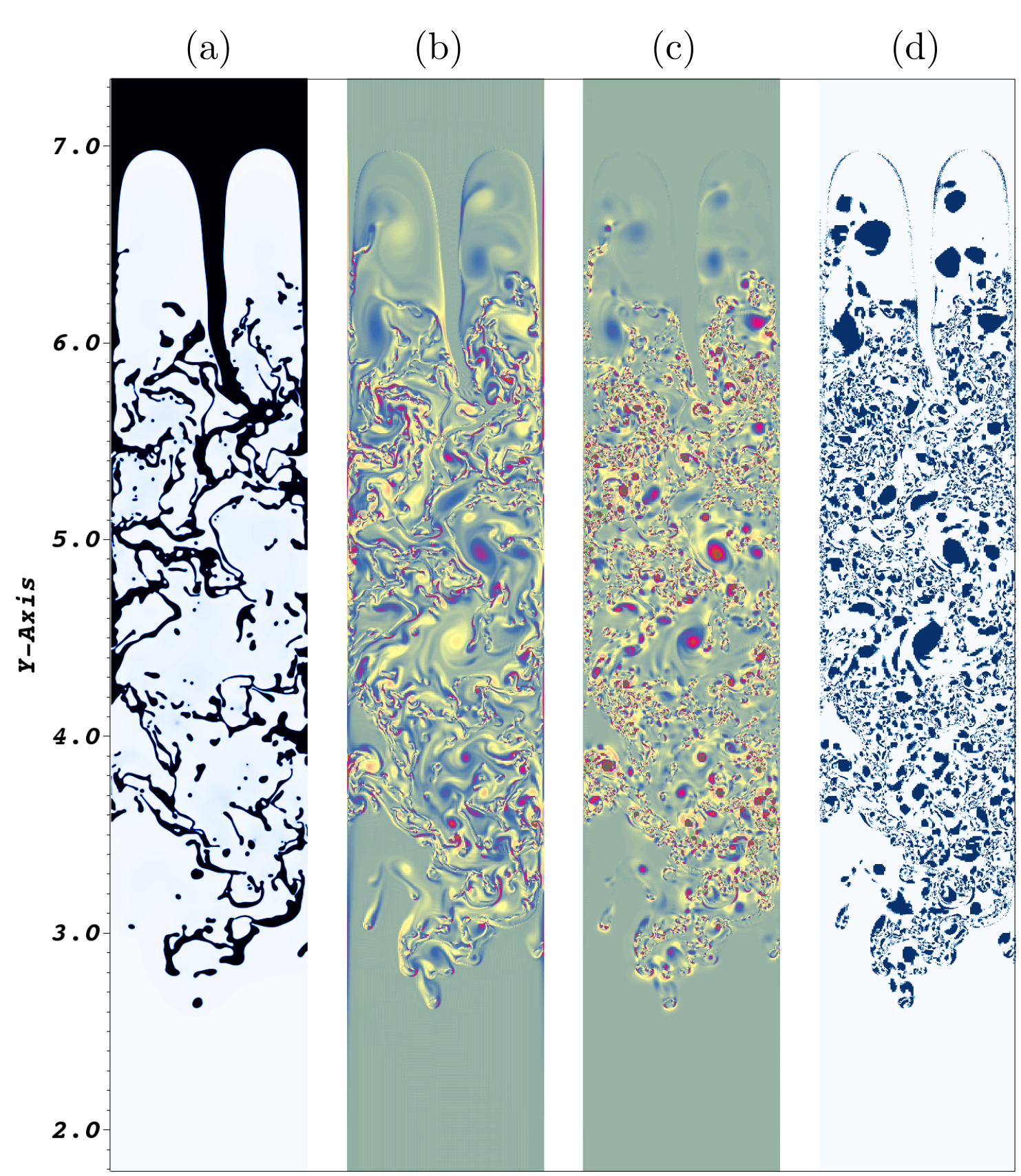}
	\caption{\textit{Rayleigh-Taylor instability in 2D} (\cref{subsec:rayleigh_taylor_2D}): turbulent dynamics in Rayleigh Taylor instability for $At = 0.82$ (density ratio of 0.1) at $t^\prime = t\sqrt{At} = 4.962$.  (a) Shows the interface; (b) shows vorticity (c) Q-criterion~\citep{Hunt1988} (d) Q-criterion threshold between $Q = 1.9$ and $Q = 2$ }
	\label{fig:rt2d_q}
\end{figure}

\begin{figure}[H]
	\centering	
	\includegraphics[width=\linewidth]{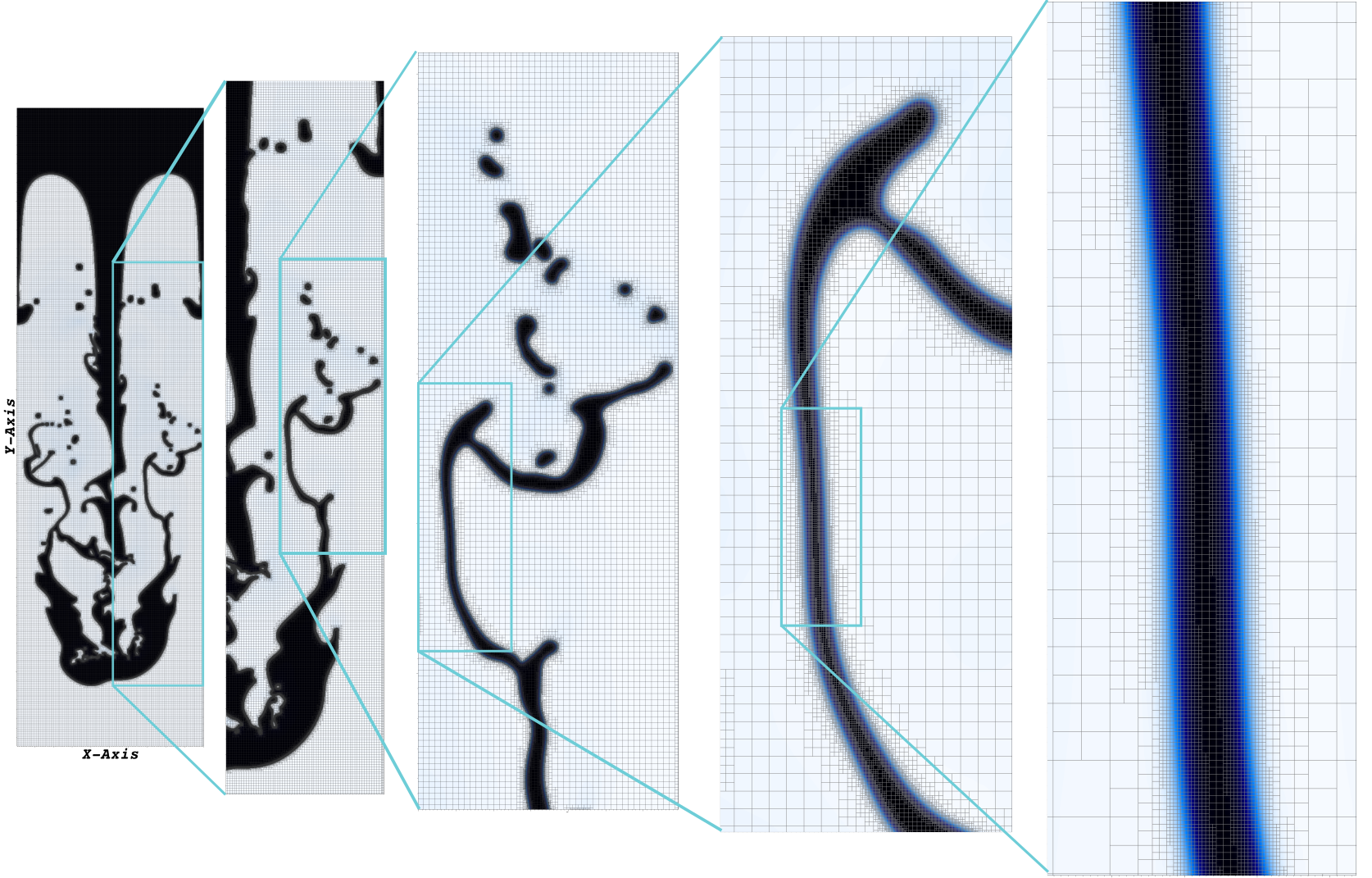}
	\caption{\textit{Rayleigh-Taylor instability in 2D} (\cref{subsec:rayleigh_taylor_2D}): Progressive zoomed in mesh for Rayleigh Taylor instability for $At = 0.82$ (density ratio of 0.1) at $t^\prime = t\sqrt{At} = 3.006$. These plots are zoomed insets of the domain near the interfacial instabilities.}
	\label{fig:rt2d_mesh}
\end{figure}

\subsection{3D simulations: Rayleigh-Taylor instability}
\label{subsec:rayleigh_taylor_3d}
We now deploy our framework in 3D and simulate the Rayleigh-Taylor instability in 3D using adaptive octree meshes.  The domain is set to be cuboid with dimensions $1 \times 8 \times 1$ in $\mathrm{x_1}$, $\mathrm{x_2}$, $\mathrm{x_3}$ directions respectively.  For the 3D simulations we choose 
the following initial condition for $\phi$ to describe the interface:
\begin{align}
\phi(\vec{x}) &= \tanh\left(\sqrt{2} \left[ \frac{x_2 - l - g\left(\vec{x}\right)}{Cn}\right]\right),\\ 
g(\vec{x}) &= 0.05 \left[\cos\left(2 \pi x_1 \right) + \cos\left(2 \pi x_3 \right)\right].
\label{eq:initialConditionRT}
\end{align}
Here $l$ is the location in the vertical direction for the interface, which in this case is chosen to be $\mathrm{x_2} = 6$. 
%Typical simulations in the literature choose a rectangular domain that only captures one wavelength of the initial condition (e.g.,~\citep{ Tryggvason1990}).  
%To illustrate the advantage of the adaptive octree framework, we choose to include four wavelengths in the initial condition, resulting in a larger domain.  
\Cref{subfig:rt3d_int_At15_snap_1} shows the initial condition. We use a $Cn=0.0075$ and $At = 0.15$.  For this lower $At$ number simulation, the effect of non-zero surface tension is important.  The non-dimensionalization follows the same logic as the 2D cases, with the Reynolds number set to $Re = 1000$, the Weber number ($We = \rho_c g D^2/\sigma$) set to $1000$, and viscosity ratio, $\nu_{+}/\nu_{-}$, set to 1. For an Atwood number $At = 0.15$ the density ratio is $\rho_{+}/\rho_{-} = 1.0/0.74$.  Similar to the 2D cases, the boundary conditions are no-slip for velocity and no-flux for $\phi$ and $\mu$ on all the walls.  This is different than the typical boundary condition of free slip on the side walls, which is used in the literature~\citep{Tryggvason1990,Liang2016,Jain2020}.  Another important thing to note is that generally zero surface tension is used, however in the current study as well as in \citet{Liang2016} the surface tension is small but finite.  We choose the no slip velocity on all walls to test the energy stability of the method, which requires the no-slip boundary conditions (see \citep{Shen2015,Khanwale2020,Khanwale2021} for the setup of the proof).    

%Due to the energy stability of the proposed numerical method we are able to take a reasonably large time-step size of $\delta t = 0.0025$.
We choose a time-step size of $\dt = 0.001$.  We refine near the interface to a level corresponding to element length of $8/2^{12}$, ensuring about six elements for resolving the diffuse interface, while the refinement away from the interface is $8/2^7$, we also refine near the walls with a level of $8/2^9$.  
%An algebraic multigrid linear solver with additive Schwarz based smoothers is setup for the linear solves in the Newton iterations (see \cref{subsec:newton_iter}).  
In \ref{subsec:app_rt2d} we provide a detailed description of the preconditioners and linear solvers along with the command-line arguments that are used. 

%As we are using an octree based adaptively refined mesh, the mesh is refined along the interface to the finest level.  It is coarsened away from the interface in way that 2:1 balancing is satisfied. 
\Cref{fig:rt3d_mesh_At15} shows the evolution of the interface along with the solution-adapted mesh.  We color the mesh with the order parameter value (blue for the heavy fluid and white for the light fluid) to show the  evolution of the system.  It is seen that as the interface evolves it deforms and expands, causing the mesh density to gradually increase.  This gradual growth helps with the efficiency of the simulation, since a uniform mesh for this case would be computationally expensive.  At the point with the largest interfacial area, the simulation has around 52 million nodes, corresponding to 152 million degrees of freedom for the largest matrix assembly (of velocity prediction).  The efficient and scalable implementation of the approach allows us to run this large scale simulation.  We run the simulation till $t = 6.8$ on SDSC Expanse with 32 nodes (4096 processes).  The simulation is subsequently run on TACC~\Frontera~with 80 Cascade Lake nodes (4480 processes).   We use this case to perform scalability analysis which is presented later in the scaling section (\cref{sec:scaling}).  
%\Cref{fig:rt_comparison} shows a qualitative comparison of the interface shape with the shape previously reported for the same density ratio in \citet{Tryggvason1990}.  Although, the initial conditions for the interface in our case (inverted Gaussian) is different than the initial conditions used in \cite{Tryggvason1990} (two-dimensional harmonic wave), the nature of the instability evolving from both of the them is similar where a blob of heavy fluid on top penetrates into light fluid at the bottom, setting up interfacial instabilities.  It can be clearly seem from \cref{fig:rt_comparison} that the shapes at this fairly evolved times are quite similar to each other qualitatively. 

\Cref{fig:rt3d_interface_At15} shows that the initial sinusoidal perturbation develops into penetrating plume of heavy fluid pushing down while the lighter fluid buckles and forms bubbles near the wall. The simulation maintains symmetry.  As different parts of the interface move in opposite directions, Kelvin-Helmholtz instabilities cause the plumes to roll up, which in turn causes causes mushroom-like structures to develop (see~\cref{subfig:rt3d_int_At15_snap_4}).  However, the Kelvin-Helmholtz instability is arrested a bit due to finite surface tension curbing breakup as seen in \cref{subfig:rt3d_int_At15_snap_6}. This behavior is also presented in \citet{Liang2016} who also used a CHNS model deployed using a Lattice Boltzman method. However, for the case of zero surface tension \citet{Jain2020} show breakup at the same time $t  = 7.8$.  Another feature that is unique to our simulations compared to the simulations in the literature is the development of the bubbles inside the domain instead of at the walls. This is due to the no-slip conditions of velocity on the side walls instead of free-slip.       
%When simulations only capture one wavelength of instability (e.g., those presented in \citep{ Liang2016, Mitchell2018, Jain2020}), the displacement of the lighter fluid near the walls forms bubble-like structures. We quantify the locations of the fronts of these plumes and bubbles as ``top" and ``bottom" fronts, respectively, in the 2D case.  
%In our case the bubbles are essentially upward \textit{spikes} of lighter fluid moving into heavier fluid that also undergo Kelvin-Helmholtz instabilities as they rise.  These upward spikes of lighter fluid displacing the heavier one at phase shifted locations relative to the \textit{spikes} of heavy fluid in~\cref{subfig:rt3d_int_At15_snap_4}.  
Although these two types of spikes (upwards/downwards) begin to develop in a checkerboard pattern that preserves symmetry, their dynamics are different due to the velocity differential that the two fronts face.  

The downward spikes undergo further deformation and we see the emergence of four protrusions from the mushroom structure (see~\cref{subfig:rt3d_int_At15_snap_6}) caused by the shear generated between the fluids. \citet{Liang2016} and~\citet{Jain2020}  also report four secondary filaments (but with breakup) in their simulations for the same $At$ number, although their simulations were for zero surface tension (i.e., dynamics similar to miscible systems).  

On the other hand, the mushroom structures from the upward spikes develop into long and thin circular films.  The upward spikes develop circular films adjacent to the wall "bubbles" (i.e., structures near the wall that the heavier fluid generates as it is displaced by the lighter one) interact with these bubbles to merge and form larger structures (see~\cref{subfig:rt3d_int_At15_snap_6}). While the central plumes have little-to-no interaction with the wall, the bubbles continue to rise and ultimately collide with the top wall. 

We compare the front location of the downward spike with \citet{Liang2016} in \cref{fig:RT_3d_comp}. We see a slight deviation from \citet{Liang2016} which we attribute to the no-slip boundary conditions. Panel~(a) from~\cref{fig:RT_3d_energy} shows the decay of the energy functional with respect to time demonstrating energy stability on the fully discrete 3D AMR meshes.  Panel~(b) from~\cref{fig:RT_3d_energy} shows the drift of mass which is expected to be as close to zero as possible.  However, we see that for long term behavior the mass drift increases, we attribute this behavior to the interpolation operations during coarsening of the adaptive mesh refinement.

%Another important difference between the upward and downward spikes is their rate of growth.  \Crefrange{subfig:rt3d_int_At15_snap_5}{subfig:rt3d_int_At15_snap_8} show that the fronts of the upward spikes move slower than the fronts of the downward spikes due to the density differential.  \Cref{subfig:rt3d_int_At15_snap_9} shows that the mushroom structures collide with top and bottom walls, which then leads to further breakup that creates the conditions for chaotic mixing. To the best of our knowledge, this is the first analysis in the literature of the dynamics for multiple wave single-mode instabilities. 

\begin{figure}[]
	\centering	
	\begin{tabular}{p{0.18\textwidth}p{0.18\textwidth}p{0.18\textwidth}}
		%%%%%%%%%%%%%%%%%%%%%% 1 %%%%%%%%%%%%%%%%%%%%%%%%%%%%%%%%%%%
		\subfigure [$t = 0.0$] {
			\includegraphics[width=\linewidth]{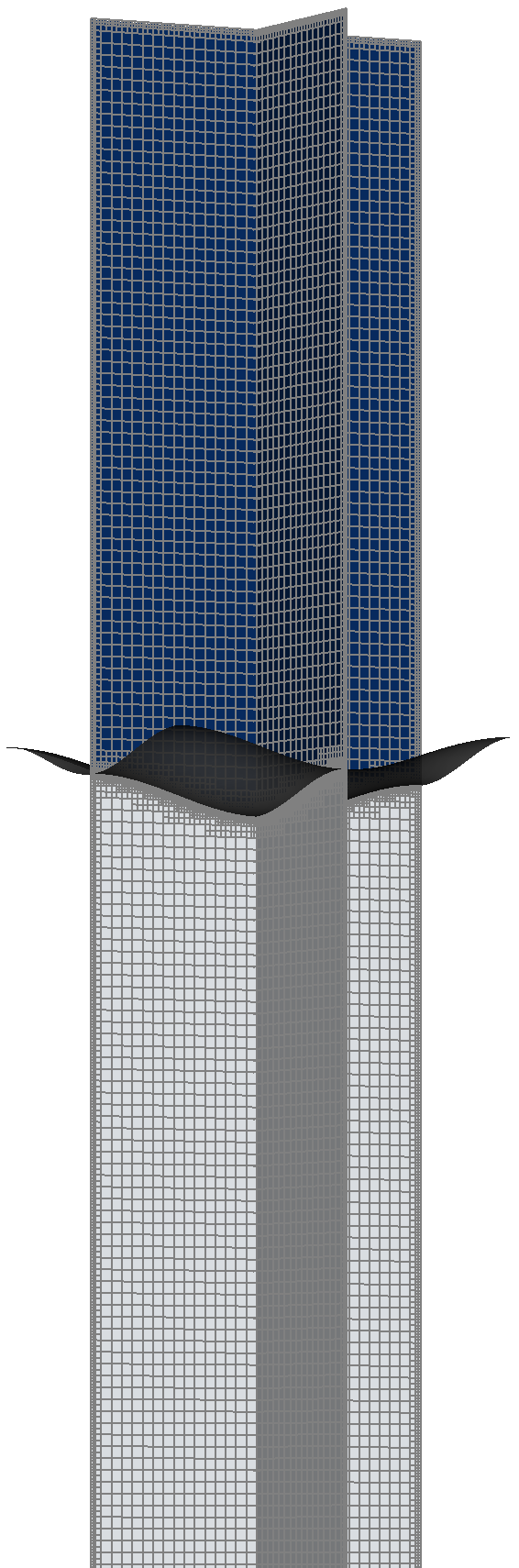}
			\label{subfig:rt3d_mesh_At15_snap_1}
		} &
		%%%%%%%%%%%%%%%%%%%%%% 2 %%%%%%%%%%%%%%%%%%%%%%%%%%%%%%%%%%%
		\subfigure [$t = 2$] {
			\includegraphics[width=\linewidth]{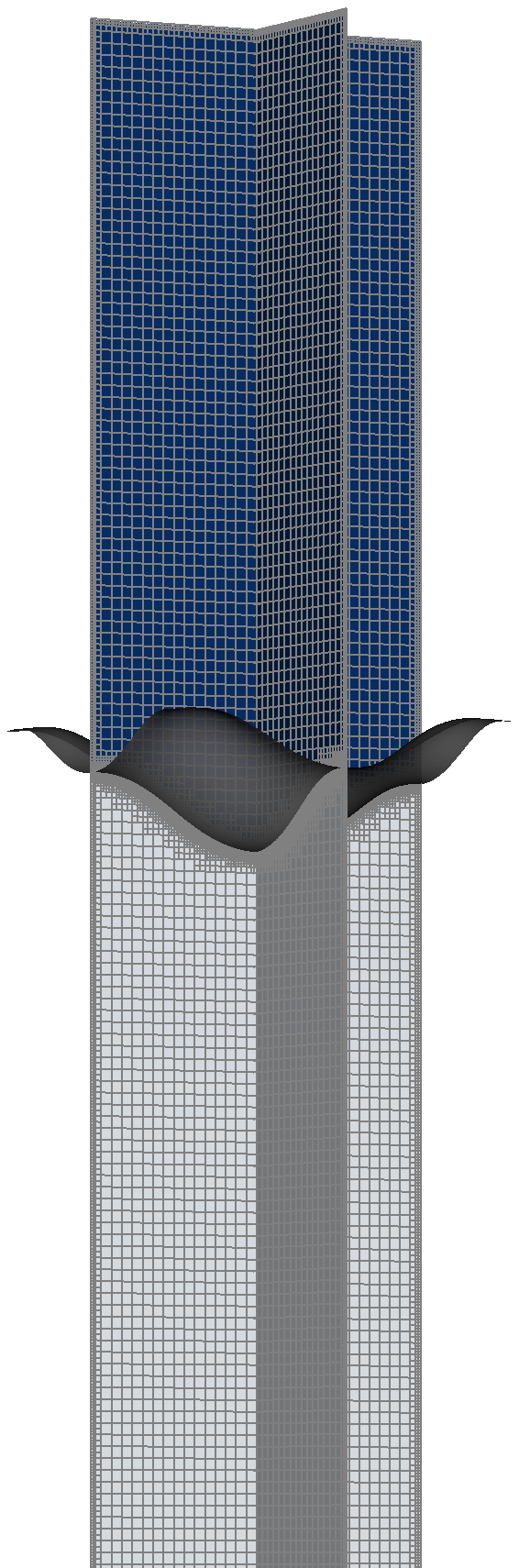}
			\label{subfig:rt3d_mesh_At15_snap_2}
		} & 
		%%%%%%%%%%%%%%%%%%%%%% 3 %%%%%%%%%%%%%%%%%%%%%%%%%%%%%%%%%%%
		\subfigure [$t = 3.92$] {
			\includegraphics[width=\linewidth]{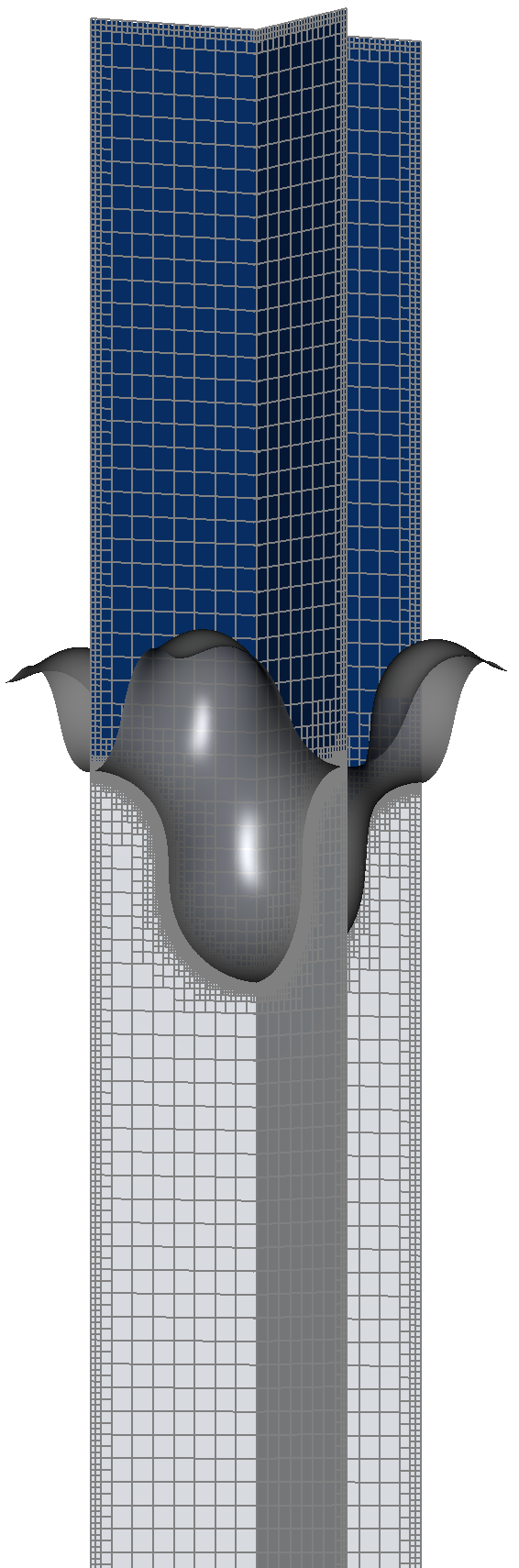}
			\label{subfig:rt3d_mesh_At15_snap_3}
		} \\
		
		%%%%%%%%%%%%%%%%%%%%%% 4 %%%%%%%%%%%%%%%%%%%%%%%%%%%%%%%%%%%
		\subfigure [$t = 5.84$] {
			\includegraphics[width=\linewidth]{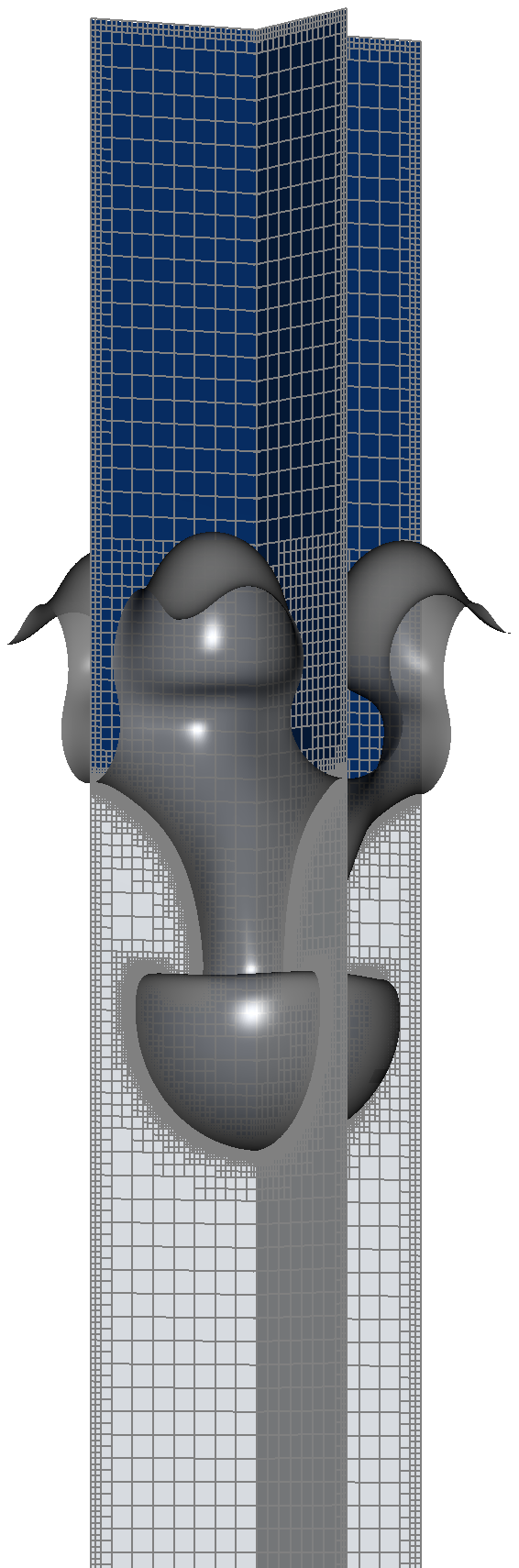}
			\label{subfig:rt3d_mesh_At15_snap_4}
		} &
		%%%%%%%%%%%%%%%%%%%%%% 5 %%%%%%%%%%%%%%%%%%%%%%%%%%%%%%%%%%%
		\subfigure [$t = 6.8$] {
			\includegraphics[width=\linewidth]{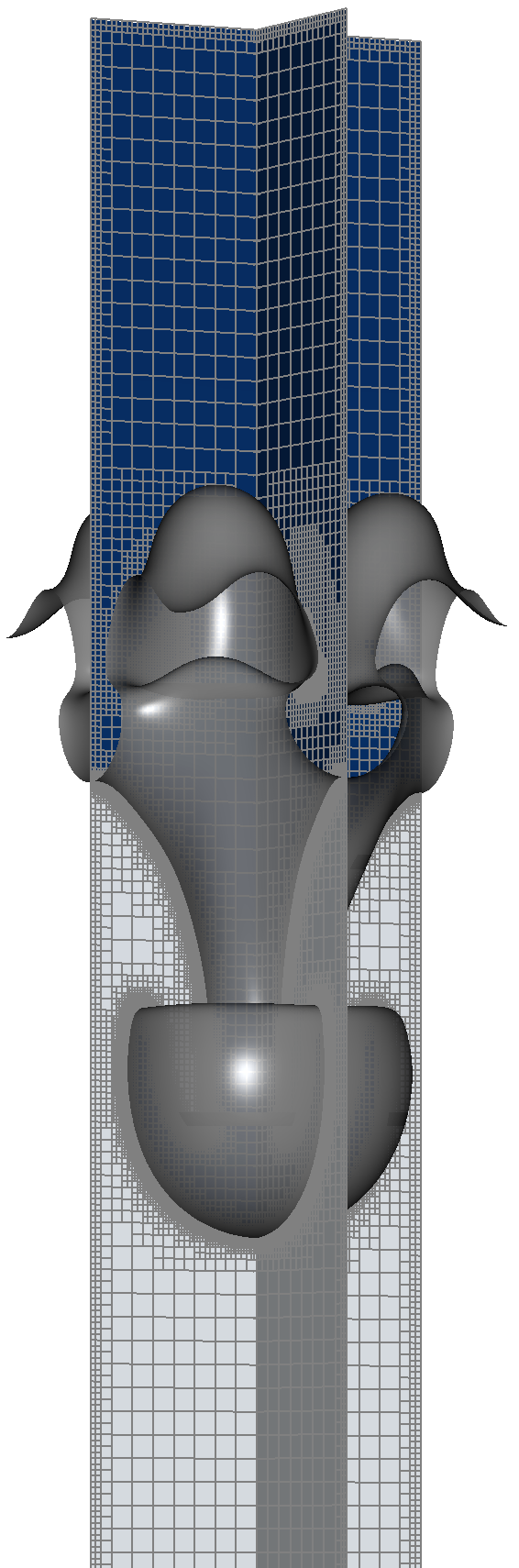}
			\label{subfig:rt3d_mesh_At15_snap_5}
		} & 
		%%%%%%%%%%%%%%%%%%%%%% 6 %%%%%%%%%%%%%%%%%%%%%%%%%%%%%%%%%%%
		\subfigure [$t = 7.84$] {
			\includegraphics[width=\linewidth]{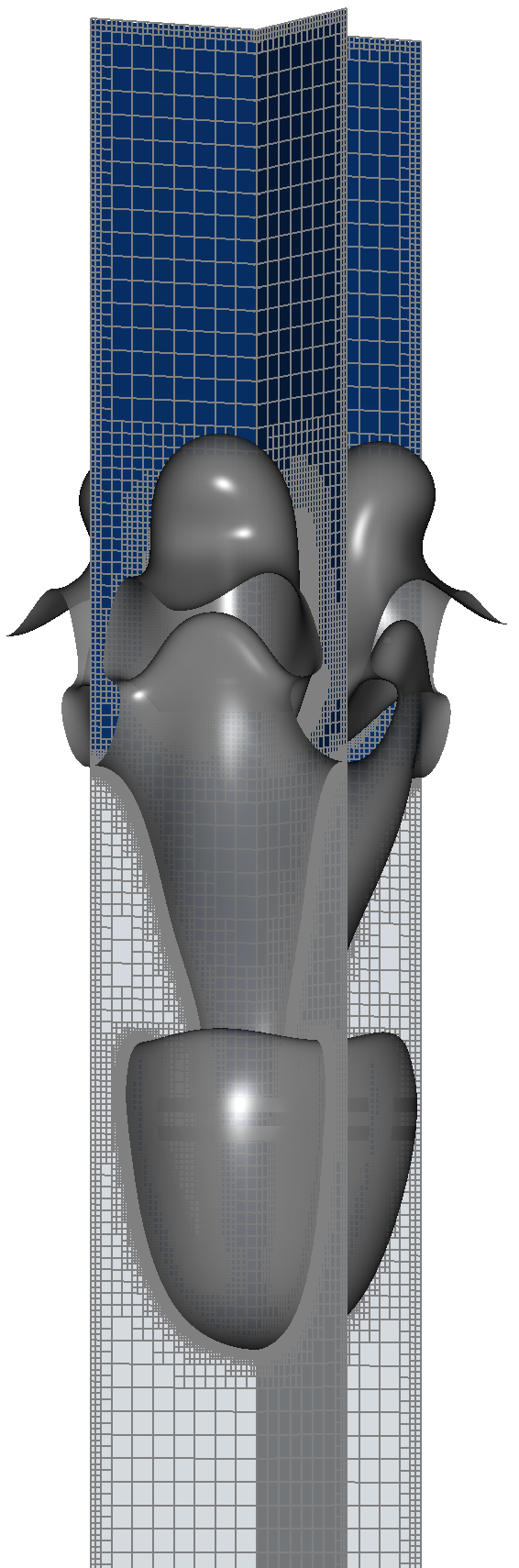}
			\label{subfig:rt3d_mesh_At15_snap_6}
		} %\\
%		
%		%%%%%%%%%%%%%%%%%%%%%% 7 %%%%%%%%%%%%%%%%%%%%%%%%%%%%%%%%%%%
%		\subfigure [$t = 9.375$] {
%			\includegraphics[width=\linewidth]{Figures/RT3D/At0dot15/case2/mesh_evol0021}
%			\label{subfig:rt3d_mesh_At15_snap_7}
%		} &
%		%%%%%%%%%%%%%%%%%%%%%% 8 %%%%%%%%%%%%%%%%%%%%%%%%%%%%%%%%%%%
%		\subfigure [$t = 10$] {
%			\includegraphics[width=\linewidth]{Figures/RT3D/At0dot15/case2/mesh_evol0026}
%			\label{subfig:rt3d_mesh_At15_snap_8}
%		} & 
%		%%%%%%%%%%%%%%%%%%%%%% 9 %%%%%%%%%%%%%%%%%%%%%%%%%%%%%%%%%%%
%		\subfigure [$t = 11.25$] {
%			\includegraphics[width=\linewidth]{Figures/RT3D/At0dot15/case2/mesh_evol0035}
%			\label{subfig:rt3d_mesh_At15_snap_9}
%		} \\
	\end{tabular}
	\caption{\textit{Rayleigh-Taylor instability in 3D} (\cref{subsec:rayleigh_taylor_3d}): Snapshots of the mesh at various time-points in the simulation for Rayleigh-Taylor instability for $At = 0.15$.  The figures show the mesh on two slices to illustrate the refinement around the interface of two fluids represented by the gray iso-surface of $\phi = 0$. The phase field $\phi$ values color the mesh, where blue represents heavy fluid and white represents light fluid.  Here $t $(-) is the non-dimensional time.}
	\label{fig:rt3d_mesh_At15}
\end{figure}

\begin{figure}[]
	\centering	
	\begin{tabular}{p{0.25\textwidth}p{0.25\textwidth}p{0.25\textwidth}}
		%%%%%%%%%%%%%%%%%%%%%% 1 %%%%%%%%%%%%%%%%%%%%%%%%%%%%%%%%%%%
		\subfigure [$t = 0.0$] {
			\includegraphics[width=\linewidth]{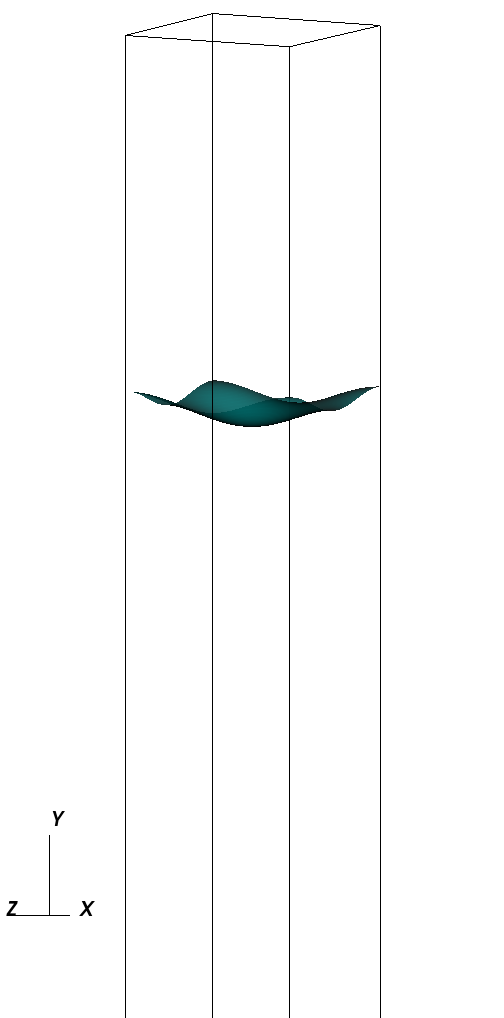}
			\label{subfig:rt3d_int_At15_snap_1}
		} &
		%%%%%%%%%%%%%%%%%%%%%% 2 %%%%%%%%%%%%%%%%%%%%%%%%%%%%%%%%%%%
		\subfigure [$t = 2$] {
			\includegraphics[width=\linewidth]{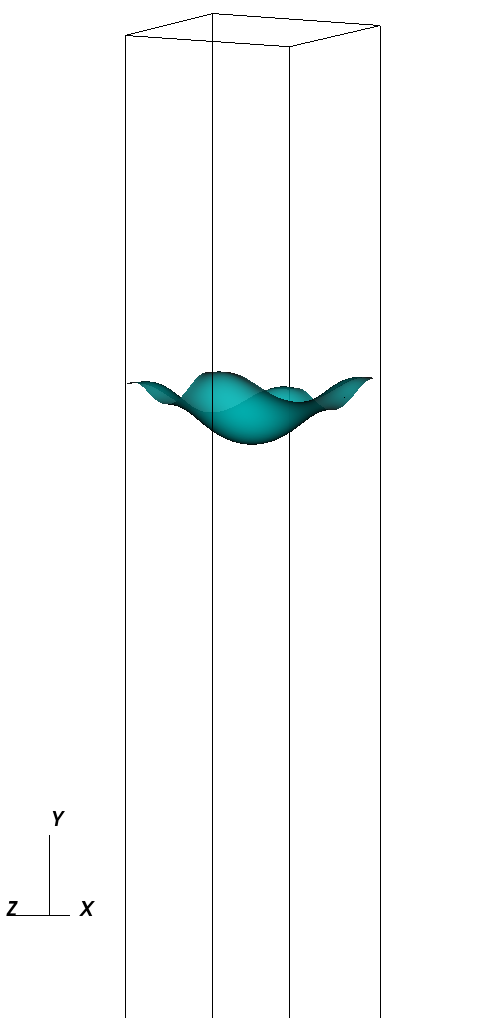}
			\label{subfig:rt3d_int_At15_snap_2}
		} & 
		%%%%%%%%%%%%%%%%%%%%%% 3 %%%%%%%%%%%%%%%%%%%%%%%%%%%%%%%%%%%
		\subfigure [$t = 3.84$] {
			\includegraphics[width=\linewidth]{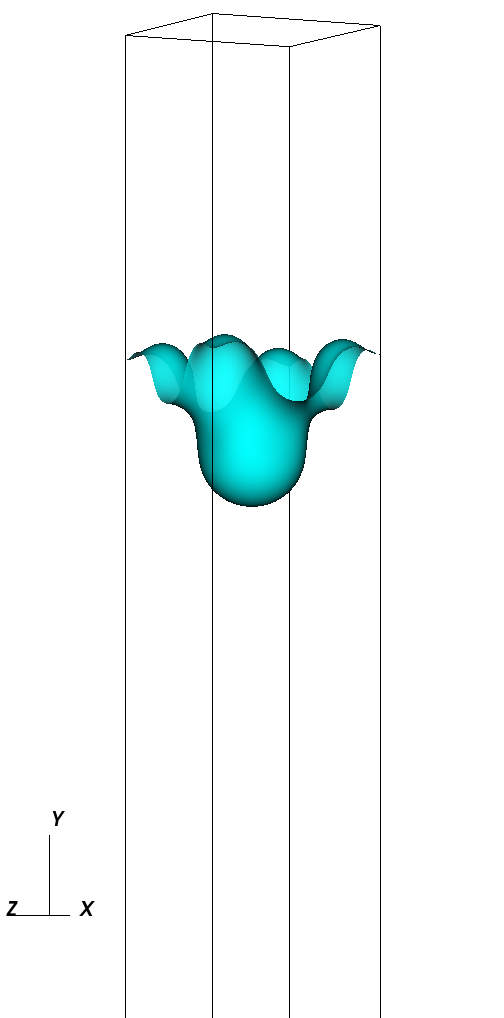}
			\label{subfig:rt3d_int_At15_snap_3}
		} \\
		
		%%%%%%%%%%%%%%%%%%%%%% 4 %%%%%%%%%%%%%%%%%%%%%%%%%%%%%%%%%%%
		\subfigure [$t = 5.76$] {
			\includegraphics[width=\linewidth]{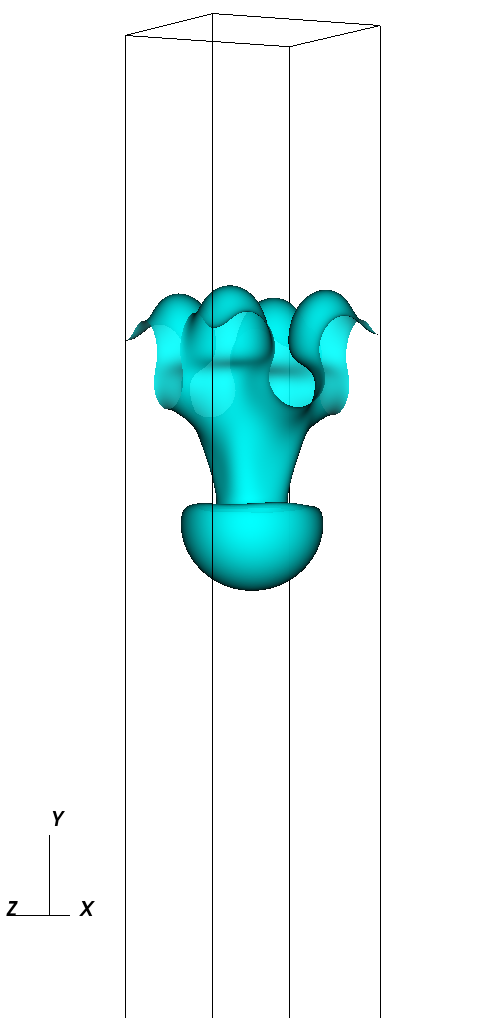}
			\label{subfig:rt3d_int_At15_snap_4}
		} &
		%%%%%%%%%%%%%%%%%%%%%% 5 %%%%%%%%%%%%%%%%%%%%%%%%%%%%%%%%%%%
		\subfigure [$t = 6.88$] {
			\includegraphics[width=\linewidth]{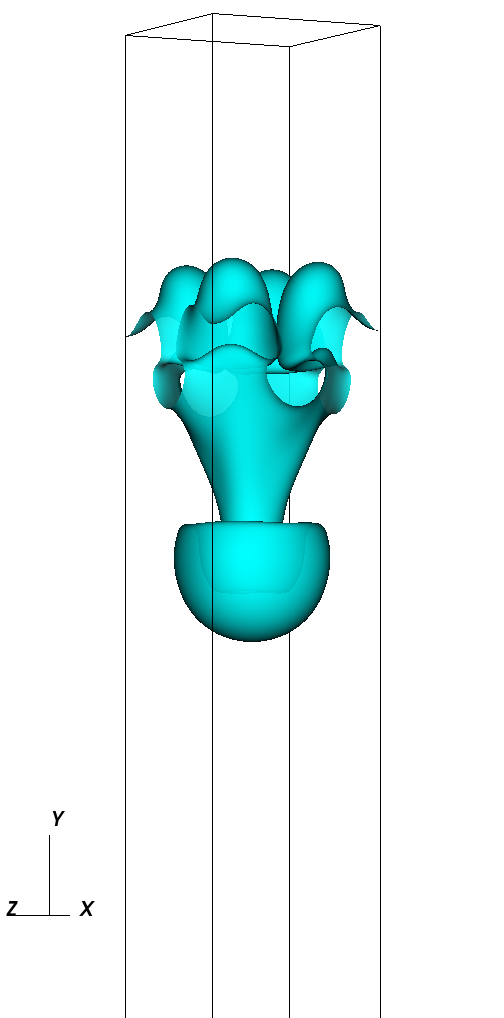}
			\label{subfig:rt3d_int_At15_snap_5}
		} & 
		%%%%%%%%%%%%%%%%%%%%%% 6 %%%%%%%%%%%%%%%%%%%%%%%%%%%%%%%%%%%
		\subfigure [$t = 7.84$] {
			\includegraphics[width=\linewidth]{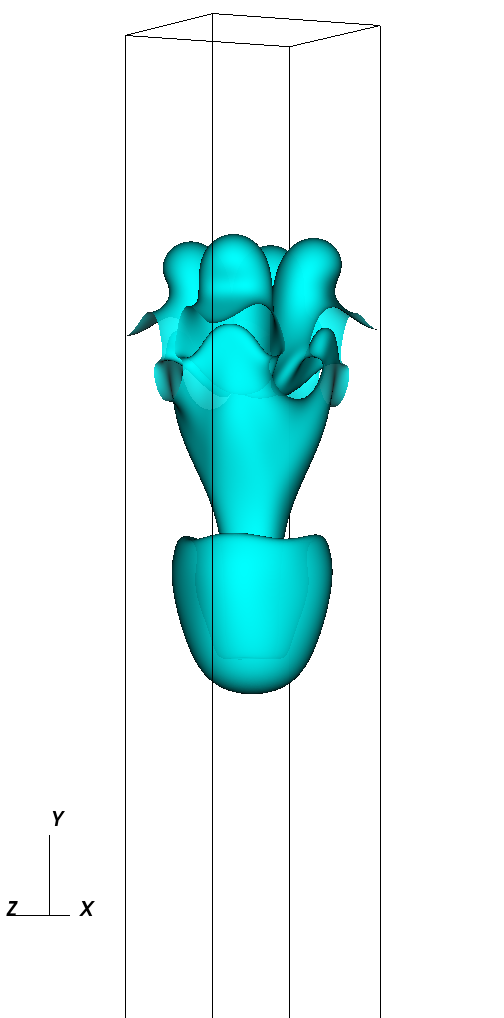}
			\label{subfig:rt3d_int_At15_snap_6}
		} %\\
		
%		%%%%%%%%%%%%%%%%%%%%%% 7 %%%%%%%%%%%%%%%%%%%%%%%%%%%%%%%%%%%
%		\subfigure [$t = 9.375$] {
%			\includegraphics[width=\linewidth]{Figures/RT3D/At0dot15/interface_evol0070}
%			\label{subfig:rt3d_int_At15_snap_7}
%		} &
%		%%%%%%%%%%%%%%%%%%%%%% 8 %%%%%%%%%%%%%%%%%%%%%%%%%%%%%%%%%%%
%		\subfigure [$t = 10$] {
%			\includegraphics[width=\linewidth]{Figures/RT3D/At0dot15/interface_evol0080}
%			\label{subfig:rt3d_int_At15_snap_8}
%		} & 
%		%%%%%%%%%%%%%%%%%%%%%% 9 %%%%%%%%%%%%%%%%%%%%%%%%%%%%%%%%%%%
%		\subfigure [$t = 11.25$] {
%			\includegraphics[width=\linewidth]{Figures/RT3D/At0dot15/interface_evol0089}
%			\label{subfig:rt3d_int_At15_snap_9}
%		} 
	\end{tabular}
	\caption{\textit{Rayleigh-Taylor instability in 3D} (\cref{subsec:rayleigh_taylor_3d}): Zoomed in snapshots of the zero isosurface of $\phi$ which represents the interface at various time-points in the simulation for Rayleigh-Taylor instability for $At = 0.15$.  Here $t $(-) is the non-dimensional time.}
	\label{fig:rt3d_interface_At15}
\end{figure}

\tikzexternaldisable
\begin{figure}[H]
	\centering
	\begin{tikzpicture}
	\begin{axis}[width=0.45\linewidth,scaled y ticks=true,xlabel={Time (-)},ylabel={$H_s$},legend style={nodes={scale=0.65, transform shape}}, 
	xmin=0, xmax=7.8,
	ymin=-0.2,ymax=2,  
	%xtick={0,1,2,3}
	%title={(a)},
	%,  title={\textbf{Decay of energy functional}}
	cycle list/Set1,
	% combine it with 'mark list*':
	cycle multiindex* list={
		mark list*\nextlist
		Set1\nextlist
	},
	legend style={nodes={scale=0.95, transform shape}, row sep=2.5pt},
	legend entries={
		\citet{Liang2016},
		present study
	},
	legend pos= north west,
	legend image post style={scale=1.0}
	]
	\addplot table [x={time},y={spike},col sep=comma]{Figures/RT_3D/rt3d_Liang_PRE_Spike_position.csv};
	\addplot +[mark = none, filter discard warning=false, unbounded coords=discard, line width=0.35mm] table [x={time},y={spike},col sep=comma] {Figures/RT_3D/boundFiles.csv};
	\end{axis}
	\end{tikzpicture}
	%Here ends the furst plot
	\hskip 5pt
	%Here begins the 3d plot
%	\begin{tikzpicture}
%	\begin{axis}[width=0.45\linewidth,scaled y ticks=true,xlabel={Time (-)},ylabel={$H_b$},legend style={nodes={scale=0.65, transform shape}}, 
%	%ymin=-5e-2,ymax=5e-2, 
%	xmin=0, xmax=8, 
%	%xtick={0,1,2,3}, 
%	title={(b)},
%	cycle list/Set1,
%	% combine it with 'mark list*':
%	cycle multiindex* list={
%		mark list*\nextlist
%		Set1\nextlist
%	},
%	]
%	\addplot +[mark = none, unbounded coords=discard]table [x={time},y={bubble},col sep=comma] {Figures/RT_3D/boundFiles.csv};
%	\end{axis}
%	\end{tikzpicture}
	\caption{Normalized spike amplitude 
		%and (b) normalized bubble amplitude 
		for the case of $Re = 1000$ and $We = 1000$ in the simulation of the 3D Rayleigh-Taylor instability (\cref{subsec:rayleigh_taylor_3d}).
	}
	\label{fig:RT_3d_comp}
\end{figure}
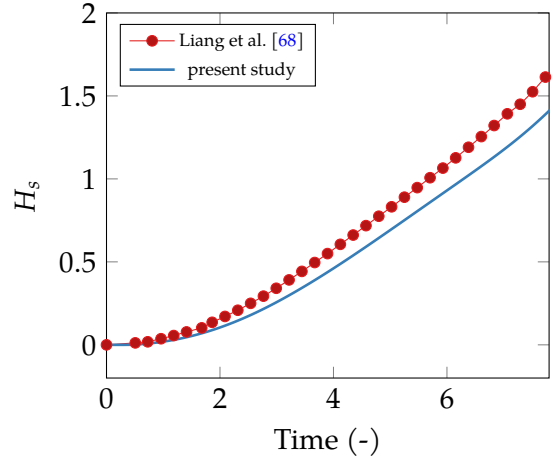

\tikzexternalenable
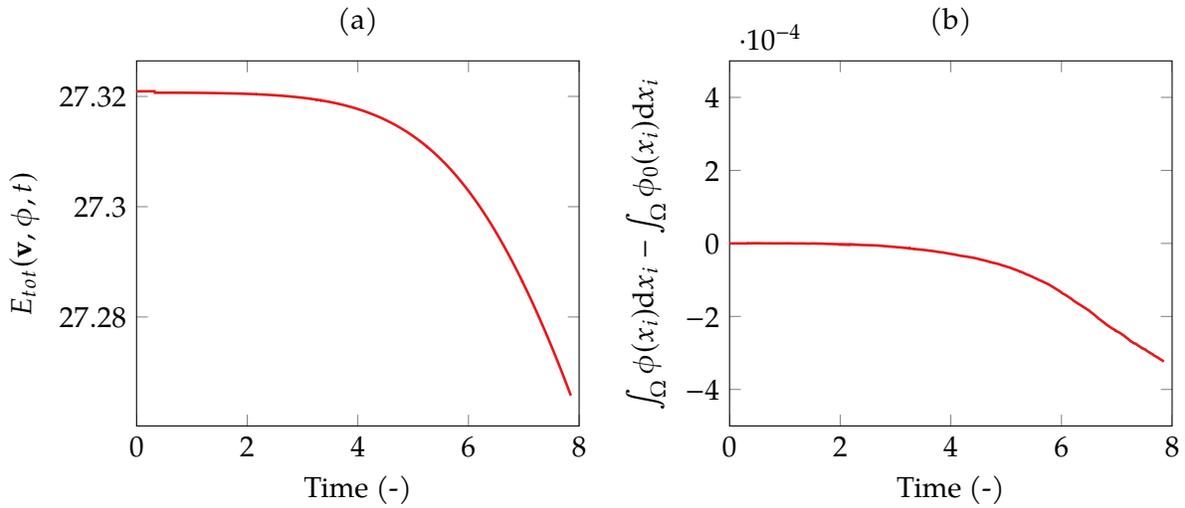
\begin{figure}[H]
	\centering
	\begin{tikzpicture}
	\begin{axis}[width=0.45\linewidth,scaled y ticks=true,xlabel={Time (-)},ylabel={$E_{tot}(\vec{v},\phi,t)$},legend style={nodes={scale=0.65, transform shape}}, xmin=0, xmax=8, %xtick={0,1,2,3}
	,  title={(a)},
		cycle list/Set1,
	% combine it with 'mark list*':
	cycle multiindex* list={
		mark list*\nextlist
		Set1\nextlist
	}
	]
	\addplot +[mark = none, each nth point=20, filter discard warning=false, unbounded coords=discard, line width=0.35mm] table [x={time},y={TotalEnergy},col sep=comma] {Figures/RT_3D/Energy_data.csv};
	\end{axis}
 	\end{tikzpicture}
	%Here ends the furst plot
	\hskip 5pt
	%Here begins the 3d plot
	\begin{tikzpicture}
	\begin{axis}[width=0.45\linewidth,scaled y ticks=true,xlabel={Time (-)},ylabel={$\int_{\Omega} \phi (x_i)\mathrm{d}x_i - \int_{\Omega} \phi_{0} (x_i)\mathrm{d}x_i$},legend style={nodes={scale=0.65, transform shape}}, 
	ymin=-5e-4,ymax=5e-4, 
	xmin=0, xmax=8, 
	title={(b)},
		cycle list/Set1,
	% combine it with 'mark list*':
	cycle multiindex* list={
		mark list*\nextlist
		Set1\nextlist
	},
	]
	\addplot +[mark = none, each nth point=20, filter discard warning=false, unbounded coords=discard, line width=0.35mm] table [x={time},y={TotalPhiMinusInit},col sep=comma] {Figures/RT_3D/Energy_data.csv};
	\end{axis}
	\end{tikzpicture}
	\caption{ (a) Decay of the energy functional; (b) Mass conservation for the case of $Re = 3000$ and $We = 1000$
		in the simulation of the 3D Rayleigh-Taylor instability (\cref{subsec:rayleigh_taylor_3d}) 
	}
	\label{fig:RT_3d_energy}
\end{figure}
\tikzexternaldisable

%!TEX root = main.tex
\section{Scaling}
\label{sec:scaling}
In this section, we show preliminary scaling results of our solver. In order to conduct the scaling we choose the 3D Rayleigh Taylor instability problem (\cref{subsec:rayleigh_taylor_3d}).  The domain is a cube with dimensions $8 \times 1 \times 1$. We chose two different meshes to conduct the scaling study with varying level of refinements. The overall mesh is characterized by three different regions of refinement: bulk/base level (refinement level in the bulk), wall level (refinement level at the wall), and interface refinement level (refinement level at the interface, indicated by $|\phi| \leq 0.95$). A refinement level of $k$, would mean a mesh resolution of $8/2^{k}$. ~\Cref{tab:refineLevel} shows the different refinement levels. As the simulation progress, the region close to the interface, $|\phi| \leq 0.95$ is refined, while regions away from the interface are coarsened. We report the scaling behavior conducted over 5 time-steps. 
\begin{table}[]
\centering
\begin{tabular}{|c|c|c|c|c|}
\hline
\multirow{2}{*}{} & \multicolumn{3}{c|}{\textbf{Refinement level}} & \multirow{2}{*}{\textbf{\begin{tabular}[c]{@{}c@{}}Mesh\\ size\end{tabular}}} \\ \cline{2-4}
 & \textbf{Bulk} & \textbf{Wall} & \textbf{Interface} &  \\ \hline
Mesh \textsc{M1} & 6 & 7 & 9 & 5 M \\ \hline
Mesh \textsc{M2} & 8 & 10 & 13 & 30 M \\ \hline
\end{tabular}
\caption{The level of refinement for two different meshes used to conduct the scaling studies.}
\label{tab:refineLevel}
\end{table}

\Cref{fig:strong_scaling_breakup} shows the strong scaling  behavior of the individual solvers. The Cahn-Hillard (CH), the velocity prediction, and the pressure Poisson (PP) updates were solved using~\petsc's~Algebraic Multigrid (AMG), whereas the velocity update  was solved using conjugate gradient with Additive Schwarz preconditioner (see~\cref{sec:app_linear_solve}). We chose AMG for the first three updates on account of its $\mathcal{O}(N)$ complexity. Overall we see a good scaling behavior until the AMG setup cost begins to dominate.  Previous studies have demonstrated the inability of AMG to scale to a very large number of processors~\citep{saurabh2021industrial,sundar2012parallel} due to the associated communication and setup costs at such scales. This becomes particularly expensive in the case of adaptive meshes, as the AMG setup needs to be executed each time re-meshing is performed.  We see that the velocity prediction suffers the most from AMG as it has the largest matrix size among the equations solved with AMG. We note that ~\petsc's~native AMG is degree-of-freedom (dof) agnostic, which makes the AMG cost grow with increasing dof. Geometric Multigrid (GMG) tends to eliminate such cost ~\cite{sundar2012parallel}, and is particularly beneficial for hierarchical data structures like octrees.   

Panels (a), (b), and (c) of~\cref{fig:strong_scaling_breakup} show the strong scaling behavior of the Cahn-Hilliard solve (\cref{defn:weak_VMS_disc_CH}), velocity prediction step (\cref{defn:weak_VMS_disc}), and pressure Poisson step (\cref{defn:weak_VMS_disc_vel_update}), respectively.  In panels (a), (b), and (c) of~\cref{fig:strong_scaling_breakup} we see good scaling behavior for the first three points and then it starts to taper off as the communication and setup cost of~\petsc's~native AMG increases.  This behavior is illustrated in the histogram of the breakdown of cost for each solver in~\cref{fig:componentAnalysis}: we can see that the preconditioner (PC) setup cost increases with increasing processors for the Cahn-Hilliard solve, velocity prediction, and pressure Poisson; all of these updates are using AMG.  On the other hand in panel (d) of~\cref{fig:strong_scaling_breakup} we can see that velocity update solve scales well, even beyond the first three points as it is using a simple conjugate gradient solver with an Additive Schwarz preconditioner, which has relatively low communication and setup costs. This point is highlighted in~\cref{fig:componentAnalysis}, which shows a histogram of the cost breakdown for the pressure Poisson solve with varying processor number. 

\Cref{fig:strong_scaling_total} shows the scaling behavior of total time to solution as a function of the number of processors. We see a good scaling behavior until the AMG setup costs begin to dominate; this effect is  shown in ~\cref{fig:strong-RT3D}, in which  the cost for each major component of the solution is tabulated as a function of the number of processors. A constant parallel cost would indicate ideal strong scaling. Overall, we see a constant growing cost with increase in processor with the AMG cost starting to dominate for $\geq 3K$ processors. 

%!TEX root = ../../main.tex
\tikzexternaldisable
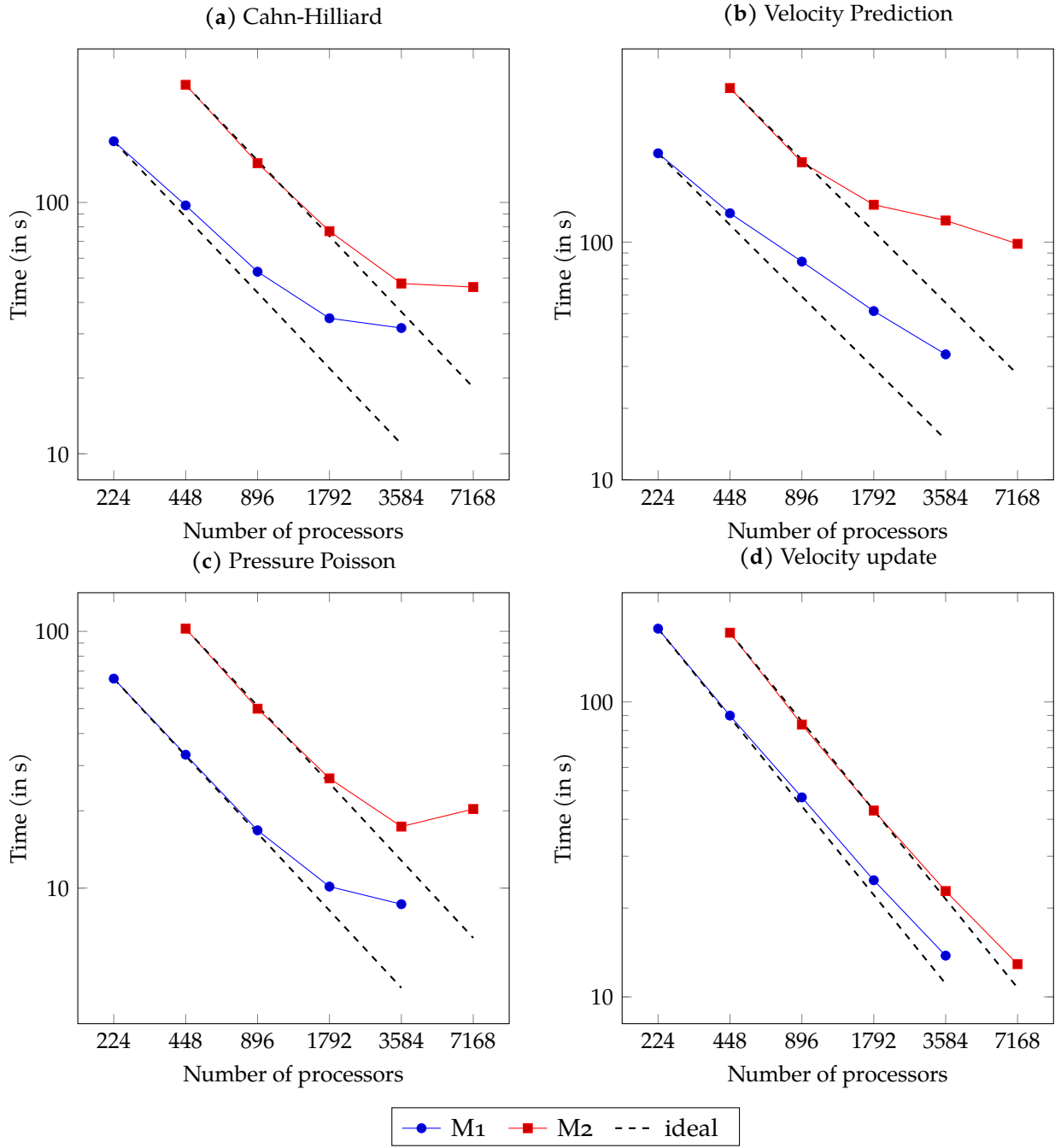
\begin{figure}[H]
	\centering
	\begin{tikzpicture}
	\begin{groupplot}[
	group style={group size=2 by 2,ylabels at=edge left,vertical sep=50pt, horizontal sep=50pt},
	ylabel style={text height=0.02\textwidth,inner ysep=-2pt},
	xlabel= Number of processors, % Set the labels
	height=0.5*\linewidth,width=0.5\linewidth,/tikz/font=\small,
	]
	\nextgroupplot[xmode=log,log basis x={2},
	ymode=log,
	xtick={224, 448, 896, 1792, 3584,7168},
	xticklabels={$224$,$448$,$896$, $1792$, $3584$, $7168$},
	ytick={10,100},
	yticklabels={$10$,$100$},
	title={\textbf{(a)} Cahn-Hilliard},
	ylabel= Time (in s),
	]%
	\addplot
	table[x expr={\thisrow{proc}))},y expr={\thisrow{Block1-ch} + \thisrow{Block2-ch}},col sep=space]{Figures/ScalingData/Scaling/BlockTimingNewMesh1.txt};
	\addplot
	table[x expr={\thisrow{proc}))},y expr={\thisrow{Block1-ch} + \thisrow{Block2-ch}},col sep=space]{Figures/ScalingData/Scaling/BlockTimingNewMesh2.txt};

	\addplot [no markers, color=black, thick, dashed ] coordinates {
		(224,  175.2748)
		(3584, 10.9375)
	};
	\coordinate (top) at (rel axis cs:0,1);% coordinate at top of the first plot
	\addplot [no markers, color=black, thick, dashed ] coordinates {
		(448,  293.86258)
		(7168, 18.36641125)
	};
	\coordinate (top) at (rel axis cs:0,1);% coordinate at top of the first plot
	\nextgroupplot[xmode=log,log basis x={2},
	ymode=log,
xtick={224, 448, 896, 1792, 3584,7168},
	xticklabels={$224$,$448$,$896$, $1792$, $3584$, $7168$},
	ymin = 10,
	ytick={10,100},
	yticklabels={$10$,$100$},
	title={\textbf{(b)} Velocity Prediction},
	ylabel= Time (in s),
	]%
	\addplot
	table[x expr={\thisrow{proc}))},y expr={\thisrow{Block1-ns} + \thisrow{Block2-ns}},col sep=space]{Figures/ScalingData/Scaling/BlockTimingNewMesh1.txt};
	\addplot
	table[x expr={\thisrow{proc}))},y expr={\thisrow{Block1-ns} + \thisrow{Block2-ns}},col sep=space]{Figures/ScalingData/Scaling/BlockTimingNewMesh2.txt};
	\addplot [no markers, color=black, thick, dashed ] coordinates {
	    (224,  236.296)
	    (3584, 14.7685)
	};
	% 	\SubLabel{plot:SSTimerelative};
	\coordinate (bot) at (rel axis cs:1,0);% coordinate at bottom of the last plot
	\addplot [no markers, color=black, thick, dashed ] coordinates {
	    (448,  444.3416)
	    (7168, 27.77135)
	};
	% 	\SubLabel{plot:SSTimerelative};
	\coordinate (bot) at (rel axis cs:1,0);% coordinate at
	\nextgroupplot[xmode=log,log basis x={2},
	ymode=log,
	xtick={224, 448, 896, 1792, 3584,7168},
	xticklabels={$224$,$448$,$896$, $1792$, $3584$, $7168$},
	ytick={10,100},
	yticklabels={$10$,$100$},
	title={\textbf{(c)} Pressure Poisson},
	ylabel= Time (in s),
	]%
	\addplot
	table[x expr={\thisrow{proc}))},y expr={\thisrow{Block1-pp} + \thisrow{Block2-pp}},col sep=space]{Figures/ScalingData/Scaling/BlockTimingNewMesh1.txt};
	\addplot
	table[x expr={\thisrow{proc}))},y expr={\thisrow{Block1-pp} + \thisrow{Block2-pp}},col sep=space]{Figures/ScalingData/Scaling/BlockTimingNewMesh2.txt};
	\addplot [no markers, color=black, thick, dashed ] coordinates {
		(224,  65.4459)
		(3584, 4.0903)
	};
	\addplot [no markers, color=black, thick, dashed ] coordinates {
		(448,  102.433)
		(7168, 6.40)
	};
	\nextgroupplot[xmode=log,log basis x={2},
	ymode=log,
	xtick={224, 448, 896, 1792, 3584,7168},
	xticklabels={$224$,$448$,$896$, $1792$, $3584$, $7168$},
	ytick={10,100},
	yticklabels={$10$,$100$},
	title={\textbf{(d)} Velocity update},
	ylabel= Time (in s),
	]%
	\addplot
	table[x expr={\thisrow{proc}))},y expr={\thisrow{Block1-vu} + \thisrow{Block2-vu}},col sep=space]{Figures/ScalingData/Scaling/BlockTimingNewMesh1.txt}; \label{plot:M1}
	\addplot
	table[x expr={\thisrow{proc}))},y expr={\thisrow{Block1-vu} + \thisrow{Block2-vu}},col sep=space]{Figures/ScalingData/Scaling/BlockTimingNewMesh2.txt}; \label{plot:M2}
	\addplot [no markers, color=black, thick, dashed ] coordinates {
		(224,  177.2801)
		(3584, 11.08)
	};\label{plot:ideal}
	\addplot [no markers, color=black, thick, dashed ] coordinates {
		(448,  171.75146)
		(7168, 10.73446625)
	};
	\end{groupplot}
	\path (top|-current bounding box.south)--
        coordinate(legendpos)
        (bot|-current bounding box.south);
  \matrix[
      matrix of nodes,
      anchor=north,
      draw,
      inner sep=0.2em,
    ]at([yshift=-1ex]legendpos)
    { \ref{plot:M1}& \textsc{M1} &[5pt]
      \ref{plot:M2}& \textsc{M2} &[5pt]
      \ref{plot:ideal}& ideal &[5pt]\\
      };

	\end{tikzpicture}
	\caption{\textit{Strong scaling:} Shown in the panels are the strong scaling behaviors on TACC~\Frontera \,
	for the (a) Cahn-Hilliard (CH),
	 (b) velocity prediction (VP), (c) pressure Poisson (PP), and (d) velocity updates (VU). 
	%Three different mesh are considered: \textsc{M1} with 280 K elements, \textsc{M2} with 1 M elements and \textsc{M3} with 19 M elements. 
	%Excellent scaling behavior is observed up to $\mathcal{O}(17K)$ processors.
	}
	\label{fig:strong_scaling_breakup}
\end{figure}

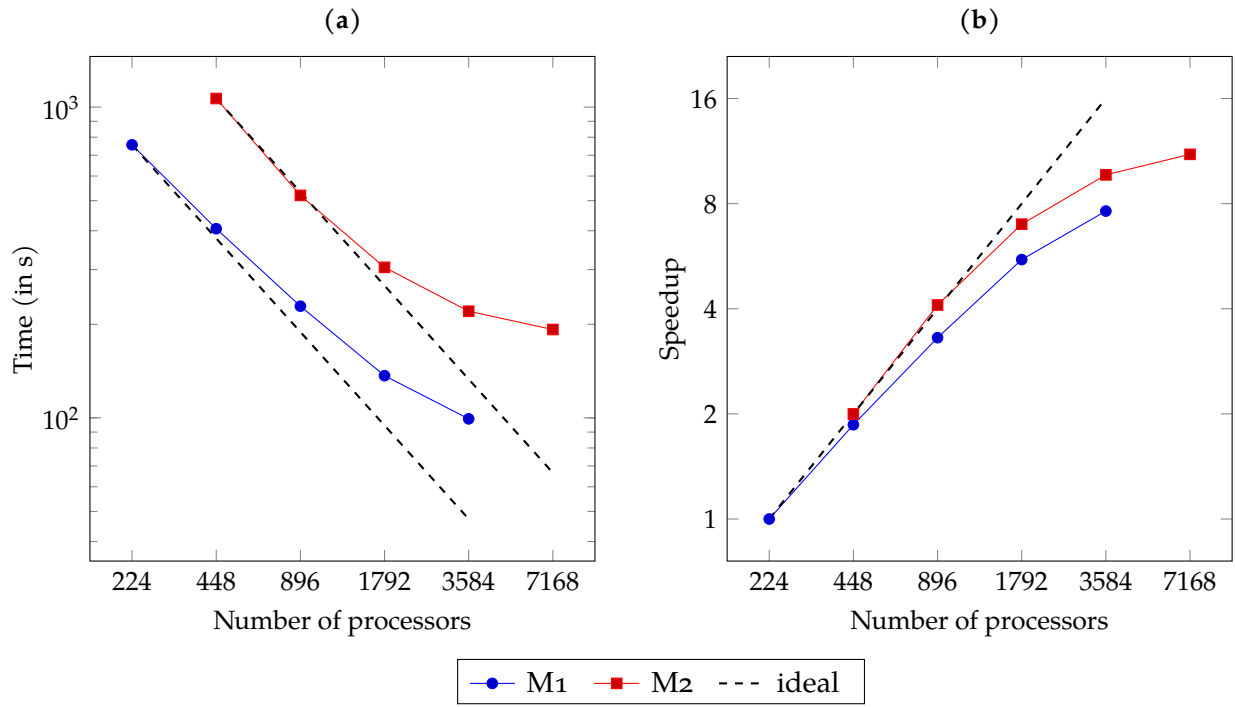
\begin{figure}[H]
	\centering
	\begin{tikzpicture}
	\begin{groupplot}[
	group style={group size=2 by 1,ylabels at=edge left,vertical sep=50pt, horizontal sep=50pt},
	ylabel style={text height=0.02\textwidth,inner ysep=-2pt},
	xlabel= Number of processors, % Set the labels
	height=0.5*\linewidth,width=0.5\linewidth,/tikz/font=\small,
	]
	\nextgroupplot[xmode=log,log basis x={2},
	ymode=log,
	xtick={224, 448, 896, 1792, 3584,7168},
	xticklabels={$224$,$448$,$896$, $1792$, $3584$, $7168$},
	title={\textbf{(a)}},
	ylabel= Time (in s),
	]%
	\addplot table[x expr={\thisrow{proc}))},y expr={\thisrow{Block1-ch} + \thisrow{Block2-ch}  + \thisrow{Block1-ns} + \thisrow{Block2-ns} + \thisrow{Block1-pp} + \thisrow{Block2-pp} + \thisrow{Block1-vu} + \thisrow{Block2-vu}
	+ \thisrow{Remesh} + \thisrow{EquationUpdate}},col sep=space]{Figures/ScalingData/Scaling/BlockTimingNewMesh1.txt};
		\addplot table[x expr={\thisrow{proc}))},y expr={\thisrow{Block1-ch} + \thisrow{Block2-ch}  + \thisrow{Block1-ns} + \thisrow{Block2-ns} + \thisrow{Block1-pp} + \thisrow{Block2-pp} + \thisrow{Block1-vu} + \thisrow{Block2-vu}
	+ \thisrow{Remesh} + \thisrow{EquationUpdate}},col sep=space]{Figures/ScalingData/Scaling/BlockTimingNewMesh2.txt};
	\addplot [no markers, thick, dashed ] coordinates {
		(224,  756.194104)
		(3584, 47.2621315)
	};
	\coordinate (top) at (rel axis cs:0,1);% coordinate at top of the first plot
	\addplot [no markers, thick, dashed ] coordinates {
		(448,  1065.71688415755)
		(7168, 66.56)
	};

	\nextgroupplot[xmode=log,log basis x={2},
	ymode=log, log basis y={2},
	xtick={224, 448, 896, 1792, 3584,7168},
	xticklabels={$224$,$448$,$896$, $1792$, $3584$, $7168$},
	ytick={1,2,4,8,16},
	yticklabels={$1$,$2$,$4$,$8$,$16$},
	title={\textbf{(b)}},
	ylabel= Speedup,
	]%
	\addplot
	table[x expr={\thisrow{proc}))},
	y expr={756.194104/(\thisrow{Block1-ch} + \thisrow{Block2-ch}  + \thisrow{Block1-ns} + \thisrow{Block2-ns} + \thisrow{Block1-pp} + \thisrow{Block2-pp} + \thisrow{Block1-vu} + \thisrow{Block2-vu} + \thisrow{Remesh} + \thisrow{EquationUpdate})},col sep=space]{Figures/ScalingData/Scaling/BlockTimingNewMesh1.txt};\label{plot:solveM1}
		\addplot
	table[x expr={\thisrow{proc}))},
	y expr={2*1065.71688415755/(\thisrow{Block1-ch} + \thisrow{Block2-ch}  + \thisrow{Block1-ns} + \thisrow{Block2-ns} + \thisrow{Block1-pp} + \thisrow{Block2-pp} + \thisrow{Block1-vu} + \thisrow{Block2-vu} + \thisrow{Remesh} + \thisrow{EquationUpdate})},col sep=space]{Figures/ScalingData/Scaling/BlockTimingNewMesh2.txt};\label{plot:solveM2}
	\addplot [no markers, color=black, thick, dashed, fill ] coordinates {
		(224, 1)
		(3584, 16)
	};\label{plot:solveideal}
		\coordinate (bot) at (rel axis cs:1,0);% coordinate at
% 	\addplot [no markers, color=black, thick, dashed, fill ] coordinates {
% 		(448, 2)
% 		(3584, 16)
% 	};
	\end{groupplot}
	%
	%
	%\path (top|-current bounding box.south)--coordinate(legendpos)
	%(bot|-current bounding box.south);
	%\matrix[
	%matrix of nodes,
	%anchor=north,
	%draw,
	%inner sep=0.2em,
	%]at([yshift=-1ex]legendpos)
	%{ \ref{plot:level3}& \textsc{M1} &[5pt]
	%	\ref{plot:level4}& \textsc{M2} &[5pt]
	%	\ref{plot:level5}& \textsc{M3} &[5pt]
	%	\ref{plot:Relideal}& ideal &[5pt]\\};  
		\path (top|-current bounding box.south)--
        coordinate(legendpos)
        (bot|-current bounding box.south);
  \matrix[
      matrix of nodes,
      anchor=north,
      draw,
      inner sep=0.2em,
    ]at([yshift=-1ex]legendpos)
    { \ref{plot:solveM1}& \textsc{M1} &[5pt]
      \ref{plot:solveM2}& \textsc{M2} &[5pt]
      \ref{plot:solveideal}& ideal &[5pt]\\
      };
	\end{tikzpicture}
	\caption{\textit{Strong scaling:} shown in the panels are (a) the total time to solution as a function of 
	the number of processors, and (b) the relative speed up as a function of 
	the number of processors. These results were computed on the TACC~\Frontera \, supercomputer. 
		%Three different mesh are considered: \textsc{M1} with 280 K elements, \textsc{M2} with 1 M elements and \textsc{M3} with 19 M elements. 
		%Excellent scaling behavior is observed up to $\mathcal{O}(17K)$ processors.
	}
	\label{fig:strong_scaling_total}
\end{figure}
% !TEX root = ../../main.tex

\begin{figure}[tbh]
  \begin{tikzpicture}
   \pgfplotsset{ybar stacked,xtick=data,xtick style={draw=none}}
    \begin{groupplot}[group style={group size= 1 by 2,vertical sep=50pt},height=0.35*\linewidth,width=\linewidth,/tikz/font=\small]
     \nextgroupplot[title = {Mesh \textsc{M1}},
     xmode=log,  log basis x={2},	grid=major,
 	 	xtick={224,448,896,1792,3584},
 	 	scaled y ticks = true,
  		scaled ticks=base 10:-3,
  			xticklabels={$224$,$448$,$896$,$1792$,$3584$},
					%  xtick style={draw=none},
  					ylabel = {Parallel cost (cores.s) $\rightarrow$},/tikz/font=\small]
  	% linear
  	\addplot [fill=div_d1] [bar shift=0cm] table[x={proc} , y expr={(\thisrow{Block1-ch}+\thisrow{Block2-ch})*\thisrow{nodes}/4.0}]{Figures/ScalingData/Scaling/BlockTimingNewMesh1.txt};
  	\addplot [fill=div_d2] [bar shift=0cm]  table[x={proc}, y expr={(\thisrow{Block1-ns} + \thisrow{Block2-ns})*\thisrow{nodes}/4}]{Figures/ScalingData/Scaling/BlockTimingNewMesh1.txt};
  	\addplot [fill=div_d3] [bar shift=0cm]  table[x={proc}, y expr={(\thisrow{Block1-pp}  +\thisrow{Block2-pp})*\thisrow{nodes}/4}]{Figures/ScalingData/Scaling/BlockTimingNewMesh1.txt};
  	\addplot [fill=div_d4] [bar shift=0cm]  table[x={proc}, y expr={(\thisrow{Block1-vu} + \thisrow{Block2-vu})*\thisrow{nodes}/4.0}]{Figures/ScalingData/Scaling/BlockTimingNewMesh1.txt};
  	\addplot [fill=div_d5] [bar shift=0cm] table[x={proc}, y expr={(\thisrow{Remesh}*\thisrow{nodes})/4.0}]{Figures/ScalingData/Scaling/BlockTimingNewMesh1.txt};
  	\coordinate (top) at (rel axis cs:0,1);% coordinate at top of the first plot
	% \addplot [fill=div_c6] [bar shift=0cm] table[x={proc}, y expr={(\thisrow{EquationUpdate}*\thisrow{nodes}/4.0)}]{Data/Scaling/BlockTiming.txt};
%   \legend{\small CH-solve, VP-Solve, PP-Solve, VU-solve,remesh}
  \nextgroupplot[title = {Mesh \textsc{M2}},
   xmode=log,  log basis x={2},	grid=major,
     xtick={448,896,1792,3584,7168},
     scaled y ticks = true,
  		scaled ticks=base 10:-3,
  				xticklabels={$448$,$896$,$1792$,$3584$,$7168$},
					xlabel = {Number of processors $\rightarrow$},
  					ylabel = {Parallel cost (cores.s) $\rightarrow$},/tikz/font=\small]
  	% linear
  	\addplot [fill=div_d1] [bar shift=0cm] table[x={proc} , y expr={(\thisrow{Block1-ch}+\thisrow{Block2-ch})*\thisrow{nodes}/8.0}]{Figures/ScalingData/Scaling/BlockTimingNewMesh2.txt};\label{scalePlot:ch}
  	\addplot [fill=div_d2] [bar shift=0cm]  table[x={proc}, y expr={(\thisrow{Block1-ns} + \thisrow{Block2-ns})*\thisrow{nodes}/8}]{Figures/ScalingData/Scaling/BlockTimingNewMesh2.txt};\label{scalePlot:ns}
  	\addplot [fill=div_d3] [bar shift=0cm]  table[x={proc}, y expr={(\thisrow{Block1-pp}  +\thisrow{Block2-pp})*\thisrow{nodes}/8}]{Figures/ScalingData/Scaling/BlockTimingNewMesh2.txt};\label{scalePlot:pp}
  	\addplot [fill=div_d4] [bar shift=0cm]  table[x={proc}, y expr={(\thisrow{Block1-vu} + \thisrow{Block2-vu})*\thisrow{nodes}/8.0}]{Figures/ScalingData/Scaling/BlockTimingNewMesh2.txt};\label{scalePlot:vu}
  	\addplot [fill=div_d5] [bar shift=0cm] table[x={proc}, y expr={(\thisrow{Remesh}*\thisrow{nodes})/8.0}]{Figures/ScalingData/Scaling/BlockTimingNewMesh2.txt};\label{scalePlot:remesh}
  	 \coordinate (bot) at (rel axis cs:1,0);% coordinate at bottom of the last plot
	% \addplot [fill=div_c6] [bar shift=0cm] table[x={proc}, y expr={(\thisrow{EquationUpdate}*\thisrow{nodes}/4.0)}]{Data/Scaling/BlockTiming.txt};
	    \end{groupplot}
	    	\path (top|-current bounding box.north)--
      		coordinate(legendpos)
      		(bot|-current bounding box.north);
	\matrix[
    	matrix of nodes,
    	anchor=south,
    	draw,
    	inner sep=0.2em,
    	draw
  	]at([yshift=1ex]legendpos)
	{
% 	CH-solve, VP-Solve, PP-Solve
    	\ref{scalePlot:ch}& CH-solve &[5pt]
    	\ref{scalePlot:ns}& VP-solve &[5pt]
    	\ref{scalePlot:pp}& PP-solve &[5pt]
    	\ref{scalePlot:vu}& VU-solve  &[5pt]
    	\ref{scalePlot:remesh}& Remesh &[5pt]\\
	};
  \end{tikzpicture}
  \caption{\textit{Strong scaling parallel cost:} (Run-time) $\times$ (number of cores) evaluated on the 3D Rayleigh taylor problem on the \Frontera~ supercomputer. A flat line would indicate ideal strong scaling. Here in the legend: CH-Solve corresponds to Cahn-Hilliard solve; VP-Solve corresponds to velocity prediction solve; PP-Solve corresponds to pressure Poisson Solve; VU-solve corresponds to velocity update solve.}
  \label{fig:strong-RT3D}

\end{figure}
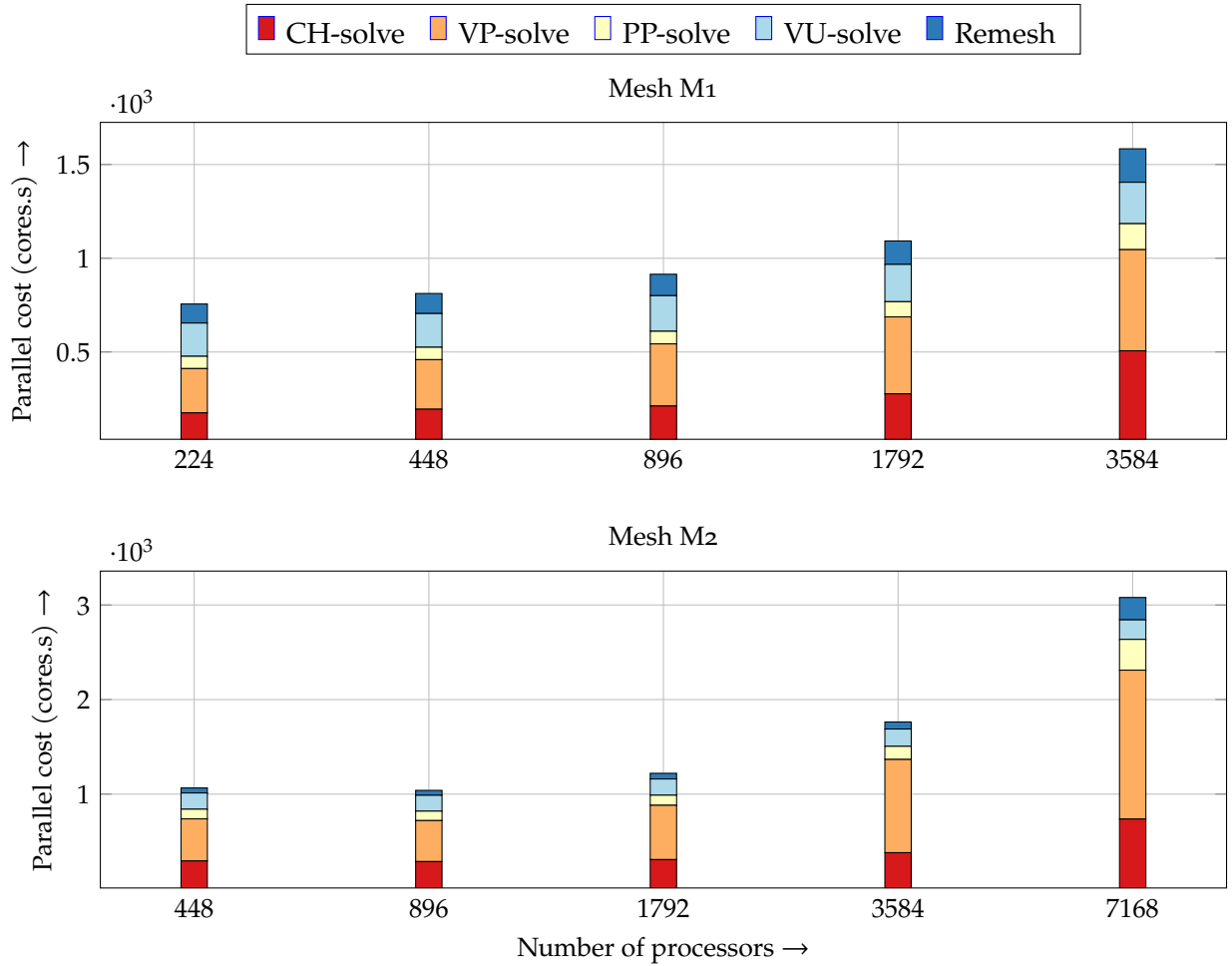

\Cref{fig:componentAnalysis} shows the time taken by individual components for mesh ~\textsc{M1} used in the scaling study. It can be clearly seen that each solver has different bottlenecks that would require different optimization strategies to overcome. For instance, the preconditioner (PC) setup and matrix assembly is dominant for the Cahn-Hilliard and velocity prediction solves, whereas vector assembly begins to equally dominate for the pressure Poisson (PP) solver. The matrix assembly becomes a bottleneck during the velocity update stage. With increasing number of processors, we can note that the percentage of time taken by AMG setup begins to dominate. This is a characteristic of AMG which has motivated the developers to look into GMG ~\cite{sundar2012parallel}, especially at exascale. This is left as an avenue of future work. A more sophisticated roofline analysis will help to suggest more advanced optimization strategies like vectorization or cache blocking for each individual steps. However, we can clearly see a major advantage of using the current projection based approach. Unlike the previous analysis carried out in ~\citet{Khanwale2021}, where PC setup was the most dominant part and can take up to 80\% of the total compute time, we see that we no longer rely on expensive ASM/LU~\footnote{Additive Schwarz preconditioners with LU were used locally for each block.} preconditioners to ensure the convergence of the solver. %This is a significant improvement from the computational point of view where the same result can be achieved by requiring less floating point operations (FLOPs).
 However, further optimization would require us to abandon the vendor optimized library, especially for improve
 performance during the matrix assembly and remeshing stages. %However, we note that we can no longer rely on the vendor optimized library for ensuring the optimized computations, and this asks for a better optimization strategy especially during the assembly and remeshing stage.

% % % !TEX root = ../../main.tex
\tikzexternaldisable
\begin{figure}
\begin{tikzpicture}
    \pgfplotsset{ybar stacked,xtick=data,xtick style={draw=none}}
    \begin{groupplot}[group style={group size= 2 by 2,vertical sep=60pt}]
    \nextgroupplot[title=CH solve,bar shift=0pt,xmode=log,  log basis x={2},xmin=2,xmax=100,ymin =0, ymax=100,
  					xtick={4,8,16,32,64},
  					xticklabels={$224$,$448$,$896$,$1792$,$3584$},
					%  xtick style={draw=none},
  					ylabel = {Percentage of time},/tikz/font=\small]
  		\node[rotate=-90] at (axis cs: 4.0,10.0) {\tiny{$\mathbf{52\%}$}};
  		\node[rotate=-90] at (axis cs: 8.0,10.0) {\tiny{$\mathbf{46\%}$}};
  		\node[rotate=-90] at (axis cs: 16.0,10.0) {\tiny{$\mathbf{38\%}$}};
  		\node[rotate=-90] at (axis cs: 32.0,10.0) {\tiny{$\mathbf{35\%}$}};
  		\node[rotate=-90] at (axis cs: 64.0,10.0) {\tiny{$\mathbf{29\%}$}};
  		
  		\node[rotate=-90] at (axis cs: 4.0,58.0) {\tiny{$\mathbf{11\%}$}};
  		\node[rotate=-90] at (axis cs: 8.0,51.0) {\tiny{$\mathbf{10\%}$}};
  		\node[rotate=-90] at (axis cs: 16.0,42.0) {\tiny{$\mathbf{9\%}$}};
  		\node[rotate=-90] at (axis cs: 32.0,39.0) {\tiny{$\mathbf{9\%}$}};
  		\node[rotate=-90] at (axis cs: 64.0,33.0) {\tiny{$\mathbf{9\%}$}};
  		
  		\node[rotate=-90] at (axis cs: 4.0,80.0) {\tiny{$\mathbf{31\%}$}};
  		\node[rotate=-90] at (axis cs: 8.0,78.0) {\tiny{$\mathbf{39\%}$}};
  		\node[rotate=-90] at (axis cs: 16.0,60.0) {\tiny{$\mathbf{36\%}$}};
  		\node[rotate=-90] at (axis cs: 32.0,60.0) {\tiny{$\mathbf{41\%}$}};
  		\node[rotate=-90] at (axis cs: 64.0,60.0) {\tiny{$\mathbf{49\%}$}};
    	\addplot[fill=div_d5,opacity=1.0]  table [x=Nodes,y expr={100*\thisrow{MatAssembly}/\thisrow{Total}}] {Figures/ScalingData/Scaling/chComponent.txt};\label{gplots:plot1}
    	\addplot[fill=div_d3,opacity=1.0]  table [x=Nodes, y expr={100*\thisrow{VecAssembly}/\thisrow{Total}}] {Figures/ScalingData/Scaling/chComponent.txt};\label{gplots:plot2}
    	\addplot[fill=div_d2,opacity=1.0]  table [x=Nodes, y expr={100*\thisrow{PC}/\thisrow{Total}}] {Figures/ScalingData/Scaling/chComponent.txt};\label{gplots:plot3}
    	\addplot[fill=div_d1,opacity=1.0]  table [x=Nodes, y expr={100*(\thisrow{Total} - \thisrow{PC} - \thisrow{MatAssembly} - \thisrow{VecAssembly})/\thisrow{Total}}] {Figures/ScalingData/Scaling/chComponent.txt};\label{gplots:plot4}
    	
    \coordinate (top) at (rel axis cs:0,1);% coordinate at top of the first plot
	%%%%%%%%%%%%%%%%%%%%%%%%%%%%%%%%%%%%%%%%%%%%%%%%%%%%%%%%%%%%%%%%%%%%
    \nextgroupplot[title=Velocity prediction solve,bar shift=0pt,xmode=log,  log basis x={2},xmin=2,xmax=100,ymin =0, ymax=100,
   					xtick={4,8,16,32,64},
   					xticklabels={$224$,$448$,$896$,$1792$,$3584$},
 					%  xtick style={draw=none},
 					/tikz/font=\small]
 					\node[rotate=-90] at (axis cs: 4.0,10.0) {\tiny{$\mathbf{70\%}$}};
  		\node[rotate=-90] at (axis cs: 8.0,10.0) {\tiny{$\mathbf{63\%}$}};
  		\node[rotate=-90] at (axis cs: 16.0,10.0) {\tiny{$\mathbf{54\%}$}};
  		\node[rotate=-90] at (axis cs: 32.0,10.0) {\tiny{$\mathbf{45\%}$}};
  		\node[rotate=-90] at (axis cs: 64.0,10.0) {\tiny{$\mathbf{37\%}$}};
  		
  		\node[rotate=-90] at (axis cs: 4.0,82.0) {\tiny{$\mathbf{21\%}$}};
  		\node[rotate=-90] at (axis cs: 8.0,80.0) {\tiny{$\mathbf{28\%}$}};
  		\node[rotate=-90] at (axis cs: 16.0,78.0) {\tiny{$\mathbf{37\%}$}};
  		\node[rotate=-90] at (axis cs: 32.0,78.0) {\tiny{$\mathbf{48\%}$}};
  		\node[rotate=-90] at (axis cs: 64.0,78.0) {\tiny{$\mathbf{56\%}$}};
     	\addplot[fill=div_d5,opacity=1.0]  table [x=Nodes,y expr={100*\thisrow{MatAssembly}/\thisrow{Total}}] {Figures/ScalingData/Scaling/nsComponent.txt};
     	\addplot[fill=div_d3,opacity=1.0]  table [x=Nodes, y expr={100*\thisrow{VecAssembly}/\thisrow{Total}}] {Figures/ScalingData/Scaling/nsComponent.txt};
     	\addplot[fill=div_d2,opacity=1.0]  table [x=Nodes, y expr={100*\thisrow{PC}/\thisrow{Total}}] {Figures/ScalingData/Scaling/nsComponent.txt};
     	\addplot[fill=div_d1,opacity=1.0]  table [x=Nodes, y expr={100*(\thisrow{Total} - \thisrow{PC} - \thisrow{MatAssembly} - \thisrow{VecAssembly})/\thisrow{Total}}] {Figures/ScalingData/Scaling/nsComponent.txt};
     	
	%%%%%%%%%%%%%%%%%%%%%%%%%%%%%%%%%%%%%%%%%%%%%%%%%%%%%%%%%%%%%%%%%%%%%%   
	\nextgroupplot[title=PP solve,bar shift=0pt,xmode=log,  log basis x={2},xmin=2,xmax=100,ymin =0, ymax=100,
  	 				xtick={4,8,16,32,64},
   					xticklabels={$224$,$448$,$896$,$1792$,$3584$},
     				xlabel={Number of processor},
 					%  xtick style={draw=none},
 					ylabel = {Percentage of time},/tikz/font=\small]
 							\node[rotate=-90] at (axis cs: 4.0,10.0) {\tiny{$\mathbf{41\%}$}};
  		\node[rotate=-90] at (axis cs: 8.0,10.0) {\tiny{$\mathbf{39\%}$}};
  		\node[rotate=-90] at (axis cs: 16.0,10.0) {\tiny{$\mathbf{37\%}$}};
  		\node[rotate=-90] at (axis cs: 32.0,10.0) {\tiny{$\mathbf{30\%}$}};
  		\node[rotate=-90] at (axis cs: 64.0,10.0) {\tiny{$\mathbf{25\%}$}};
  		
  		\node[rotate=-90] at (axis cs: 4.0,50.0) {\tiny{$\mathbf{30\%}$}};
  		\node[rotate=-90] at (axis cs: 8.0,50.0) {\tiny{$\mathbf{29\%}$}};
  		\node[rotate=-90] at (axis cs: 16.0,50.0) {\tiny{$\mathbf{29\%}$}};
  		\node[rotate=-90] at (axis cs: 32.0,45.0) {\tiny{$\mathbf{25\%}$}};
  		\node[rotate=-90] at (axis cs: 64.0,40.0) {\tiny{$\mathbf{23\%}$}};
  		
  		\node[rotate=-90] at (axis cs: 4.0,81.0) {\tiny{$\mathbf{16\%}$}};
  		\node[rotate=-90] at (axis cs: 8.0,81.0) {\tiny{$\mathbf{20\%}$}};
  		\node[rotate=-90] at (axis cs: 16.0,81.0) {\tiny{$\mathbf{24\%}$}};
  		\node[rotate=-90] at (axis cs: 32.0,81.0) {\tiny{$\mathbf{33\%}$}};
  		\node[rotate=-90] at (axis cs: 64.0,81.0) {\tiny{$\mathbf{38\%}$}};
  		
    	\addplot[fill=div_d5,opacity=1.0]  table [x=Nodes,y expr={100*\thisrow{MatAssembly}/\thisrow{Total}}] {Figures/ScalingData/Scaling/ppComponent.txt};
     	\addplot[fill=div_d3,opacity=1.0]  table [x=Nodes, y expr={100*\thisrow{VecAssembly}/\thisrow{Total}}] {Figures/ScalingData/Scaling/ppComponent.txt};
     	\addplot[fill=div_d2,opacity=1.0]  table [x=Nodes, y expr={100*\thisrow{PC}/\thisrow{Total}}] {Figures/ScalingData/Scaling/ppComponent.txt};
     	\addplot[fill=div_d1,opacity=1.0]  table [x=Nodes, y expr={100*(\thisrow{Total} - \thisrow{PC} - \thisrow{MatAssembly} - \thisrow{VecAssembly})/\thisrow{Total}}] {Figures/ScalingData/Scaling/ppComponent.txt};
    %%%%%%%%%%%%%%%%%%%%%%%%%%%%%%%%%%%%%%%%%%%%%%%%%%%%%%%%%%%%%%%%%%%%%%
    \nextgroupplot[title=Velocity update solve,bar shift=0pt,xmode=log,  log basis x={2},xmin=2,xmax=100,ymin =0, ymax=100,
      				xlabel={Number of processor},
   					xtick={4,8,16,32,64},
   					xticklabels={$224$,$448$,$896$,$1792$,$3584$},
 					%  xtick style={draw=none},
 					/tikz/font=\small]
 					 		\node[rotate=-90] at (axis cs: 4.0,10.0) {\tiny{$\mathbf{86\%}$}};
  		\node[rotate=-90] at (axis cs: 8.0,10.0) {\tiny{$\mathbf{85\%}$}};
  		\node[rotate=-90] at (axis cs: 16.0,10.0) {\tiny{$\mathbf{84\%}$}};
  		\node[rotate=-90] at (axis cs: 32.0,10.0) {\tiny{$\mathbf{79\%}$}};
  		\node[rotate=-90] at (axis cs: 64.0,10.0) {\tiny{$\mathbf{73\%}$}};
  		
 		\node[rotate=-90] at (axis cs: 4.0,90.0) {\tiny{$\mathbf{7\%}$}};
  		\node[rotate=-90] at (axis cs: 8.0,89.0) {\tiny{$\mathbf{8\%}$}};
  		\node[rotate=-90] at (axis cs: 16.0,89.0) {\tiny{$\mathbf{8\%}$}};
  		\node[rotate=-90] at (axis cs: 32.0,84.0) {\tiny{$\mathbf{8\%}$}};
  		\node[rotate=-90] at (axis cs: 64.0,84.0) {\tiny{$\mathbf{10\%}$}};
     	\addplot[fill=div_d5,opacity=1.0]  table [x=Nodes,y expr={100*\thisrow{MatAssembly}/\thisrow{Total}}] {Figures/ScalingData/Scaling/vuComponent.txt};
     	\addplot[fill=div_d3,opacity=1.0]  table [x=Nodes, y expr={100*\thisrow{VecAssembly}/\thisrow{Total}}] {Figures/ScalingData/Scaling/vuComponent.txt};
     	\addplot[fill=div_d2,opacity=1.0]  table [x=Nodes, y expr={100*\thisrow{PC}/\thisrow{Total}}] {Figures/ScalingData/Scaling/vuComponent.txt};
     	\addplot[fill=div_d1,opacity=1.0]  table [x=Nodes, y expr={100*(\thisrow{Total} - \thisrow{PC} - \thisrow{MatAssembly} - \thisrow{VecAssembly})/\thisrow{Total}}] {Figures/ScalingData/Scaling/vuComponent.txt};
    %%%%%%%%%%%%%%%%%%%%%%%%%%%%%%%%%%%%%%%%%%%%%%%%%%%%%%%%%%%%%%%%%%%%%%
    \coordinate (bot) at (rel axis cs:1,0);% coordinate at bottom of the last plot
    \end{groupplot}
%    \path (top-|current bounding box.west)-- 
%          node[anchor=south,rotate=90] {throughput} 
%          (bot-|current bounding box.west);
	% legend
	\path (top|-current bounding box.north)--
      		coordinate(legendpos)
      		(bot|-current bounding box.north);
	\matrix[
    	matrix of nodes,
    	anchor=south,
    	draw,
    	inner sep=0.2em,
    	draw
  	]at([yshift=1ex]legendpos)
	{
    	\ref{gplots:plot1}& Matrix Assembly &[5pt]
    	\ref{gplots:plot2}& Vector Assembly &[5pt]
    	\ref{gplots:plot3}& PC Setup        &[5pt]
    	\ref{gplots:plot4}& Others          &[5pt]\\
	};
\end{tikzpicture}
\caption{The panels show the percentage of time taken for matrix assembly, vector assembly, preconditioner (PC) setup, and other parts of the update as a function of the number of processors. The panels are organized as follows: (a) Cahn-Hilliard,
(b) velocity prediction, (c) pressure Poisson, and (d) velocity update.}
\label{fig:componentAnalysis}
\end{figure}
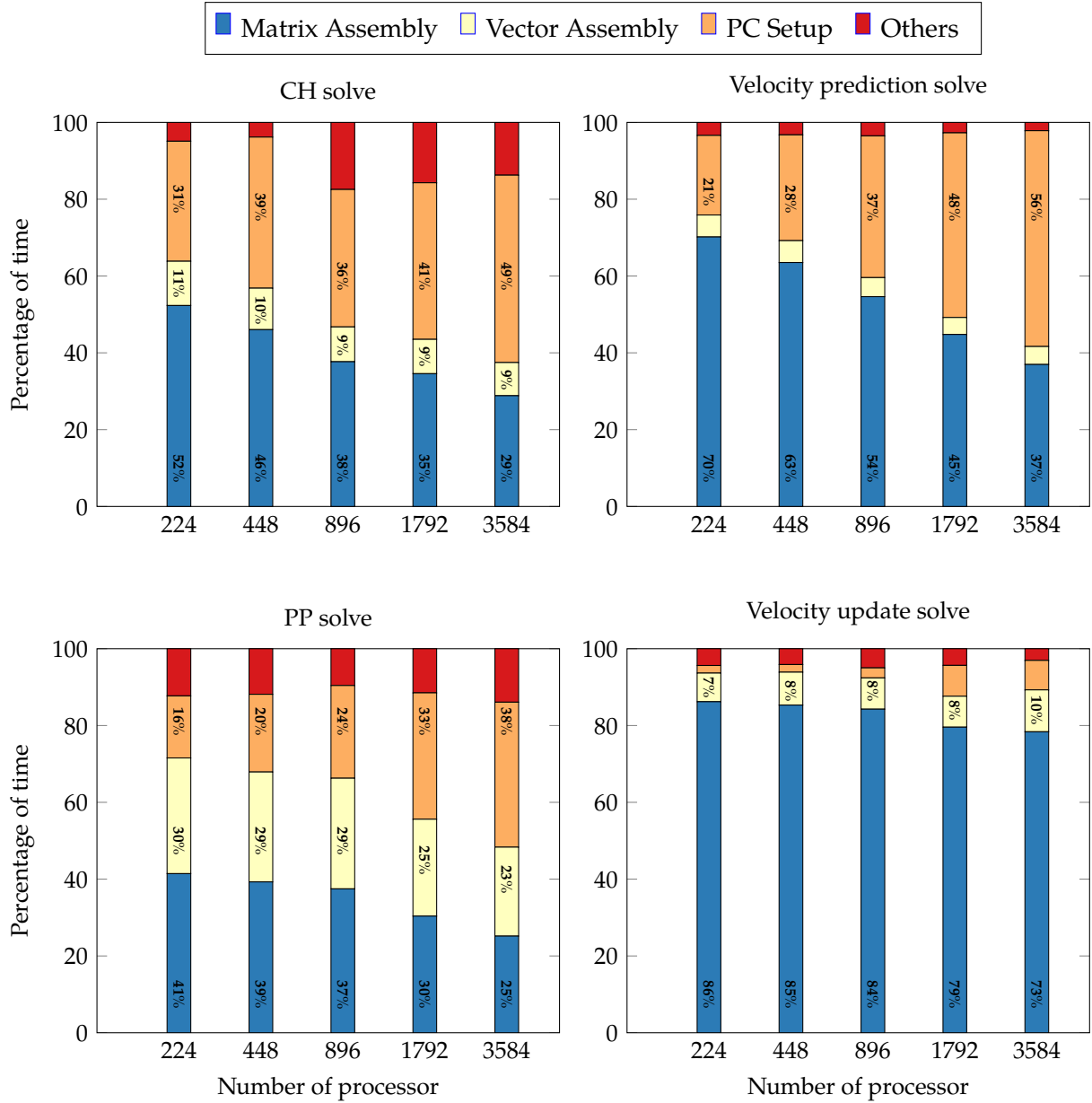

\subsection{Performance comparison with the fully coupled method:}
The projection method presented in this paper in conjunction with the variational multiscale approach, allows us to construct operators which are elliptic in each block. This allows for a selection of cheaper preconditioners compared to the pressure stabilized approach in~\citet{Khanwale2021}.  Additionally, the block projection approach also allows us to selectively provide implicit treatment and decreases cost, e.g. Cahn-Hilliard operators are fully implicit, while the Navier-Stokes operators are now semi-implicit, which linearizes the operator.  To assess the performance gains of the projection based method in this paper compared to the fully-implicit method in~\citet{Khanwale2021}, we perform a strong scaling analysis for both methods.  For better comparison, the fully-coupled method from~\citet{Khanwale2021} is implemented in the current parallel framework using the same mesh (M2 from \cref{tab:refineLevel}).  We use the same case of 3D Rayleigh Taylor instability used for performing the rest of scaling analysis.  The timestep size used in both the methods is also kept the same, so there is no artificial speedup due to small timestep for the projection method.  \Cref{fig:strong_scaling_comp} shows the strong scaling  behavior for both the fully coupled method from~\citet{Khanwale2021} and the projection method presented in the current work.  We can observed that both methods scale well, however the projection method is consistently cheaper than the fully-coupled method.  We observe a speedup of roughly 2.5 to 4 times with the block projection method. 
\begin{figure}
	\centering
	\begin{tikzpicture}[thick,scale=0.7, every node/.style={transform shape}]
		\begin{loglogaxis}[width=300pt,scaled y ticks=true,title={Total Solve Time},xlabel={Number of processors},ylabel={Time (s)},
			legend style={nodes={scale=0.65, transform shape}, text = black},
			xtick={272,544,1088,2176,4352,8704},
			xticklabels={$272$,$544$,$1K$,$2K$,$4K$,$8K$},
			ytick={200,400,600,1000,1400,2000},
			yticklabels={$200$,$400$,$600$,$1000$,$1400$,$2000$},
			]
			\addplot table [x={ntasks},y={solve_time},col sep=comma] {Figures/scalingComparison/level10_block_projection.csv};
			\addplot table [x={ntasks},y={solve_time},col sep=comma] {Figures/scalingComparison/level10_fully_coupled.csv};
			     \legend{Block projection(this work), Fully coupled (\citet{Khanwale2021}) }
		\end{loglogaxis}
	\end{tikzpicture} 
	\caption{Comparison of total solve time presented in this work in comparison to the fully coupled method of ~\citet{Khanwale2021}.}
	%\caption{Strong scaling}
	\label{fig:strong_scaling_comp}		
\end{figure}
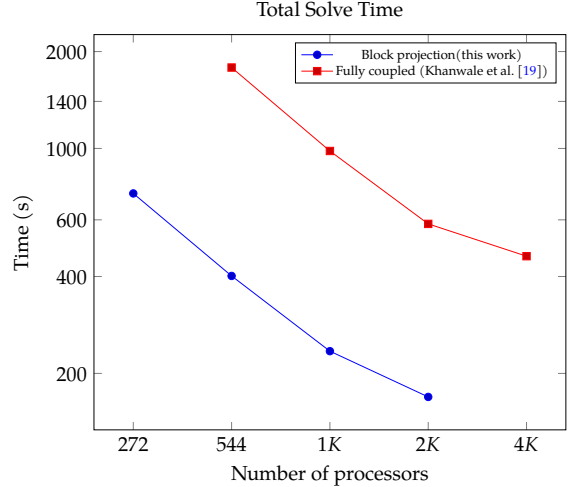
%!TEX root = main.tex

\section{Conclusions and future work}

In this work we developed a projection based numerical framework for solving the Cahn-Hilliard Navier-Stokes (CHNS) model of two-phase flows.  We developed a projection based second order in time numerical scheme to improve on the fully implicit scheme in~\citet{Khanwale2021}.   The pressure projection method allowed us to decouple the pressure equation.  The decoupling of pressure allowed us to use highly efficient and parallel Algebraic Multigrid Methods as all operators are close to elliptic. We developed a block iterative, hybrid semi-implicit-fully-implicit in time method. The new method requires Newton iteration only for Cahn-Hilliard resulting faster solution time at relatively larger time steps.  We utilized a conforming Galerkin method for spatial discretization with variational multiscale approach (VMS) approach.  VMS method allows us to stabilize the pressure projection equation as pressure and velocity is represented by equal order basis functions.  We deployed this approach on the massively parallel octree based adaptive meshing framework based on KD-trees called~\dendroKT.  We verifed numerically that the scheme is energy stable and mass conserving.  We tested the framework for a comprehensive set of canonical cases for validation.  We run the long time turbulent 2D Rayleigh-Taylor cases with adaptive meshing, resolving thin stable filaments with a very high resolution.  
%We also simulated 3D Rayleigh-Taylor instability to capture high deformation of the interface in 3D. 
Further, we performed a detailed scaling analysis of numerical framework and showed that the time required for preconditioning and Jacobian assembly is minimized compared to the method in~\citet{Khanwale2021}.

We identify three main avenues of future work.  First one entails coming up with efficient mass conserving interpolation strategies with conforming Galerkin methods. This is currently an open question.  The second avenue is rigorously proving energy stability of the current numerical scheme. Rigorous energy stability of such second order projection schemes are tedious to prove.  Finally, we notice that the algebraic multigrid methods become very expensive after certain number of processors are reached due to excessive communication. We are working on a tailor made multi-grid method for octree meshes which optimizes communication for extreme scalability of the iterative solvers.  

\section{Acknowledgements}
The authors acknowledge XSEDE grant number TG-CTS110007 for computing time on TACC Stampede2 and SDSC Expanse.  The authors also acknowledge computing allocation through Pathways award number PHY20033 on TACC Frontera, and Iowa State University computing resources through NSF 2018594. JAR was supported in part by NSF Grants DMS--1620128 and DMS--2012699. BG, KS, MAK were funded in part by NSF grant 1935255, 1855902. HS and MI were funded in part by NSF grant 1912930. 

\bibliographystyle{elsarticle-num-names}
\bibliography{total_references,hari}

\newpage

%!TEX root = main.tex
\appendix

\section{Details of solver selection for the numerical experiments}
\label{sec:app_linear_solve}
Solver settings presented below are broken up in 4 structures for 4 blocks we solve i.e, velocity prediction (see~\cref{defn:weak_VMS_disc}), pressure Poisson (see~\cref{defn:weak_VMS_disc_pp_vel}), velocity update (see~\cref{defn:weak_VMS_disc_vel_update}), and Cahn-Hilliard (CH) nonlinear solve (see~\cref{defn:weak_VMS_disc_CH}) respectively.
\subsection{Manufactured solutions}
\label{subsec:app_manufac_sol}
For the cases presented in \cref{subsec:manfactured_soln_result} we use the flexible GMRES linear solver algebraic multigrid based preconditioning for pressure Poisson (\texttt{solver\_options\_pp}), CH (\texttt{solver\_options\_ch}), and a biCGstab based Krylov solver with Additive Schwarz based preconditioning for velocity prediction (\texttt{solver\_options\_momentum}).  For better reproduction, the command line options we provide {\sc petsc} are given below which include some commands used for printing some norms as well.  Here \texttt{ksp\_atol} here is the absolute tolerance of the linear solver, \texttt{ksp\_rtol} is the relative tolerance of the linear solver, and \texttt{snes\_rtol}, \texttt{snes\_atol} are relative and absolute tolerances of Newton iteration.  \petsc~AMG's inner smoothers are chosen to be a gmres (\texttt{gmres}) method with successive over relaxation preconditioning (\texttt{sor}).   
\begin{lstlisting}
solver_options_momentum = {
	ksp_atol = 1e-10
	ksp_rtol = 1e-10
	ksp_type = "bcgs"
	pc_type = "asm"
# monitor
	ksp_converged_reason = ""
};

solver_options_pp = {
	ksp_atol = 1e-10
	ksp_rtol = 1e-10
#multigrid
	ksp_type = "fgmres"
	pc_type = "gamg"
	pc_gamg_asm_use_agg = True
	mg_levels_ksp_type = "gmres"
	mg_levels_pc_type = "sor"
	mg_levels_ksp_max_it = 4
# monitor
	ksp_monitor = ""
	ksp_converged_reason = ""
};

solver_options_vupdate = {
	ksp_atol = 1e-10
	ksp_rtol = 1e-10
#ksp_monitor = ""
	ksp_converged_reason=""
	ksp_type = "cg"
	pc_type = "asm"
};

solver_options_ch = {
	snes_rtol = 1e-15
	snes_atol = 1e-15
	ksp_atol = 1e-12
	ksp_rtol = 1e-12
	ksp_stol = 1e-12
#multigrid
	ksp_type = "fgmres"
	pc_type = "gamg"
	pc_gamg_asm_use_agg = True
	mg_levels_ksp_type = "gmres"
	mg_levels_pc_type = "sor"
	mg_levels_ksp_max_it = 4
# monitor
	snes_monitor = ""
	snes_converged_reason = ""
	ksp_converged_reason = ""
};
\end{lstlisting}
\subsection{single bubble rise in 2D}
\label{subsec:app_bubble_rise}
In this case we use the flexible GMRES linear solver algebraic multigrid based preconditioning for velocity prediction (\texttt{solver\_options\_momentum}), pressure Poisson (\texttt{solver\_options\_pp}), and CH (\texttt{solver\_options\_ch}). A simple conjugate gradient method with no Additive Schwarz preconditioning is used for velocity update (\texttt{solver\_options\_vu}). 
\begin{lstlisting}
solver_options_momentum = {
	ksp_atol = 1e-8
	ksp_rtol = 1e-7

#multigrid
	ksp_type = "fgmres"
	pc_type = "gamg"
	pc_gamg_asm_use_agg = True
	mg_levels_ksp_type = "gmres"
	mg_levels_pc_type = "sor"
	mg_levels_ksp_max_it = 10
# monitor
	ksp_monitor = ""
	ksp_converged_reason = ""
};

solver_options_pp = {
	ksp_atol = 1e-8
	ksp_rtol = 1e-7
	ksp_stol = 1e-7

#multigrid
	ksp_type = "fgmres"
	pc_type = "gamg"
	pc_gamg_asm_use_agg = True
	mg_levels_ksp_type = "gmres"
	mg_levels_pc_type = "sor"
	mg_levels_ksp_max_it = 4
	pc_gamg_sym_graph = True

# monitor
	ksp_monitor = ""
	ksp_converged_reason = ""
};

solver_options_vupdate = {
	ksp_atol = 1e-8
	ksp_rtol = 1e-7
	ksp_stol = 1e-7
	ksp_converged_reason=""
	ksp_type = "cg"
	pc_type = "asm"
};

solver_options_ch = {
	snes_rtol = 1e-10
	snes_atol = 1e-10
	ksp_atol = 1e-8
	ksp_rtol = 1e-7
	ksp_stol = 1e-7

#multigrid
	ksp_type = "fgmres"
	pc_type = "gamg"
	pc_gamg_asm_use_agg = True
	mg_levels_ksp_type = "gmres"
	mg_levels_pc_type = "sor"
	mg_levels_ksp_max_it = 15

# monitor
	snes_monitor = ""
	snes_converged_reason = ""
	ksp_converged_reason = ""
};
\end{lstlisting}
\subsection{Rayleigh-Taylor instablity}
\label{subsec:app_rt2d}
In this case we use the flexible GMRES linear solver algebraic multigrid based preconditioning for velocity prediction (\texttt{solver\_options\_momentum}), pressure Poisson (\texttt{solver\_options\_pp}), and CH (\texttt{solver\_options\_ch}). A simple conjugate gradient method with no Additive Schwarz preconditioning is used for velocity update (\texttt{solver\_options\_vu}). 
\begin{lstlisting}
solver_options_momentum = {
	ksp_atol = 1e-8
	ksp_rtol = 1e-7

#multigrid
	ksp_type = "fgmres"
	pc_type = "gamg"
	pc_type = "none"
	pc_gamg_asm_use_agg = True
	mg_levels_ksp_type = "gmres"
	mg_levels_pc_type = "sor"
	mg_levels_ksp_max_it = 10

# monitor
	ksp_monitor = ""
	ksp_converged_reason = ""
};

solver_options_pp = {
	ksp_atol = 1e-8
	ksp_rtol = 1e-7
	ksp_stol = 1e-7

#multigrid
	ksp_type = "fgmres"
	pc_type = "gamg"
	pc_gamg_asm_use_agg = True
	mg_levels_ksp_type = "gmres"
	mg_levels_pc_type = "sor"
	mg_levels_ksp_max_it = 4

# monitor
	ksp_monitor = ""
	ksp_converged_reason = ""
};

solver_options_vupdate = {
	ksp_atol = 1e-8
	ksp_rtol = 1e-7
	ksp_converged_reason=""
	ksp_type = "cg"
	pc_type = "asm"
};

solver_options_ch = {
	snes_rtol = 1e-10
	snes_atol = 1e-10
	ksp_atol = 1e-8
	ksp_rtol = 1e-7

#multigrid
	ksp_type = "fgmres"
	pc_type = "gamg"
	pc_gamg_asm_use_agg = True
	mg_levels_ksp_type = "gmres"
	mg_levels_pc_type = "sor"
	mg_levels_ksp_max_it = 15

# monitor
	snes_monitor = ""
	snes_converged_reason = ""
	ksp_converged_reason = ""
};
\end{lstlisting}
This setup works very well with a relatively constant number of Krylov iterations as the number of processes are increased in the massively parallel setting. These same options are used for the simulation of Rayleigh-Taylor in 3D.  The scaling results we present use the same setup of solvers except for the no of max ksp iterations on each multigrid level is increased to 30 (\texttt{mg\_levels\_ksp\_max\_it}). 

%\newpage
%\listoftodos[Notes]

\end{document}